\documentclass[11pt,reqno]{amsart}
\usepackage{geometry}                
\usepackage{amsmath,amsthm,amssymb,amsfonts,relsize}
\usepackage[english]{babel}
\usepackage{breqn}
\usepackage{tikz}
\usetikzlibrary{calc,hobby}
\usepackage{bbm}
\usepackage[shortlabels]{enumitem}
\usepackage{graphicx}
\usepackage{theoremref}
\newtheorem{theorem}{Theorem}[section]
\numberwithin{equation}{section}
\newtheorem{remark}[theorem]{Remark}
\newtheorem{prop}[theorem]{Proposition}
\newtheorem{corollary}[theorem]{Corollary}
\newtheorem{lemma}[theorem]{Lemma}

\usepackage{amssymb}
\usepackage{blindtext}
\usepackage{xcolor}
\usepackage{mathrsfs}
\usepackage{ulem}
\usepackage{stmaryrd}
\usepackage{hyperref}
\hypersetup{
    colorlinks=true,
    linkcolor=blue,
    filecolor=magenta,      
    urlcolor=cyan,
}
\newcommand\sometext

\newcommand{\mucr}{\mu_{\mathrm{cr}}}
\newcommand{\Fcr}{F_{\mathrm{cr}}}
\newcommand{\loc}{\mathrm{loc}}

\newcommand{\ts}{\textsuperscript}
\newcommand{\jump}[1]{\left\llbracket{#1}\right\rrbracket}
\newcommand{\norm}[1]{\left\lVert#1\right\rVert}
\DeclareMathOperator{\sech}{sech}

\title[Solitary waves in two-layer stratified water] {Large-amplitude solitary waves in two-layer density stratified water}
\author[D. Sinambela]{Daniel Sinambela}
\address{Department of Mathematics, University of Missouri, Columbia, MO 65211} 
\email{dsf25@mail.missouri.edu}

\begin{document}
\maketitle
\begin{abstract}
    We present a large-amplitude existence theory for two-dimensional solitary  waves propagating through a two layer body of water. The domain of the fluid is bounded below by an impermeable flat ocean floor and above by a free boundary at constant pressure. For any piecewise smooth upstream density distribution and laminar background current, we construct a global curve of solutions. This curve bifurcates from the background current and, following along the curve, we find waves that are arbitrarily close to having horizontal stagnation points.
    
    The small-amplitude waves are constructed using a center manifold reduction technique. The large-amplitude theory is obtained through analytical global bifurcation together with refined qualitative properties of the waves.
   
\end{abstract}

\setcounter{tocdepth}{1}
\tableofcontents
\section{Introduction}

Internal waves are a common sight in the vicinity of complex coastline structures and narrow passages such as straits or fjords \cite{helfrich2006review}. They are a product of the density heterogeneity brought on by variations in temperature and salinity.  Together, these effects result in the water being \textit{stratified} into superposed layers, and it is the interfaces between these layers along which internal waves move. While comparatively slow, they can be large in amplitude and carry enormous amounts of energy. In the Lombok Strait, for instance, internal waves have been observed with amplitude exceeding 100 meters and average speed of approximately 1.96 m/s \cite{susanto2005ocean}.  They play a pivotal role in general ocean dynamics by transporting and mixing biogenic and non-biogenic components in the water bulk.

The present work concerns the existence of two-dimensional internal solitary waves in a stratified body of water. Solitary here means that the waves take the form of a spatially localized disturbance moving over a background current. These have been the subject of extensive research, beginning in the 19\ts{th} century with the famous observations of Russel \cite{Russell1884}. Exact existence results for steady water waves took nearly a century more to prove, but this theory has progressed considerably in recent decades due to advancements in nonlinear functional analysis, harmonic analysis, and PDE theory. In particular, local and global bifurcation techniques have been used to construct small- and large-amplitude traveling water waves in a number of physical settings. Most of these works study the irrotational and homogeneous density case; see, for example, surveys in \cite{constantin2011nonlinear,Groves2016lectures}. A major challenge when considering internal waves is that stratification generically creates vorticity, and thus it is necessary to work in the rotational regime.

The first rigorous existence theory in the heterogeneous setting was obtained by Dubreil-Jacotin \cite{Jacotin1934Surla} who constructed small-amplitude periodic waves. Ter-Krikorov \cite{ter1960existence} later  showed the existence of infinitesimal solitary water waves as the limit of periodic waves taking period to infinity. For large-amplitude stratified solitary water waves, the first result can be found in the work of Amick \cite{amick1984semilinear} and Amick--Turner \cite{amick1986global}. They considered heterogeneous fluid bounded above and below by  infinitely long rigid walls with uniform background current. 

Our primary contribution in this paper is to establish the existence of large-amplitude solitary waves allowing an \textit{arbitrary} piecewise smooth background current and density distribution. This is done via a global bifurcation theoretic argument that furnishes a locally analytic curve of solutions.  We further prove that, following this family to its extreme, one finds waves that coming arbitrarily close to horizontal stagnation.  This is consistent with the limiting behavior of homogeneous density irrotational solitary waves \cite{amick1981solitary}, which are known to terminate at a wave of greatest height with a stagnation point at its crest.  A similar result was obtained by Chen, Walsh, and Wheeler \cite{chen2018existence} for continuously stratified solitary waves.  The two-layered stratified case is considerably more complicated, however, and a definitive proof of the stagnation limit requires substantial new analysis.  

\subsection{Governing equation}

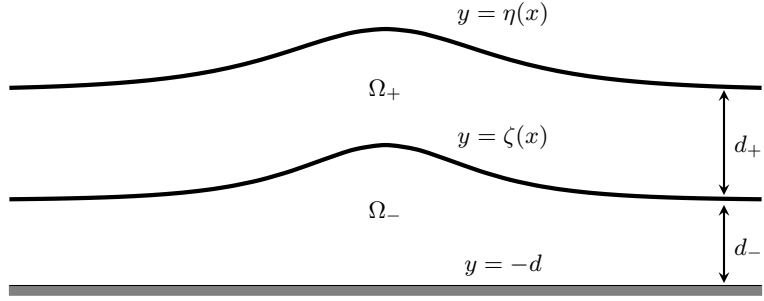
\begin{figure}
\centering
\begin{tikzpicture}[xscale=1.25,yscale=0.65]
\draw [fill=gray,thin,gray] (-4,-1.7) rectangle (4,-1.5);				
\draw [thin,-] (-4, -1.5) -- (4, -1.5);

\draw [ultra thick,domain=-4:4] plot[smooth] (\x, {0.25+2.25/(exp(1.25*\x)+exp(-1.25*\x))});		
\draw [domain=-4:4,ultra thick] plot[smooth] (\x, {2.5+2.5/(exp(\x)+exp(-\x))});		

\node [below  ] at (1.25,2) {\footnotesize $y = \zeta(x)$};
\node [below  ] at (1.25,4.5) {\footnotesize $y =\eta(x)$};
\node [below  ] at (1.25,-0.65) {\footnotesize $y = -d$};

\node [right] at (3.6,1.35) {\footnotesize $d_+$};
\draw [thick,stealth-stealth]  (3.6,0.35) -- (3.6, 2.5) ; 

\node [right] at (3.6,-0.75) {\footnotesize $d_-$};
\draw [thick,stealth-stealth]  (3.6,0.15) -- (3.6, -1.45) ; 

\node  at (0,2.5) {\footnotesize $\Omega_+$};
\node  at (0,0) {\footnotesize $\Omega_-$};

\end{tikzpicture}
    \caption{Configuration of the fluid domain}
    \label{fluid domain figure}
\end{figure}

Let us now formulate the problem mathematically. We consider waves propagating through a two-dimensional body of water. They are traveling in the sense that they evolve by translating to the right with a fixed wave speed $c>0$. By adopting a moving reference frame, all time dependence in the system can therefore be eliminated.  Suppose the water is organized into two continuously density stratified layers that are separated by a free boundary. The lower layer is bounded from below by an impermeable bed at $\{y=-d\}$ for a fixed $d>0.$ The upper layer lies below a free boundary above which is vacuum at constant pressure. We, therefore, write the fluid domain as $\Omega= \Omega_+ \cup \Omega_-$, where $\Omega_+$ is the upper layer
\[\Omega_+=\{(x,y)\in \mathbb{R}^2: \zeta(x)<y<\eta(x)\},\] and
$\Omega_-$ is the lower layer
\[\Omega_-=\{(x,y)\in \mathbb{R}^2: -d<y<\zeta(x)\}.\] 
This assumes that the internal and upper interfaces are the graphs of the a priori unknown functions $\zeta$ and $\eta$, respectively. For solitary waves, we must have that $\zeta$ and $\eta$ limit to some far-field heights as $x \to \pm\infty$.  We let $d_\pm$ denotes the asymptotic thickness of the layer $\Omega_\pm$, so that $d = d_+ + d_-$; see Figure~\ref{fluid domain figure}.

Denote by $(u,v) : \overline{\Omega_+} \cup \overline{\Omega_-} \to \mathbb{R}^2$ the fluid velocity, $P : \overline{\Omega} \to \mathbb{R}$ the pressure, and let $\varrho : \overline{\Omega_+} \cup \overline{\Omega_-} \to \mathbb{R}$ be the density. For physical reasons, we require that $\varrho$ be strictly positive and that $y \mapsto \varrho(\cdot,y)$ is non-increasing.  It is important to note here that the velocity and density will in general not be continuous over the internal interface.

In the moving frame, traveling water waves are governed by the incompressible steady Euler system:
\begin{equation}\label{IncomEuler}
 \left\{ \begin{aligned}
    u_x+v_y&=0 \\
    \varrho(u-c)u_x+\varrho v u_y&=-P_x \\\varrho(u-c)v_x+\varrho v v_y&=-P_y-g\varrho\\
\end{aligned} \right. \qquad \textup{in} \; \Omega,
\end{equation} 
where $g>0$ is the gravitational constant. We also assume mass conservation along the flow which is formulated in the following form
\begin{equation}\label{massconservation}
    (u-c)\varrho_x +v \varrho_y=0 \quad \textup{in} \; \Omega.
\end{equation} On the boundaries, we impose the standard kinematic and dynamic conditions,
\begin{equation}\label{boundarycond}
 \begin{cases}
    v=0 &\quad \textup{on}\;y=-d,\\
   v=(u-c)\eta_x &\quad \textup{on}\;y=\eta(x),\\
    v=(u-c)\zeta_x &\quad \textup{on}\;y=\zeta(x),
    \\P=P_{\textup{atm}}&\quad \textup{on}\;y=\eta(x),\\\jump{P}=0 &\quad \textup{on}\;y=\zeta(x).\\
\end{cases}  
\end{equation}
Here $P_{\textrm{atm}}$ is the (constant) atmospheric pressure.  Note also that the third equation in \eqref{boundarycond} implies that  ${v}/(u-c)$ is continuous over the internal interface.

We also require that there is no horizontal stagnation:
\begin{equation}\label{nohorzstag}
    u-c<0 \quad \textup{in}\; \overline{\Omega}.
\end{equation}This assumption will be crucial for our reformulation of the problem later. Recall that the \textit{streamlines} are the integral curves of the relative velocity field $(u-c,v)$.  The first three equations in \eqref{boundarycond} ensure that the bed, internal interface, and upper boundary are streamlines.  As a consequence of \eqref{nohorzstag}, every streamlines extend from $-\infty$ to $\infty$ and is given by the graph of a single-valued function of $x$. 

To study solitary waves, we must also specify the background current.  This takes the form of the asymptotic conditions
\begin{equation}\label{asymptotic}
    (u,v)\rightarrow{(\mathring{u},0)},\quad \varrho \rightarrow{\mathring{\varrho}},\quad \eta \rightarrow{0}, \quad \zeta \rightarrow{-d_+} \qquad \textup{as} \;|x| \rightarrow{\infty},
\end{equation}
where the convergence is uniform in $y$.  Here, $\mathring{u}:=\mathring{u}(y)$ is the far field horizontal velocity profile and $\mathring{\varrho}=\mathring{\varrho}(y)$ is the far field density profile. It is convenient to replace $\mathring{u}$ by the scaled asymptotic relative horizontal velocity $u^*$ given by
\begin{equation} \label{ asymptotic of horizontal velocity}
    \mathring{u} =c-Fu^*.
\end{equation} 
The parameter $F>0$ is referred to as the \textit{Froude number} and can be thought of as the dimensionless wave speed. From the literature of traveling waves, it is expected that there exists a critical Froude number $\Fcr$ that separates the regimes where periodic waves and solitary waves exist.  In particular, (nontrivial) solitary waves must be \textit{supercritical} in that $F > \Fcr$.  This fact has recently been proved for the case of homogeneous density by Kozlov, Lokharu, and Wheeler \cite{kozlov2020nonexistence}.  A rigorous definition of $\Fcr$ for the present system is given in Section~\ref{Linearized operators}.

Lastly, we recall some terminology describing the qualitative properties of water waves. A \textit{laminar flow} is a wave whose streamlines are all parallel to the bed; these form the class of trivial solutions of the problem. A \textit{solitary wave of elevation} is a wave where each streamline lies above its limiting height upstream and downstream. In particular, given our coordinates system, this implies $\eta$ is strictly positive. A traveling wave is called \textit{symmetric} provided that $u$ and $\eta$ are even in $x$ while $v$ is odd. Finally, a symmetric waves is \textit{monotone} if the slope of the streamlines, $v/(u-c)$, is negative to the left of the crest at $x=0$ and above the bed.

\subsection{Statement of results}
Our first contribution is a systematic existence theory for large-amplitude stratified solitary waves with arbitrary piecewise smooth density distribution and horizontal velocity profile at infinity. 
\begin{theorem}[Large-amplitude solitary waves]\label{thm:main}
Fix H\"older exponent $\alpha \in (0,1)$, wave speed $c>0$, far-field depths $d_+, d_- > 0$, gravitational constant $g>0$.  For any (strictly positive) asymptotic relative velocity and density profiles
\begin{equation} \label{regularity ustar}
u^*, \, \mathring{\varrho} \in C^{8+\alpha}([-d,-d_+],\mathbb{R}_+) \cap C^{8+\alpha}([-d_+,0],\mathbb{R}_+),
\end{equation}
there exists a continuous global curve
\begin{equation}\label{globalcurve}
    \mathscr{C}=\{(u(s),v(s),\eta(s),\zeta(s), F(s)):s\in (0,\infty)\}
\end{equation}
of solitary wave solutions to \eqref{IncomEuler}--\eqref{asymptotic}, exhibiting the regularity
\begin{equation}\label{regularityofuv}
    u(s), \, v(s) \in C^{8+\alpha}(\overline{\Omega_+(s)}) \cap C^{8+\alpha}(\overline{\Omega_-(s)}), \qquad \eta(s), \, \zeta(s) \in C^{9+\alpha}(\mathbb{R})
\end{equation} 
where $\Omega(s):=\Omega_+(s) \cup \Omega_-(s)$ is the corresponding fluid domain. The global solution curve $\mathscr{C}$ enjoys the following properties:
\begin{enumerate}[label=\normalfont{(\alph*)}]
    \item\label{stagnation}\textup{(Stagnation limit)} Following $\mathscr{C}$, we encounter waves that are arbitrarily close to having horizontal stagnation:
    \begin{equation*}
        \lim_{s\rightarrow{\infty}} \inf_{\Omega(s)}|c-u(s)|=0.
    \end{equation*}
    
    \item\label{critical laminar} \textup{(Critical laminar flow)} The curve begins at the critical laminar flow:
    \begin{equation*}
        \lim_{s\rightarrow{0}}(u(s),v(s),\eta(s),\zeta(s),F(s))=(c-F_{\mathrm{cr}}u^*,0,0,-d_+,F_{\textup{cr}}).
    \end{equation*}
    
    \item\label{symmetry and monotonicity} \textup{(Symmetry and monotonicity)} All solutions on $\mathscr{C}$ are symmetric waves of elevation, monotone, and supercritical.
\end{enumerate}
\end{theorem}

\begin{remark} Let us make a few remarks.
\begin{enumerate}[label=\normalfont{(\roman*)}]
\item This result assumes a single discontinuity in the far-field density profile. In fact, the theory easily extends to finitely many discontinuities with the only cost being more cumbersome notation.  The resulting waves would then be organized into many layers.

\item The $C^{8+\alpha}$ regularity asked for here is almost certainly much more than necessary. We impose it in order to satisfy the hypothesis of the center manifold reduction method from \cite{chen2019center}, which in turn only needs it due to the technical lemma \cite[Lemma 2.1]{Amick1994Center}.  We conjecture that the regularity of $u^*$ and $\mathring{\varrho}$ in each layer can be relaxed to $C^{2+\alpha}$, which will then give solutions with 
\begin{equation}\label{conjectured regularityofuv}
    u, \, v \in C^{2+\alpha}(\overline{\Omega_+}) \cap C^{2+\alpha}(\overline{\Omega_-}), \qquad \eta, \, \zeta \in C^{3+\alpha}(\mathbb{R}).
\end{equation}
A proof of this fact would require a lengthy digression into the details of those two papers, and so we do not pursue it here.  Following the approach in \cite{akers2019solitary}, moreover, one expects that it should be possible to take $u^*$ to be merely Lipschitz continuous in each layer.
\end{enumerate}
\end{remark}

We also establish a number of qualitative properties of stratified solitary waves.  These are of independent interest but also crucially important to the proof of Theorem~\ref{thm:main}.  We list here the two most significant, but others can be found in Section~\ref{Qualitative}.

The first result states that supercritical solitary waves of elevation are necessarily symmetric and monotone.  This is achieved through a moving planes argument in the spirit of Li \cite{Li1991monotonicity} and Maia \cite{Maia1997symmetryofinternalwaves}.  

\begin{theorem}[Symmetry]\label{symmetry} Let $(u, v,\eta,\zeta ,F)$ be a supercritical wave of elevation that solves \eqref{IncomEuler}--\eqref{nohorzstag} and enjoys the regularity \eqref{conjectured regularityofuv} with 
\[\norm{ u }_{C^2(\Omega_+) \cap C^2(\Omega_-)}, \norm{ v }_{C^2(\Omega_+) \cap C^2(\Omega_-)}, \norm{ \eta }_{C^3(\mathbb{R})},\norm{ \zeta }_{C^3(\mathbb{R})}<\infty.\]
Suppose that \[(u,v)\rightarrow{(\mathring{u},0)}\] uniformly as $x \rightarrow{+\infty}$ or as $(x\rightarrow{-\infty}),$ then after an appropriate translation, the wave is a monotone and symmetric solitary wave. 
\end{theorem}

The second theorem gives a uniform upper bound on the velocity for stratified solitary waves in terms of a lower bound on the Froude number and a bound away from horizontal stagnation.  

\begin{theorem}[Velocity bound] \label{thm:velocity}
Let $(u, v,\eta,\zeta ,F)$ be a solution to  \eqref{IncomEuler}--\eqref{ asymptotic of horizontal velocity} that enjoys the regularity \eqref{conjectured regularityofuv} and satisfies 
\[ F \geq F_0 > 0, \qquad \sup_{\Omega} (u-c) \leq -\delta < 0.  \]
Then,
\[ \sup_{\Omega} \left( (u-c)^2 + v^2\right) < C,\]
for a constant $C = C(F_0, \delta,u^*,\mathring{\varrho}) > 0$.
\end{theorem}

Note that through Bernoulli's law, the above theorem can also be used to control the pressure. A result of this type is proved by Chen, Walsh, Wheeler \cite{chen2018existence} for one-layer stratified waves based on  Varvaruca's \cite{Varvaruca2009extremewaves} treatment of the constant density case.  That method, however, is not sufficient for the present setting, as the maximum principle argument it relies on struggles with the discontinuity of the velocity across the layers.  Our approach combines pressure bounds in the bulk with the ``almost monotonicity formula'' of Caffarelli--Kenig--Jerison \cite{Caffarelli2002monotonicity} to control the velocity near the internal interface.  

\subsection{Plan of the paper}\label{planofthepaper}
Let us now outline the general structure of the paper while explaining the main mathematical difficulties and how we will approach them.
 
We begin in Section~\ref{Formulation section} by non-dimensionalizing the governing equations. Applying the Dubreil-Jacotin transformation sends the fluid domain $\Omega$ to a slitted rectangular strip. In these variables the incompressible steady Euler system becomes a quasilinear elliptic PDE coupled with nonlinear transmission boundary conditions. Written as an abstract operator equation, it takes the form \[\mathcal{F}(w,F)=0,\] where $w$ is a new unknown measuring the deviation of the streamlines relative to the background current. To lay the ground work for the small-amplitude theory, in Section~\ref{Linearized operators}, we investigate the linearized operator at $w=0$. Restricting its domain to laminar flows, we arrive at a Sturm--Liouville type problem with a transmission condition. It is shown that there exists a critical value of the Froude number, $F=F_{\textup{cr}}$, for which $0$ is the principal eigenvalue.

As further preparation for the existence theory, in Section~\ref{Qualitative} we prove the qualitative results mentioned above. We also present a result on asymptotic monotonicity and nodal pattern of the solutions. The main tools used here are the maximum principle as well as integral identities.

Section~\ref{Small Amplitude} is where the small-amplitude existence theory is established. These solutions lie on a local curve, denoted by $\mathscr{C}_{\textup{loc}}$, that bifurcates from $(w,F)=(0,F_{\textup{cr}})$. For periodic waves, small solutions are usually found via the classical Lyapunov--Schmidt reduction. However, this method can not be applied directly here because $\mathcal{F}_w(0,\Fcr)$ is \textit{not Fredholm} as a consequence of the unboundedness of the domain and the definition of $\Fcr$.  This analytical challenge is intrinsic to the study of (small-amplitude) solitary waves.  For constant density rotational waves, Hur \cite{hur2008exact} constructed solutions using a Nash--Moser technique that generalized Beale's \cite{beale1977existence} treatment of the  irrotational case.  Considering the same problem, Groves and Wahl\'en in \cite{Wahlen2008small-amplitude} used a Hamiltonian spatial dynamics approach. This argument was adapted by Chen, Walsh, and Wheeler \cite{chen2018existence} to the one-layer continuously stratified regime, and by Wang \cite{wang2017solitary} for two-phase flows with constant density in each layer.  In the present paper, however, we employ a center manifold reduction ``without a phase space'' based the recent paper \cite{chen2019center}.
 
 In Section~\ref{largeamplitude}, we continue $\mathscr{C}_{\textup{loc}}$  globally to obtain the curve $\mathscr{C}$ that extends into the large-amplitude regime. Again, the unboundedness of the domain presents a significant obstruction to standard bifurcation theoretic techniques. For example, it is not a priori clear that $\mathcal{F}^{-1}(0)$ is locally pre-compact or that $\mathcal{F}$ is locally proper. This is not just a technical concern. Indeed, it is well-known that in other stratified regimes, solitary waves may broaden into an infinitely long ``table top''; see, for example, \cite{turner1988broadening}.
 Because these waves remain bounded in any H\"older space but do not converge to a localized solution, this scenario implies a lack of compactness for the zero set of $\mathcal{F}$.
 
 The classical strategy for constructing large-amplitude solitary waves is to view them as the limit of periodic waves as the period tends to infinity.  This is done, for example, by Amick and Toland  \cite{amicktoland1981periodic} in their study of the constant density irrotational wave case.  They first construct global families of periodic waves, then take the period to infinity using a uniform estimates and an application of the Whyburn lemma. This results in a global connected set of solutions.
 
 Our approach is based on the analytic global bifurcation theory introduced by Chen, Walsh, and Wheeler \cite{chen2018existence} which is a variant of the classical work of Dancer \cite{dancer1973bifurcation,dancer1973global} and Buffoni--Toland \cite{buffoni2003analytic}. Essentially, we treat the loss of compactness as an alternative and show that it must manifest as the broadening phenomena mentioned above. Using the qualitative theory, we can then rule out this possibility leaving only the stagnation limit.  
 
 Lastly, for the convenience of the reader, Appendix~\ref{quoted results appendix} contains some results from the literature that are drawn upon throughout the paper.
 
\section{Formulation}\label{Formulation section}

In this section, we introduce several reformulations of the problem that will make it more amenable to analysis.  We also record a number of notational conventions used throughout the paper.

\subsection{Non-dimensionalization}
Let us denote the density along the free surface as follows:
\begin{equation}\label{denistyonthesurface}
\varrho_0:=\mathring{\varrho}(0).
\end{equation} Next, we normalize the $u^*$ to satisfy 

\begin{equation}\label{normalizedu*}
    \int_{-d}^{0} \sqrt{\mathring{\varrho}(y)}u^*(y)\;\,dy =\sqrt{g\varrho_0d^3}.
\end{equation} In addition, we consider the \textit{(relative) pseudo-volumetric mass} $m>0$:
\begin{equation}\label{volumetric}
    m:=\int_{-d}^{\eta(x)}\sqrt{\varrho(x,y)}(c-u(x,y))\,dy.
\end{equation} One can check that $m$ is independent of $x$. Letting $|x| \rightarrow{\infty},$ we obtain
\[
    m=\int_{-d}^{0}\sqrt{\mathring{\varrho}(y)}(c-\mathring{u}(y))\,dy=F\int_{-d}^{0}\sqrt{\mathring{\varrho}(y)}u^*(y)\,dy.
\]Using the above equation and  \eqref{normalizedu*}, we see that
\begin{equation}\label{F,m}
    \dfrac{g \varrho_0 d^3}{m^2}=\dfrac{1}{F^2}.
\end{equation}
\begin{subequations}

We non-dimensionalize the coordinates using the asymptotic depth $d$ as the characteristic length scale, which gives us
\begin{equation*}
   (\tilde{x},\tilde{y}):=\dfrac{1}{d}(x,y), \quad \tilde{\eta}(\tilde{x}):=\dfrac{1}{d}\eta(x), \quad \tilde{\zeta}(\tilde{x}):=\dfrac{1}{d}\zeta(x).
\end{equation*} Likewise, the density is rescaled using $\varrho_0$ in \eqref{denistyonthesurface}
\begin{equation*}
    \tilde{\varrho}(\tilde{x},\tilde{y})=\dfrac{1}{\varrho_0}\varrho(x,y), \quad \tilde{\mathring{\varrho}}(\tilde{y}):=\dfrac{1}{\varrho_0}\mathring{\varrho}(y),
\end{equation*} 
and the velocity is non-dimensionalized via the Froude number
\begin{equation*}
\begin{split}
    \tilde{u}(\tilde{x},\tilde{y}):&=\dfrac{\sqrt{\varrho_0}d}{m}u(x,y), \quad \tilde{v}(\tilde{x},\tilde{y}):=\dfrac{\sqrt{\varrho_0}d}{m}v(x,y), \\& \tilde{c}:=\dfrac{\sqrt{\varrho_0}d}{m}c, \quad \tilde{\mathring{u}}(\tilde{y}):=\dfrac{\sqrt{\varrho_0}d}{m}\mathring{u}(y).
\end{split}
\end{equation*}
Finally, the pressure is rescaled by taking
\begin{equation}\label{rescaledeulerincom}
    \tilde{P}(\tilde{x},\tilde{y}):=\dfrac{d^2}{m^2}\left(P(x,y)-P_{\textup{atm}}\right).
\end{equation}
\end{subequations}

Rewriting \eqref{IncomEuler} and \eqref{massconservation}, we finally obtain the non-dimensionalized system
\begin{equation*}
  \left\{\begin{aligned}
    \tilde{u}_{\tilde{x}}+\tilde{v}_{\tilde{y}}&=0\\
    \tilde{\varrho} (\tilde{u}-\tilde{c})\tilde{u}_{\tilde{x}}+\tilde{\varrho} \tilde{v} \tilde{u}_{\tilde{y}}&=-\tilde{P}_{\tilde{x}} \\\tilde{\varrho} (\tilde{u}-\tilde{c})\tilde{u}_{\tilde{x}}+\tilde{\varrho} \tilde{v} \tilde{u}_{\tilde{y}}&=-\tilde{P}_{\tilde{y}}-\dfrac{1}{F^2}\tilde{\varrho}\\
    (\tilde{u}-\tilde{c})\tilde{\varrho}_{\tilde{x}} +\tilde{v} \varrho_y&=0
\end{aligned}\right. \qquad \textup{in} \; \tilde{\Omega},
\end{equation*} 
where $\tilde{\Omega}$ is the rescaled domain:
\begin{equation*}
    \tilde{\Omega}:=\{(\tilde{x},\tilde{y}) \in \mathbb{R}^2:-1<\tilde{y}<\tilde{\zeta}(\tilde{x})\cup \tilde{\zeta}(\tilde{x})<y<\tilde{\eta}(\tilde{x})\}.
\end{equation*} 

The boundary conditions after rescaling read
\begin{equation}\label{Dimensionless Kinematic and Dynamic BC}
\left \{\begin{aligned}
    \tilde{v}&=0 &\quad \textup{on}\;\tilde{y}&=-1,\\
    \tilde{v}&=(\tilde{u}-\tilde{c})\tilde{\eta}_{\tilde{x}} \quad& \textup{on}\;\tilde{y}&=\tilde{\eta}(\tilde{x}),\\
    \tilde{v}&=(\tilde{u}-\tilde{c})\tilde{\zeta}_{\tilde{x}} \quad& \textup{on}\;\tilde{y}&=\tilde{\zeta}(\tilde{x}),
  \\P&=P_{\textup{atm}}\quad& \textup{on}\;\tilde{y}&=\tilde{\eta}(\tilde{x}),\\\jump{ P } &=0\quad& \textup{on}\;\tilde{y}&=\tilde{\zeta}(\tilde{x}).\\
\end{aligned}  \right.
\end{equation} 
Moreover, the asymptotic condition in \eqref{asymptotic} become
\begin{equation}\label{rescaledasym}
    (\tilde{u},\tilde{v})\rightarrow{(\tilde{\mathring{u}},0)},\quad \tilde{\varrho} \rightarrow{\tilde{\mathring{\varrho}}},\quad \tilde{\eta} \rightarrow{0}, \quad \tilde{\zeta} \rightarrow{-\dfrac{d_+}{d}} \qquad \textup{as} \;|\tilde{x}| \rightarrow{\infty}.
\end{equation} Combining \eqref{F,m} and \eqref{asymptotic} gives us

\begin{equation}\label{horizontalvelocityatinfinity}
    \tilde{\mathring{u}}(\tilde{y})-\tilde{c}=-\dfrac{1}{\sqrt{gd}}\tilde{u}^*(\tilde{y}).
\end{equation} 
Note that this means that the asymptotic state in \eqref{rescaledasym} is independent of $F$.

For the sake of cleaner notation, in what follows we will use the dimensionless variables but drop the tildes.  

\subsection{Stream function formulation}
Let us introduce the following \textit{relative pseudo stream function} $\psi$:
\begin{equation*}
    \psi_x=-\sqrt{\varrho}v, \quad \psi_y=\sqrt{\varrho}(u-c).
\end{equation*} 
The existence of $\psi$ is guaranteed by the incompressibility of the flow and the fact that density is constant along the streamlines. Indeed, from this definition we see that the streamlines are precisely the level sets of $\psi$. In particular, the kinematic boundary condition tells us that $\psi$ is constant on the surface, internal interface, and bed.  Without loss of generality, we may set $\psi=0$ on $\{y=\eta(x)\}.$ Thanks to equation \eqref{volumetric} together with the rescaling of coordinate, density and velocities, we then have $\psi=1$ on the floor $\{y=-1\}$. Let us denote its value on $\{y=\zeta(x)\}$ by $-\hat{p}$; the reason for this will become clear in the next subsection.  Observe that the assumption of no horizontal stagnation \eqref{nohorzstag} becomes:
\begin{equation}\label{nostagrnationstreamfuction}
    \psi_y<0\quad \textup{in}\; \overline{\Omega}.
\end{equation} 

Via the mass conservation in \eqref{massconservation}, we know that density is transported, hence constant, along each stream line. That allows us to rewrite the density in terms of $\psi$, otherwise known as \textit{streamline density function}:
\begin{equation*}
    \varrho(x,y)=\rho(-\psi(x,y)).
\end{equation*} 
Naturally, $\rho$ is determined by the limiting density profile $\mathring{\varrho}$.  It is easily verified that the regularity assumption \eqref{regularity ustar} implies $\rho \in C^{8+\alpha}([-1,\hat{p}]) \cap C^{8+\alpha}([\hat{p},0])$.  As the stratification here is assumed to be stable, moreover, we have that $\rho' \leq 0$ in upper and lower domain.

By Bernoulli's law, we know that the quantity
\begin{equation}\label{bernoulliequation}
    E=\dfrac{\varrho}{2}\big((u-c)^2+v^2\big)+P+\dfrac{1}{F^2}\varrho y
\end{equation} 
is constant along each streamline. This fact together with the no horizontal stagnation implies that there exists a so-called \textit{Bernoulli function} $\beta$ such that
\begin{equation}\label{bernoullifunction}
    \dfrac{dE}{d\psi}=-\beta(\psi)\quad \textup{in} \; \Omega.
\end{equation}
Since it is constant on streamlines, all of which extend fully upstream and downstream only on the stream function, $\beta$ can be reconstructed from the background current and density profile; see Remark~\ref{beta remark} below.  In particular, for the regularity assumed in \eqref{regularity ustar}, we find that $\beta \in C^{7+\alpha}([0,-\hat{p}]) \cap C^{7+\alpha}([-\hat{p},1])$.  Note that the somewhat odd looking choice to view $\rho$ as a function of $-\psi$ while $\beta$ is a function of $\psi$ is done here to be in accordance with previous results in the literature. 

Following \cite[Lemma A.2]{ChenWalsh2016Continuous}, the governing equations in \eqref{IncomEuler} with the absence stagnation can be reformulated as Yih's equation:
\begin{equation}\label{YIH}
    \Delta \psi-\dfrac{1}{F^2}y\rho'(-\psi)+\beta(\psi)=0 \quad \textup{in}\; \Omega.
\end{equation}

Likewise, the boundary conditions become
\begin{equation}\label{boundaryconditioninstreamfunction}
\left \{\begin{aligned}
    \psi&=0 &\quad \textup{on}\;y&=\eta(x),\\
    |\nabla \psi|^2 + \dfrac{2}{F^2}\varrho(y+1)&=Q^{\eta}&\quad \textup{on}\;y&=\eta(x),\\
    \jump{|\nabla \psi|^2} + \dfrac{2}{F^2}\jump{\varrho} (y+1)&=Q^{\zeta}&\quad \textup{on}\;y&=\zeta(x),
    \\\psi&=1 &\quad \textup{on}\;y&=-1,\\
\end{aligned}  \right.
\end{equation} 
where 
\begin{equation*}
    Q^{\eta}:=\left.2\left(E+\dfrac{1}{F^2}\varrho\right)\right|_{y=\eta(x)}
    \textrm{and }
    Q^{\zeta}:=\left.2\left(\jump{ E} +\dfrac{1}{F^2}\jump{\varrho}\right)\right|_{y=\zeta(x)}
\end{equation*}
are constants.  Lastly, the asymptotic conditions in \eqref{asymptotic} now read
\begin{equation}\label{asymptoticconditioninpsi}
\nabla \psi \rightarrow{\left(0,\sqrt{\mathring{\varrho}}(\mathring{u}-c)\right)}, \quad \eta \rightarrow{0}, \quad \zeta \rightarrow{-d_+}, \quad \varrho \rightarrow{\mathring{\varrho}}  \quad \textup{as} \; |x|\rightarrow{\infty}.
\end{equation} 
For later use, we introduce the convention that $\psi_\pm$ denotes the restriction of $\psi$ to $\Omega_\pm$. Via the continuity of the pressure and Bernoulli's law on the internal interface, we have
\begin{equation}\label{jump in grad psi}
    \dfrac{1}{2}\left(|\nabla {\psi}_+|^2-|\nabla {\psi}_-|^2\right)=\dfrac{\jump{\varrho }}{F^2}y-\jump{ E }.
\end{equation}
\begin{remark} \label{beta remark}
The Bernoulli function $\beta$ can be expressed in terms of $\mathring{\varrho}$ and $\mathring{u}$ as follows.  Letting $\mathring{y}(p)$ be the asymptotic $y$-coordinate of the streamline $\{\psi=-p\}$, and defining $\mathring{U}(p):=\mathring{u}(\mathring{y}(p))$, by \eqref{asymptoticconditioninpsi} we  have: 
\begin{equation}\label{asymptoticofy}
    \mathring{y}(p)=\mathlarger{\mathlarger{\int_{-1}^{p}}} \dfrac{1}{\sqrt{\rho(s)}\left(c-\mathring{U}(s)\right)}\, ds-1.
\end{equation}
Solving for $\beta$ in equation \eqref{YIH}, sending $x\to\pm\infty$ and applying \eqref{asymptoticconditioninpsi}, we obtain
\[\beta(-p)=\left(\dfrac{1}{F^2}\mathring{y}-\dfrac{1}{2}\left(\mathring{U}-c\right)^2\right)\rho_p+\rho \left(\mathring{U}-c\right)\mathring{U}_p.\]
\end{remark}
\subsection{Height function formulation}

\begin{figure}
\centering
\begin{tikzpicture}[xscale=1.25,yscale=0.75]

\draw [fill=gray,thin,gray] (-4,-1.7) rectangle (4,-1.5);
\draw [thin,-] (-4, -1.5) -- (4, -1.5);

\draw [dashed,domain=-4:4] plot[smooth] (\x, {-0.55+2.75/(exp(1.75*\x)+exp(-1.75*\x))}) node [above left] {\footnotesize $\psi = -p$};		
\draw [domain=-4:4,thick] plot[smooth] (\x, {0.25+2.5/(exp(1.2*\x)+exp(-1.2*\x))});		
\draw [domain=-4:4,thick] plot[smooth] (\x, {2.5+2.5/(exp(1.2*\x)+exp(-1.2*\x))});		

\draw [thick, stealth-stealth] (0.75,-1.5)--(0.75,0.1) node [midway,xshift=15pt] {\footnotesize $h(q,p)$};

\draw [thick,stealth-stealth] (3.5,-1.5)--(3.5,-.6) node [midway,xshift=12pt] {\footnotesize $H(p)$};

\draw[dashed] (0.75,0.2)--(0.75,4) node [below right] {\footnotesize $x = q$};

\node [below  ] at (-3.25,-0.9) {\footnotesize $\psi = 1$};
\node [below  ] at (-3.25,0.95) {\footnotesize $\psi = -\hat{p}$};
\node [below  ] at (-3.25,3.15) {\footnotesize $\psi = 0$};
\end{tikzpicture}

    \caption{The fluid domain with (non-dimensionalized) streamline values labeled.  The thick lines represent the upper and internal free boundaries.  Also depicted are the height function $h$ and asymptotic height $H$.}
    \label{dj figure}
\end{figure}
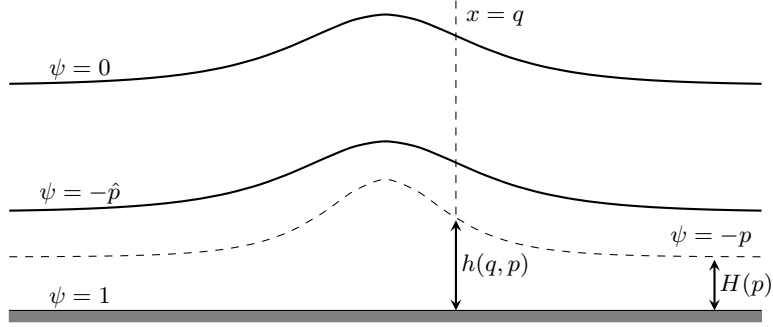

The fact that the Yih's equation is scalar is already a considerably simplification of the system in \eqref{IncomEuler}. However, due to the free boundary, the domain $\Omega$ remains a priori unknown which presents a serious difficulty for existence theory. To get around this, we employ the Dubreil-Jacotin transformation to send the domain $\Omega$ into a fixed slitted rectangular strip $R$:
\begin{equation}\label{semilagrangian}
    (x,y) \mapsto (x,-\psi)=:(q,p),
\end{equation}where
\begin{equation*}
    R:=R^{+} \cup R^{-}=\{(q,p)\in \mathbb{R}^2:p\in(-1,\hat{p})\} \cup \{(q,p)\in \mathbb{R}^2:p\in(\hat{p},0)\}.
   \end{equation*} The free boundary $\{y=\eta(x)\}$, floor $\{y=0\}$, and internal interface  $\{y=\zeta(x)\}$ are mapped to $T:=\{p=0\}$, $B:=\{p=-1\}$ and  $I:=\{p=\hat{p}\}$, respectively.
The new coordinates $(q,p)$ are often referred to as \textit{semi-Lagrangian variables}. 

Define the \textit{height function} which measures the height above the flat ocean floor,
\begin{equation*}
    h(q,p):=y+1 \geq 0 \quad \textup{in}\; \overline{R}.
\end{equation*} 
See Figure~\ref{dj figure} for an illustration.  
Via elementary computations, we obtain
\begin{equation*}
    h_q=\dfrac{v}{u-c}, \quad h_p=\dfrac{1}{\sqrt{\varrho}(c-u)},
\end{equation*} where the left-hand side is evaluated at $(q,p)$ and the right-hand side is evaluated at $(x,y)$. As a consequence, the absence of horizontal stagnation now translates to
\begin{equation*}
    h_p>0.
\end{equation*} 
Furthermore, the asymptotic conditions in \eqref{asymptotic} become
\begin{equation*}
    h(q,p) \rightarrow{H(p)}, \quad h_q(q,p)\rightarrow{0}, \quad h_p(q,p)\rightarrow{H_p(p)} \quad \textrm{as } |q| \rightarrow{\infty}.
\end{equation*} 
From equations \eqref{horizontalvelocityatinfinity}, \eqref{asymptoticconditioninpsi}, and \eqref{asymptoticofy}, we can view the asymptotic height function $H$ (downstream and upstream) as the solution to the following boundary value problem
\begin{equation*}
    \left \{\begin{aligned}
    H_p(p)&=\left.\dfrac{1}{\sqrt{\mathring{\rho}}(c-\mathring{u})}\right|_{y=H(p)-1} &\quad \textup{in}\; \left[-1,\hat{p}\right) \; \text{and} \;  \left[\hat{p},0\right),
    \\H(-1)&=0,\qquad H(\hat{p})=\dfrac{-d_+}{d}+1, \qquad H(0)=1.
\end{aligned}  \right.
\end{equation*} 
Yih's equation in \eqref{YIH} and the boundary conditions in\eqref{boundaryconditioninstreamfunction} can be written as the following quasi-linear PDE with transmission boundary condition:
\begin{equation}\label{quasi}
  \begin{cases}
    \left(-\dfrac{1+h_q^2}{2h_p^2}+\dfrac{1}{2H_p^2}\right)_p+\left(\dfrac{h_q}{h_p}\right)_q-\dfrac{1}{F^2}\rho_p(h-H)&=0 \qquad  \text{in}\; R,\\
\dfrac{1+h_q^2}{2h_p^2}-\dfrac{1}{2H_p^2}+\dfrac{1}{F^2}\rho (h-1)&=0 \qquad  \text{on}\; T,\\
\jump{ \dfrac{1+h_q^2}{2h_p^2} }-\jump{ \dfrac{1}{2H_p^2}} +\dfrac{1}{F^2}\jump{ \rho } (h-H)&=0 \qquad  \text{on}\; I,\\
h&=0 \qquad  \text{on}\; B.\\
\end{cases}  
\end{equation} 
The PDE in \eqref{quasi} is elliptic as long as $\inf_{R}h_p>0$.  The boundary condition on $I$ is of transmission type, while that on $T$ is oblique. Observe that for stably stratified flow, $-\rho_p \geq 0$, and hence the maximum principle cannot be applied directly.  This is a well-known feature of the problem that we will have to contend with at several stages of the analysis.

\subsection{Flow force} \label{flow force section} The \textit{flow force} is defined to be the 
\begin{equation}\label{flow force}
    \mathscr{S}(x)=\int_{-1}^{\eta(x)} \left( P+\varrho(u-c)^2 \right) \,dy.
\end{equation} 
One can check that this quantity is independent of $x$ if evaluated at a solution of the Euler equation.  Rewritten in semi-Lagrangian variables, it takes the form
\begin{equation} \label{flow force functional}
    \mathscr{S}(h)=\int_{-1}^{0}\left(\dfrac{1}{2h_p^2}+\dfrac{1}{2H_p^2}-\dfrac{1}{F^2}\rho(h-H)-\dfrac{1}{F^2}\int_{0}^{p} \rho H_p \,dp'\right)h_p\,dp,
\end{equation} 
where now we are viewing it as a functional acting on $h$ with $H$.  We will make use of the flow force in many ways.  For instance, in Section \ref{Small Amplitude} of small-amplitude theory, it gives rise to a conserved quantity on the center manifold that is essential to the construction.  More generally, the flow force is one of the three conserved quantities that determine the set of \textit{conjugate flows} for the system; see \cite{benjamin1971unified}.

\subsection{Function spaces and the operator equation}
In this subsection, we introduce the function spaces that we shall be working in. For a generic $D \subset \mathbb{R}^2$, non-negative integer $k$, and $\alpha \in [0,1)$, we define
\begin{equation*}
\begin{aligned}
C^{k+\alpha}(D)&:=\left\{f\in C^k(D): \norm{\phi f}_{C^{k+\alpha}}<\infty \textrm{ for all } \phi \in C^{\infty}_{0}\right\},\\
C^{k+\alpha}_{\textup{b}}(\overline{D})&:=\left\{f\in C^k(D): \norm{ f}_{C^{k+\alpha}}<\infty\right\},\\
C^{k+\alpha}_0(\overline{D})&:=\left\{f\in C^{k+\alpha}_{\textup{b}}(\overline{D}): \lim_{r \rightarrow{\infty}}\sup_{|x|=r}|\partial^j f(x)|=0 \textrm{ for all } 0\leq j\leq k\right\}.
\end{aligned}
\end{equation*} 
In particular, we emphasize that $C^{k+\alpha}$ refers to \textit{locally} H\"older continuous functions.  

The center manifold reduction carried out in 
Section~\ref{Small Amplitude} requires us to work with exponentially weighted H\"older space. For $\nu \in \mathbb{R}$, define
\begin{equation}\label{weightedholderspace}
    C^{k+\alpha}_{\nu}(\overline{D}):=\left\{f\in C^{k+\alpha}(\overline{D}):\norm{f}_{C^{k+\alpha}_{\nu}(D)}< \infty \right\},
\end{equation}
where the norm
\begin{equation}\label{normofweightedspace}
    \norm{f}_{C^{k+\alpha}_{\nu}(D)}:=\sum_{|\beta|\leq k}\norm{\sech(\nu q)\partial^{\beta}f}_{C^0(D)}+\sum_{|\beta|=k}\norm{\sech(\nu q)|\partial^{\beta}f|_{\alpha}}_{C^0(D)},
\end{equation}
and $|\cdot|_{\alpha}$ is the usual local H\"older seminorm. 

Finally, let $w = w(q,p)$ be
\begin{equation*}
 w(q,p):=h(q,p)-H(p),
\end{equation*}
which measures the deviation of the height function $h$ in the near-field from its limiting height $H$ at $q=\pm \infty$.  Note that  the decay of $w$ at infinity implies that the asymptotic conditions are satisfied.  One can see that the height equation \eqref{quasi} can be formulated in terms of $w$ as follows:
\begin{equation}\label{quasiw}
  \left \{\begin{aligned}
    \left(-\dfrac{1+w_q^2}{2(H_p+w_p)^2}+\dfrac{1}{2H_p^2}\right)_p+\left(\dfrac{w_q}{H_p+w_p}\right)_q-\dfrac{1}{F^2}\rho_pw&=0 \qquad  \text{in}\; R,\\
\dfrac{1+w_q^2}{2(H_p+w_p)^2}-\dfrac{1}{2H_p^2}+\dfrac{1}{F^2}\rho w&=0 \qquad  \text{on}\; T,\\
\jump{ \dfrac{1+w_q^2}{2(H_p + w_p)^2} }-\jump{ \dfrac{1}{2H_p^2}} +\dfrac{1}{F^2}\jump{ \rho } w &=0 \qquad  \text{on}\; I,\\
w&=0 \qquad  \text{on}\; B.\\
\end{aligned} \right.
\end{equation}

Define the following Banach Spaces,
\begin{equation}\label{BanachSpaces}
\begin{aligned}
    &X:= \{w \in C^{9+\alpha}_{\textup{b,e}}(\overline{R^{+}}) \cap C^{9+\alpha}_{\textup{b,e}}(\overline{R^{-}}) \cap C^{0}_0(\overline{R}) \cap C^{8}_{0}(\overline{R^{+}})\cap C^{8}_{0}(\overline{R^{-}}): w=0\; \text{on}\; B \},\\
    &Y_{1}:=C^{7+\alpha}_{\textup{b,e}}(\overline{R^{+}})\cap C^{7+\alpha}_{\textup{b,e}}(\overline{R^{-}}) \cap C^{6}_{0}(\overline{R^{+}}) \cap C^{6}_{0}(\overline{R^{-}}),\\ &Y_{2}:=C^{8+\alpha}_{\textup{b,e}}(T) \cap C^{7}_{0}(\overline{R^{+}})\cap C^{7}_{0}(\overline{R^{-}}),
\end{aligned}
\end{equation}
and set $Y:= Y_{1} \times Y_{2}$.
Throughout the paper, the subscript ``e" is used to indicates that the functions are even in $q$. We write \eqref{quasi} as an operator equation acting on functions in the Banach spaces above as
\[
    \mathcal{F}(w,F)=0,
\]
for the mapping
\begin{equation*}
    \mathcal{F}=(\mathcal{F}_1,\mathcal{F}_2,\mathcal{F}_3): U \subset X \times \mathbb{R} \rightarrow{Y}
\end{equation*} given by
\begin{equation}\label{F}
\begin{split}
    &\mathcal{F}_1(w,F):=\left(-\dfrac{1+w_q^2}{2(H_p+w_p)^2}+\dfrac{1}{2H_p^2}\right)_p+\left(\dfrac{w_q}{H_p+w_p}\right)_q-\dfrac{1}{F^2}\rho_p w,\\&
    \mathcal{F}_2(w,F):=\dfrac{1+w_q^2}{2(H_p+w_p)^2}-\dfrac{1}{2H_p^2}+\dfrac{1}{F^2}\rho w,\\&
    \mathcal{F}_3(w,F):=\jump{ \dfrac{1+w_q^2}{2(H_p + w_p)^2} }-\jump{ \dfrac{1}{2H_p^2}} +\dfrac{1}{F^2}\jump{ \rho } w.
\end{split}
\end{equation}
We are looking for solutions that belong in the open subset 
\begin{equation}\label{solutions set}
\mathscr{U}:=\left\{(w,F)\in X \times \mathbb{R}:\inf_R(w_p+H_p)>0, F> F_{\textup{cr}}\right\}\subset X \times \mathbb{R}.
\end{equation} 
Here $\Fcr$ is the critical Froude number which will be defined later in Section~\ref{Strum--Liouville type}. Since $\mathcal{F}$ is a rational function of $w$ and its derivatives, then it is a real-analytic mapping from $\mathscr{U}$ to $Y$. 

\section{Linearized operators}\label{Linearized operators}
This section is devoted to investigating the linearized operator $\mathcal{F}_w(w,F)$. The results presented in Section~\ref{Strum--Liouville type} concern a Sturm--Liouille-type problem related to the case $w=0$; this will be used to define the critical Froude number $\Fcr$. Section~\ref{Fredholm property} analyzes the linearized operator at an arbitrary $(w,F)$, which plays a crucial role in proving the local and global existence theory in Sections~\ref{Small Amplitude} and \ref{largeamplitude}.
\subsection{Sturm--Liouville type problems} \label{Strum--Liouville type}
Let us first consider the spectrum of the \textit{transversal linearized operator} at the laminar flow $(w,F) = (0,F)$, by which we mean the restriction of $\mathcal{F}_w(0,F)$ to functions that are independent of $q$. Thus, we obtain the following Sturm--Liouville-type problem

\begin{equation} \label{sturm}
  \left \{\begin{aligned}
    \left(\dfrac{\dot{w}_p}{H_p^{3}}\right)_p-\mu \rho_{p}\dot{w}&=-\nu \dfrac{\dot{w}}{H_p} &\qquad  &\text{in}\;\left(-1,\hat{p}\right) \; \text{and} \;  \left(\hat{p},0\right),\\
-\dfrac{\dot{w}_p}{H_p^3}+\mu \rho \dot{w}&=0 &\qquad  &\text{on}\; p=0,\\
-\jump{ \dfrac{\dot{w}_p}{H_p^3} }+\mu \jump{ \rho} \dot{w} &=0 &\qquad  &\text{on}\; p=\hat{p},\\
\dot{w}&=0 &\qquad  &\text{on}\; p=-1,\\
\end{aligned}  \right.
\end{equation}
where $\mu := 1/F^2$ and $\nu$ is the eigenvalue.

Heuristically, we expect all spatially localized gravity waves to be supercritical in that their wave speed is faster than any linear periodic wave.  This idea underlies our approach to constructing small-amplitude solitary waves in Section~\ref{Small Amplitude}.  By separating variables, we see that the linearized problem at $w =0$ admits periodic solutions provided \eqref{sturm} has a positive eigenvalue.  Thus we wish to identify a critical Froude number (which recall is the non-dimensionalized wave speed) at which the transversal linearized problem at the background laminar flow has a $0$ as its principal eigenvalue. 

With that in mind, take $\nu=0$ in \eqref{sturm} above and look for the largest value of $\mu$ such that
\begin{equation} \label{eval0}
 \left \{ \begin{aligned}
    \left(\dfrac{\dot{w}_p}{H_p^{3}}\right)_p-\mu \rho_{p}\dot{w}&=0\qquad  \text{in}\;\left[-1,\hat{p}\right) \; \text{and} \;  \left[\hat{p},0\right),\\
-\dfrac{\dot{w}_p}{H_p^3}+\mu \rho \dot{w}&=0 \qquad  \text{on}\; p=0,\\
-\jump{ \dfrac{\dot{w}_p}{H_p^3} }+\mu \jump{ \rho} \dot{w} &=0 \qquad  \text{on}\; p=\hat{p},\\
\dot{w}&=0 \qquad  \text{on}\; p=-1,\\
\end{aligned} \right.
\end{equation}
has a nontrivial solution. To achieve this, we will consider the solution $\Phi (p; \mu)$ to the initial value problem
\begin{subequations}\label{IVP}
\begin{equation}\label{IVP1}
  \left \{ \begin{aligned}
    \left(\dfrac{\Phi_p}{H_p^{3}}\right)_p&=\mu \rho_{p}\Phi& \qquad & \textrm{in } \left(-1,\hat{p}\right), \\
    \Phi&=0 & \qquad & \textrm{on } p=-1,\\
    \Phi_p&=1 & \qquad & \textrm{on } p=-1.\\
\end{aligned} \right.
\end{equation}
Using the jump condition on $\{p = \hat{p}\}$ in \eqref{eval0}, we continue this solution into the upper layer corresponding to $p\in \left(\hat{p},0\right)$.  We denote this extended function by $\Upsilon$, which is thus determined by 
\begin{equation}\label{IVP2}
  \left \{ \begin{aligned}
  \left(\dfrac{\Upsilon_p}{H_p^{3}}\right)_p&=\mu \rho_{p}\Upsilon & \qquad & \textrm{in } \left(\hat{p},0\right), \\
  \Upsilon&=\Phi & \qquad & \textrm{on } p=\hat{p}, \\
  \Upsilon_p&=H_p^3(\hat{p}^+)\dfrac{\Phi_p(\hat{p}^-)}{H_p^3(\hat{p}^-)}+\mu \jump{ \rho } H_p^3(\hat{p}^+)\Phi  & \qquad & \textrm{on } p=\hat{p}.\\
\end{aligned} \right.
\end{equation}
\end{subequations}
Finally, to satisfy the Bernoulli condition on $\{p = 0\}$, we introduce the function 
\begin{equation}\label{A}
    A(\mu)=\dfrac{-\Upsilon(0;\mu)}{H_p^3(0)}+\mu \rho(0) \Upsilon(0;\mu).
\end{equation}
The idea here is that if $A(\mu)=0$, then
\begin{equation}\label{newfunction}
\Psi(p;\mu):=\begin{cases}
    \Upsilon(p;\mu)\; \text{for}\; \hat{p}\leq p \leq 0\\
    \Phi(p;\mu)\; \text{for}\; -1 \leq p \leq \hat{p}\\
\end{cases}.
\end{equation}
solves the IVP \eqref{IVP}, and hence is an eigenfunction for \eqref{sturm} corresponding to the eigenvalue $\nu=0$.
\begin{lemma} \label{zerotheigenfunction}
There exists a unique $\mu_{\textup{cr}} > 0$ such that all of the following hold.
\begin{enumerate}[label=\normalfont{(\alph*)}]
    \item\label{a} For $\mu = \mucr$, the problem \eqref{eval0} has a nontrivial solution $\dot{w} = \Psi{(p;\mu_{\textup{cr}})}$.
     \item\label{b} For $0 \leq \mu \leq \mu_{\textup{cr}}$, $\Psi (p;\mu)>0$ for $-1<p\leq0$ and $\Psi_p(p,\mu)>0$\; for\; $-1 \leq p < \hat{p}$ and $\hat{p}< p \leq 0$.
     \item\label{c} For $0 \leq \mu < \mucr$, $A(\mu)<0$. 
\end{enumerate}
\end{lemma}
\begin{proof}
Note that \eqref{eval0} has a non-trivial solution provided $B(\mu)=\mu \rho(0) H_p^3(0)$ where we define \begin{equation}\label{Bmu0}
    B(\mu):=\dfrac{\Upsilon_p(0;\mu)}{\Upsilon(0;\mu)}.
\end{equation}
Setting $\mu=0$ and integrating the first equation in \eqref{IVP1} and \eqref{IVP2}, we obtain
\begin{equation}\label{Psi_pPsi}
    \Psi_p(p;0)= \left\{\begin{aligned}
    &\dfrac{H_p^3(p)}{H_p^3(-1)}&\qquad &\textrm{for } p \in [-1, \hat{p}), \\&\dfrac{\Upsilon_p(\hat{p}^+)}{H_p^3(\hat{p}^+)}H_p^3(p)&\qquad &\textrm{for } p \in(\hat{p},0].
    \end{aligned} \right.
\end{equation}
Anti-differentiating \eqref{Psi_pPsi} once more then gives
\begin{equation}\label{PsiPsi}
    \Psi(p;0)=\left \{\begin{aligned}
    &\dfrac{1}{H_p^3(-1)} \int_{-1}^{p}H_p^3(s)\, ds&\qquad& \textrm{for } p \in[-1,\hat{p}),\\&
    \Upsilon(\hat{p}^+)+\dfrac{\Upsilon_p(\hat{p}^+)}{H_p^3(\hat{p}^+)} \int_{\hat{p}}^{p} H_p^3(s)\, ds &\qquad&\textrm{for } p \in(\hat{p},0].
    \end{aligned} \right.
\end{equation} 
Inserting this into \eqref{Bmu0} yields the formula
\begin{equation}\label{B(0)}
    B(0)=\dfrac{\Upsilon_p(\hat{p}^+)H_p^3(0)}{H_p^3(\hat{p}^+)\Upsilon(\hat{p}^+)+\Upsilon_p(\hat{p}^+)\int_{\hat{p}}^{0} H_p^3(p)\, dp}.
\end{equation} 

We claim that 
\begin{equation}\label{B(0)positive}
    B(0)>0.
\end{equation}
Observe that substituting $p=\hat{p}^-$ into the first equation in \eqref{Psi_pPsi} and \eqref{PsiPsi} leads to 
\begin{subequations}
\begin{equation}\label{PHIp}
    \Phi_p(\hat{p}^-;0)=\dfrac{H_p^3(\hat{p}^-)}{H_p^3(-1)}>0,
\end{equation}
and
\begin{equation}\label{Phi}
    \Phi(\hat{p}^-;0)=\dfrac{1}{H_p^3(-1)}\int_{-1}^{\hat{p}} H_p^3(p) \, dp>0.
\end{equation}
\end{subequations}
Because $\Phi(\hat{p})=\Upsilon(\hat{p})$, we have from the inequality in \eqref{Phi} that $\Upsilon(\hat{p}^+;0)>0$. Using equation \eqref{PHIp} in concert with the last equation in \eqref{IVP2} leads to $\Upsilon_p(\hat{p}^+;0)>0$. The desired inequality \eqref{B(0)positive} now follows.

Next, we claim that 
\begin{equation}\label{Bmu<0}
B_{\mu}<0.
\end{equation}
Differentiating $B$ with respect to $\mu$ gives
\begin{equation}\label{Bmu}
    B_{\mu}=\left.\dfrac{\Upsilon_{\mu p} \Upsilon - \Upsilon_p \Upsilon_{\mu}}{\Upsilon^2}\right|_{p=0}.
\end{equation}
Note, differentiating the last two equations of \eqref{IVP2} with respect to $\mu$ yields 
\begin{equation}\label{mu}
     \left \{ \begin{aligned}
    \Upsilon_{\mu}&=\Phi_{\mu}\\
    \Upsilon_{\mu p}&=H_p^3(\hat{p}^+)\dfrac{\Phi_{\mu p}(\hat{p}^-)}{H_p^3(\hat{p}^-)}+ \jump{ \rho } H_p^3(\hat{p}^+)\Phi+\mu \jump{ \rho } H_p^3(\hat{p}^+)\Phi_{\mu}\\
\end{aligned}\qquad \textrm{on } p=\hat{p}. \right.  
\end{equation}
Furthermore, by differentiating \eqref{IVP1} with respect to $\mu$, we obtain the following problem
\begin{equation}\label{IVP1mu}
  \left \{ \begin{aligned}
    \left(\dfrac{\Phi_{\mu p}}{H_p^{3}}\right)_p&=\rho_p \Phi +\mu \rho_{p}\Phi_{\mu} \qquad & \textrm{in } \left(-1,\hat{p}\right),\\
    \Phi_{\mu}=0&=\Phi_{\mu p} \qquad & \textrm{on } p= -1.
\end{aligned} \right. 
\end{equation}
A computation using integration by parts gives us, 
\begin{equation}\label{IBP}
   \left.\left( \dfrac{\Upsilon_{\mu p}\Upsilon-\Upsilon_p \Upsilon_{\mu}}{H_p^3}\right)\right|_{p=0}-\left.\left(\dfrac{\Upsilon_{\mu p}\Upsilon-\Upsilon_p \Upsilon_{\mu}}{H_p^3}\right)\right|_{p=\hat{p}}=\int_{\hat{p}}^{0} \rho_p \Upsilon^2(p;\mu)\;dp.
\end{equation}
By rearranging terms we have that
\begin{equation*}
\begin{aligned}
 B_{\mu}&=\left.\left(\dfrac{\Upsilon_{\mu p}\Upsilon-\Upsilon_p \Upsilon_{\mu}}{\Upsilon^2}\right)\right|_{p=0}\\&=\dfrac{H_p^3(0)}{\Upsilon^2(0)}\left[\left.\left(\dfrac{\Upsilon_{\mu p}\Upsilon-\Upsilon_p \Upsilon_{\mu}}{H_p^3}\right)\right|_{p=\hat{p}}+\int_{\hat{p}}^{0} \rho_p \Upsilon^2(p;\mu)\; dp\right].
\end{aligned}
\end{equation*}
Therefore, to verify our claim in \eqref{Bmu}, it suffices to show
\begin{equation*}
    \left.\left( \Upsilon_{\mu p}\Upsilon-\Upsilon_p \Upsilon_{\mu}\right)\right|_{p=\hat{p}}<0.
\end{equation*}
Differentiating the first equation in \eqref{IVP1} with respect to $\mu$ and testing against $\Phi$ yields
\begin{equation*}
   \left.\left(\dfrac{\Phi_{\mu p}\Phi-\Phi_p \Phi_{\mu}}{H_p^3}\right)\right|_{p=-1}^{p=\hat{p}}=\int_{-1}^{\hat{p}} \rho_p \Phi^2(p;\mu)\;dp <0.
\end{equation*} Recall that $\Phi_{\mu p}(-1)=0=\Phi_{\mu}(-1)$. Hence, the above inequality simplifies into
\begin{equation}\label{claim}
    \dfrac{\Phi_{\mu p}(\hat{p}^-)}{\Phi_p(\hat{p}^-)}<\dfrac{\Phi_{\mu}(\hat{p}^-)}{\Phi(\hat{p}^-)}=\dfrac{\Upsilon_{\mu}(\hat{p}^+)}{\Upsilon(\hat{p}^+)},
\end{equation}
which is equivalent to 
\begin{equation}\label{differencePhiUpsilon}
    \Phi_{\mu p}\Upsilon-\Phi_{p}\Upsilon_{\mu}<0 \qquad \textrm{on } p=\hat{p}.
\end{equation}
Multiplying the second equation in \eqref{mu} by $\Upsilon_{\mu}$ reveals that
\begin{equation}\label{1}
    \Upsilon_{\mu p}\Upsilon=H_p^3(\hat{p}^+)\dfrac{\Phi_{\mu p}(\hat{p}^-)\Upsilon}{H_p^3(\hat{p}^-)}+ \jump{ \rho } H_p^3(\hat{p}^+)\Phi \Upsilon+\mu \jump{ \rho } H_p^3(\hat{p}^+)\Phi_{\mu}\Upsilon  \qquad  \textrm{on } p=\hat{p}.
\end{equation}
On the other hand, multiplying the third equation in \eqref{IVP2} by $\Upsilon_{\mu}$ and evaluating it at $p=\hat{p}$ gives
\begin{equation}\label{2}
    \Upsilon_p\Upsilon_{\mu} =H_p^3(\hat{p}^+)\dfrac{\Phi_p(\hat{p}^-)\Upsilon_{\mu}}{H_p^3(\hat{p}^-)}+\mu \jump{ \rho } H_p^3(\hat{p}^+)\Phi \Upsilon_{\mu}  \qquad  \textrm{on } p=\hat{p}.
\end{equation}
Subtracting \eqref{1} from \eqref{2} and using the fact in \eqref{differencePhiUpsilon}, we know that $\left(\Upsilon_{\mu p}\Upsilon-\Upsilon_p \Upsilon_{\mu}\right)<0$ on $p=\hat{p}$. This proves the claim in \eqref{Bmu<0} provided that $\Upsilon(0) \neq 0$.

Now, combining our earlier observations that $B_{\mu}<0$ and $B(0)>0,$ we can infer that there exists a unique smallest $\mu_{\textup{cr}}$ such that $B(\mu)=\mu \rho(0) H_p^3(0).$ This proves part~\ref{a} of the lemma. Observe, by the uniqueness of solution to initial value problem, the numerator and denominator in \eqref{Bmu0} cannot vanish altogether. Thus collectively these facts show that $B(\mu),\Upsilon_p(0,\mu),\Upsilon(0,\mu)$ are all strictly positive quantities for $0\leq \mu \leq \mu_{\textup{cr}}$. Part~\ref{c} of the lemma is then a direct consequence of the fact that $B(\mu)>\mu \rho(0) H_p^3(0)$ for all $0\leq \mu < \mu_{\textup{cr}}$. 

It remains only to prove part~\ref{b}.  We first consider the sign of $\Psi$ through a continuity argument. Define the set 
\begin{equation*}
        \mathcal{E}:=\left\{\mu \in [0,\mu_{\textup{cr}}]:\Psi(p;\mu)>0\; \text{for } p \in (-1,0] \right\}. 
\end{equation*}
Observe that $0\in \mathcal{E}$ due to \eqref{Psi_pPsi}. We claim that $\mathcal{E}$ is closed.  Seeking a contradiction, suppose that $\mathcal{E}$ has a limit point $\tilde{\mu}$ and there exists $\tilde{p}\in [-1,0]$ so that $\Psi(\tilde{p},\tilde{\mu})=0.$ By continuity, we can infer that $\Psi(p,\tilde{\mu})\geq 0$ for all $p \in [-1,0]$, and hence $\Psi(\cdot;\tilde{\mu})$ attains its minimum at $p=\tilde{p}.$ In particular, this implies $\Psi_p(\tilde{p},\tilde{\mu})=0$, where notice this would be true even in the case $\tilde p = \hat{p}$. But then $\Psi_p(\tilde{p},\tilde{\mu})=\Psi(\tilde{p},\tilde{\mu})=0$,  and so $\Psi$ vanishes identically by uniqueness. Thus we have arrived at a contradiction meaning $\mathcal{E}$ is closed.  On the other hand, $\mathcal{E}$ is clearly open because $\Psi$ is continuous in $\mu$ and we have already shown that $\Psi_p > 0$ at $p=0,-1$.  It follows then that $\mathcal{E} = [0,\mucr]$.

Finally, we establish the sign of $\Psi_p$ claimed in part~\ref{b}. Fix $\mu\in[0,\mu_{\textup{cr}}]$ and consider the function
\[g(p):=\dfrac{\Psi_p(p;\mu)}{H_p^3(p)}.\] 
Clearly, $g(-1)>0$ and since $\Psi_p(0;\mu) > 0$, we have $g(0)>0.$  Moreover, the equation satisfied by $\Psi$ gives the identity $g_p=\mu \rho_p \Psi.$ From this it is easily seen that $g(\hat{p}^+)>0.$ This shows that $\Psi_p(p;\mu)>0$ for $\hat{p}< p \leq 0.$ Furthermore, from \eqref{IVP2}, we have
\[\dfrac{\Upsilon_p(\hat{p}^+;\mu)}{H_p^3(\hat{p}^+)}-\mu \jump{ \rho } \Phi(\hat{p};\mu)= \dfrac{\Phi_p (\hat{p}^-;\mu)}{H_p^3(\hat{p}^-)}.\] Notice that the left hand side of the above equation is strictly positive. We then conclude that $g(\hat{p}^-)>0$. Hence, we can conclude that $\Psi_p(p;\mu)>0$ for all $-1 \leq p < \hat{p}.$ This gives the desired inequality in part \ref{b} of the lemma. 
\end{proof}

\begin{lemma}[Spectrum] \label{SpectrumofEvalues}
Let $\Sigma$ denote the set of eigenvalues for the problem in \eqref{sturm} at $\mu=\mucr$. 
\begin{enumerate}[label=\normalfont{(\alph*)}]
    \item\label{SigmaA} $\Sigma =\{\nu_j\}_{j=0}^{\infty}$ such that $\nu_j \rightarrow{-\infty}$ as $j\rightarrow{\infty}$ and $\{\nu_j\}_{j=0}^{j=\infty}$ is a strictly decreasing sequence,
     \item\label{SigmaB} $\nu_{0}=0$, and
     \item\label{SigmaC} each  eigenvalue has algebraic and geometric multiplicity $1$.
\end{enumerate}
\end{lemma}
\begin{proof}
Fix $\mu=\mucr$. Similar in spirit to the proof of Lemma \ref{zerotheigenfunction}, we begin by introducing the function $N(p;\nu)$ which solves the following initial value problem
\begin{equation}\label{eigenvalueproblem1}
  \left \{ \begin{aligned}
    \left(\dfrac{N_p}{H_p^{3}}\right)_p-\mu_{\textup{cr}} \rho_{p}N&=-\nu \dfrac{N}{H_p} & \qquad & \textrm{in } \left(-1,\hat{p}\right) \textrm{ and } (\hat{p}, 0), \\
  \jump{\dfrac{N_p}{H_p^3}}&=\mucr \jump{ \rho }N  & \qquad & \textrm{on } p=\hat{p},\\
    N&=0 & \qquad & \textrm{on } p=-1,\\
    N_p&=1 & \qquad & \textrm{on } p=-1,\\
\end{aligned} \right.
\end{equation} 
and the associated function
\begin{equation}\label{Bnu}
    B(\nu):=\dfrac{N_p(0;\nu)}{N(0;\nu)}.
\end{equation}
Observe that $\dot{w}:=N(p;\nu)$ solves \eqref{sturm} provided that $B(\nu)=\mu_{\textup{cr}}\rho(0) H_p^3(0)$. By construction, $B$ has singularity at each eigenvalue $\nu_{\textrm{D}}$ of the Dirichlet problem
\begin{equation} \label{dirichletproblem}
  \left \{\begin{aligned}
   \left(\dfrac{\dot{w}_p}{H_p^{3}}\right)_p-\mu_{\textup{cr}} \rho_{p}\dot{w}&=\nu_{\mathrm{D}} \dfrac{\dot{w}}{H_p}, &\qquad  &\text{in}\;\left(-1,\hat{p}\right) \; \text{and} \;  \left(\hat{p},0\right),\\
\dot{w}&=0 &\qquad  &\text{on}\; p=0,\\
\jump{ \dfrac{\dot{w}_p}{H_p^3} }-\mu_{\textup{cr}}  \jump{ \rho} \dot{w} &=0 &\qquad  &\text{on}\; p=\hat{p},\\
\dot{w}&=0 &\qquad  &\text{on}\; p=-1.
\end{aligned}  \right.
\end{equation}

It is well-known that the set of Dirichlet eigenvalues takes the form $\Sigma_{\textrm{D}}=\{\nu_{\textrm{D}}^{(j)}\}_{j=1}^{\infty}$, where each $\nu_{\mathrm{D}}^{(j)}$ is simple, $\nu_{\textrm{D}}^{(j)} \geq \nu_{\textrm{D}}^{(j+1)}$ for all $j\in \mathbb{Z}^+$, and $\nu_{\textrm{D}}^{(j)} \rightarrow{-\infty}$ as $j\rightarrow{\infty}$. We claim, moreover, that each Dirichlet eigenvalue is negative. For the sake of contradiction, suppose that there exists a $\nu_{\textrm{D}} \geq 0$ in $\Sigma_{\mathrm{D}}$ with eigenfunction $\dot{w}$. For $0 < \delta \ll 1$, define
\begin{equation*}
    \dot{v}^\delta := \dfrac{\dot{w}}{  -(\Psi_{\textup{cr}}+\delta)}.
\end{equation*} Using equations \eqref{IVP1}
 and \eqref{IVP2} for $\Psi$ together with the Dirichlet problem \eqref{dirichletproblem}, one can show that 
 \begin{equation}\label{modifiedirichlet}
     -\left(\dfrac{\Psi_{\textup{cr}}+\delta}{H_p^3}\dot{v}_p^{\delta}\right)_p-\dfrac{\left(\Psi_{\textup{cr}}\right)_p}{H_p^3}\dot{v}_p^{\delta}+\left(\mu_{\textup{cr}} \rho_p \delta+\nu_{\mathrm{D}}\dfrac{\Psi_{\textup{cr}}+\delta}{H_p}\right)\dot{v}^{\delta}=0.
 \end{equation} If $\nu_D>0$, we can choose small enough $\delta$ such that the coefficient of the zeroth order term in \eqref{modifiedirichlet} is positive. For $\nu_D=0$, we can send  $\delta \to 0$ such that the coefficient of the zeroth order term in \eqref{modifiedirichlet} to vanish. Thus, in both cases, taking $0 < \delta \ll 1$, we can apply the maximum principle to conclude $\dot{v}^{\delta}\equiv0$ and thus $\dot{w} \equiv 0$. Having arrived at a contradiction, we therefore infer that all the elements of $\Sigma_{\mathrm{D}}$ are strictly negative.
 
 Now, differentiating \eqref{eigenvalueproblem1} with respect to $\nu$ yields
 \begin{equation*}
  \left \{ \begin{aligned}
    \left(\dfrac{N_{\nu p}}{H_p^{3}}\right)_p-\mu_{\textup{cr}} \rho_{p}N_\nu&=-\dfrac{N}{H_p}-\dfrac{\nu N_\nu}{H_p} & \qquad & \textrm{in } \left(\hat{p},0\right), \\
    N_\nu(-1;\nu)&=0=N_{\nu p}(-1;\nu).
\end{aligned} \right.
\end{equation*} 
Testing the above equation against $N$ and comparing it to  \eqref{eigenvalueproblem1} tested against $N_\nu$ yields the Green's identity
\begin{equation*}
   \left.\left(\dfrac{N_p N_\nu}{H_p^3}-\dfrac{N_{\nu p}N}{H_p^3}\right)\right|_{p=\hat{p}}^{0}=\int_{\hat{p}}^{0} \dfrac{N^2(p;\nu)}{H_p}\;dp.
\end{equation*}Hence we have
\begin{equation*}
    B'(\nu)=\dfrac{-H_p^3(0)}{N^2(0)}\int_{\hat{p}}^{0} \dfrac{N^2(p;\nu)}{H_p}\;dp<0,
\end{equation*} as long as $N(0;\nu)\neq 0$. We then conclude that $B$ is strictly decreasing on $\Sigma_{\mathrm{D}}^c$.  Thus, we must have that $B(\nu) \rightarrow{\pm \infty}$ as $\nu \rightarrow{\nu_D^{{(j)}\pm}}$ for each $j\in \mathbb{Z}^+$. In particular, on each connected component of $\Sigma_{\mathrm{D}}^c$, there exists a unique $\nu \in (\nu_D^{(j)},\nu_D^{(j+1)})$ such that $B(\nu)=\mu_{\textup{cr}}\rho(0) H_p^3(0)$.  Likewise, on  $(\nu_D^{(1)},\infty)$ there can be at most one such value of $\nu$.  

This analysis shows that the eigenvalues of \eqref{sturm} are intertwined with those of the Dirichlet problem \eqref{dirichletproblem}.  We have already proved in Lemma~\ref{zerotheigenfunction}\ref{a} that $0 \in \Sigma$, and hence it is the unique element of $\Sigma$ in the interval $(\nu_{\mathrm{D}}^{(1)}, \infty)$. This implies further implies that $\nu_{\mathrm{D}}^{(2)} < 0$, so parts~\ref{SigmaA} and \ref{SigmaB} now follow.  Part~\ref{SigmaC} is easily verified from classical Sturm--Liouville theory.
\end{proof}

Finally, let us conclude the subsection by recalling that $\mu=1/F^2$, hence the critical Froude number is defined by 
\begin{equation} \label{def Fcr}
\Fcr^2 := \frac{1}{\mucr},
\end{equation}
with $\mucr$ given by Lemma~\ref{zerotheigenfunction}.

\subsection{Fredholm property}\label{Fredholm property} Now we focus our attention on the full linearized operator at $(w,F)\in \mathscr{U}$. Consider first the  problem $\mathcal{F}_w(0,F)\dot{w}=(f_1,f_2,f_3)$, which reads
\begin{equation*}
 \begin{aligned} \label{F_w(0,F)}
    \left(\dfrac{\dot{w}_p}{H_p^{3}}\right)_p+ \left(\dfrac{\dot{w}_q}{H_p}\right)_q-\dfrac{1}{F^2} \rho_{p}\dot{w}&=f_1\qquad  \text{in } R,\\
-\dfrac{\dot{w}_p}{H_p^3}+\dfrac{1}{F^2} \rho \dot{w}&=f_2 \qquad  \text{on } T,\\
-\jump{ \dfrac{\dot{w}_p}{H_p^3} }+\dfrac{1}{F^2}\jump{ \rho} \dot{w} &=f_3 \qquad  \text{on }I,\\
\dot{w}&=0 \qquad \; \; \text{on }B.\\
\end{aligned} 
\end{equation*}

Although we know that $\mathcal{F}_w(0,F)$ is a map from $X$ to $Y$,  here we shall view it as a map between the larger spaces $X_{\textup{b}}$ to $Y_{\textup{b}}$, where
\begin{align*}
X_{\textup{b}} & :=\left\{w \in C^0_0(\overline{R}) \cap  C^{9+\alpha}_{\textup{b}}(\overline{R^+})\cap C^{9+\alpha}_{\textup{b}}(\overline{R^-}): w|_{B}=0\right\}, \\ Y_{\textup{b}} &:= \left(C^{7+\alpha}_{\textup{b}}(\overline{R^+}) \cap C^{7+\alpha}_{\textup{b}}(\overline{R^-})\right) \times C^{8+\alpha}_{\textup{b}}(T)\times C^{8+\alpha}_{\textup{b}}(I).
\end{align*}
That is, the requirement that the solution decays at infinity has temporarily been lifted.

Observe that the zeroth-order term in the interior equation in \eqref{F_w(0,F)} has the ``bad" sign in that it does not satisfy the  the assumptions of the maximum principle Theorem~\ref{Maximum Principle}. For supercritical waves, we can fix this by introducing a function $\tilde{\Psi}$ that is a variant of the function $\Psi$ found in Lemma~\ref{zerotheigenfunction}. It is defined as the solution to the same ODE \eqref{IVP} but with initial conditions $\tilde \Psi(-1) = \epsilon$ and $\tilde \Psi_p(-1) = 1$, for some $0 < \epsilon \ll 1$ that depends only on $F$.
\begin{lemma}\label{Psitilde}
Suppose $F>\Fcr$, then for $\epsilon>0$ sufficiently small, there exists  $\tilde \Psi$ satisfying
\[ \left( \frac{\tilde \Psi_p}{H_p^3}\right)_p - \frac{1}{F^2} \rho_p \tilde \Psi = 0 \qquad \textrm{in } (-1,\hat{p}) \cup (\hat{p}, 0),\]
with 
\begin{subequations}
\begin{equation}\label{Psi tilde Psi_p Tilde}
    \tilde{\Psi} >0 \textrm{ for } -1< p\leq 0,\quad \tilde{\Psi}_p>0  \textrm{ for } [-1, \hat{p})\; and \; (\hat{p},0],
\end{equation}
\begin{equation}\label{Inequality A(Psi tilde)}
-\dfrac{\tilde{\Psi}_p}{H^{3}_p}+\dfrac{1}{F^2}\rho \tilde{\Psi}<0\quad \text{ on} \quad p=0,
\end{equation} and
\begin{equation}\label{Jump Psi tilde}
\jump{\dfrac{\tilde{\Psi}_p}{H^{3}_p}}-\dfrac{1}{F^2} \jump{ \rho } \tilde{\Psi}<0\quad  \text{on} \quad p=\hat{p}.
\end{equation}
\end{subequations}
\end{lemma}
\begin{proof}
By definition, $\tilde{\Psi}=\Psi$ when $\epsilon=0$. Adapting the proof of Lemma~\ref{zerotheigenfunction}, it is easy to see that equations \eqref{Inequality A(Psi tilde)} and the second inequality in \eqref{Psi tilde Psi_p Tilde} hold for $0<\epsilon\ll1$. Further, the first inequality in \eqref{Psi tilde Psi_p Tilde} can be obtained by integrating the second one where $\epsilon>0$ is chosen to be sufficiently small. Lastly, to arrive at \eqref{Jump Psi tilde}, we use the transmission equation from the ODE \eqref{IVP2} together with the boundary condition $\tilde{\Psi}(-1,\mu)=\epsilon$.
\end{proof}

Now, letting $\dot{w}=:\tilde{\Psi}v$, we see that \eqref{F_w(0,F)} is equivalent to the following more amenable problem for $v$:
\begin{equation}\label{modified F_w(0,F)}
 \left \{ \begin{aligned} 
    \left(\dfrac{v_p}{H_p^{3}}\right)_p+ \left(\dfrac{v_q}{H_p}\right)_q&=\dfrac{f_1}{\tilde{\Psi}}\qquad  \text{in } R,\\
-\dfrac{v_p}{H_p^3}+\dfrac{1}{\tilde{\Psi}} \left(-\dfrac{\tilde{\Psi}_p}{H_p^3}+\dfrac{1}{F^2}\rho \tilde{\Psi}\right)v &=\dfrac{f_2}{\tilde{\Psi}}\qquad  \text{on } T,\\
-\jump{ \dfrac{v_p}{H_p^3} }&=\dfrac{f_3}{\tilde{\Psi}} \qquad  \text{on }I,\\
v&=0 \qquad \;\; \text{on }B.\\
\end{aligned} \right.
\end{equation}
\begin{lemma}[Invertibility] \label{Invertibility of $F_w(0,F)$}
For $F > \Fcr$, $\mathcal{F}_w(0,F)$ is an invertible map from $X_{\textup{b}} \textrm{ to } Y_{\textup{b}}$ and from $X\textrm{ to } Y$.
\end{lemma}
\begin{proof}
Because the maximum principle can be applied to \eqref{modified F_w(0,F)}, one can show following \cite[Lemma A.5]{wheeler2013large} and \cite[Lemma A.1]{chen2018existence} that  $\mathcal{F}_w(0,F)$ is an injective map between $X_{\textup{b}}$ and $Y_{\textup{b}}$. Surjectivity between these spaces, moreover, follows from the same argument as in \cite[Lemma A.3]{chen2018existence}. Finally, using \cite[Corollary A.11]{wheeler2013large}, we obtain the invertibility of $\mathcal{F}_w(0,F)$ between the spaces $X$ and $Y$.
\end{proof}

\begin{lemma}\label{Fedholm of index zero of the linearized opt}
For all $(w,F)\in \mathscr{U}$, $\mathcal{F}_w(w,F)$ is Fredholm index $0$ as a map $X \to Y$.
\end{lemma}
\begin{proof}
Fix $(w,F)\in \mathscr{U}$. Since $w \in X$ then the coefficients of the operator $\mathcal{F}_w(w,F)$ go to the coefficients of $\mathcal{F}_w(0,F)$ as $|q|\to \infty$. By Lemma~\ref{Invertibility of $F_w(0,F)$}, we know that $\mathcal{F}_w(0,F)$ is invertible as a map $X$ to $Y$. The proof then follows from an application of \cite[Lemma A.12 and A.13]{wheeler2015solitary}.
\end{proof}

\section{{Qualitative properties}}\label{Qualitative}
\subsection{Bounds on the Froude number} In this section, we derive both upper and lower bounds on the Froude number in our heterogenous regime. The analysis follows closely the arguments from \cite[Section 4]{chen2018existence}. These bounds play a crucial role in the global theory discussed in Section \ref{largeamplitude}. In particular, they allow us to infer that blowup in norm implies stagnation. 

There has been a number of significant applied works devoted to estimating the Froude number of irrotational solitary waves. In this regime, Starr \cite{starr1947momentum} gave a formal proof of Froude number bounds written in terms of an integral of the free surface profile. Numerically, Longuet-Higgins--Fenton \cite{Long1974massmomentum} obtained the upper and lower bounds on the Froude number, $1<F<1.286$. However, for rotational solitary waves, much less is known. In addition to the results in \cite{chen2018existence}, we also mention the work of Wheeler \cite{wheeler2015froude} where a number of bounds on Froude number in the rotational (but constant density) case are obtained.

\subsubsection*{Lower bound} We begin by showing that every wave with critical Froude number must be trivial. In the global continuation argument, this allows us to conclude that all waves on the solution curve are supercritical.  The first step is to establish an integral identity.

\begin{lemma}\label{convergence}
Let 
\begin{equation} \label{qual w regularity}
(w,F)\in C_{\textup{b},\textup{e}}^2(\overline{R^+})\cap C_{\textup{b},\textup{e}}^2(\overline{R^-})  \cap C_0^1(\overline{R^+})\cap C_{0}^1(\overline{R^-}) \cap C^0_0(\overline{R}) \times \mathbb{R}
\end{equation}
solve equation \eqref{quasiw}. For $M \geq 0$, define
\begin{equation*}
    \textbf{I}(M):=\int_{-M}^{M}\int_{-1}^{0}\dfrac{H_p^3w_q^2+(H_p+2h_p)w_p^2}{2h_p^2H_p^3}\Psi_p \,dp \, dq +A\left(\dfrac{1}{F^2}\right)\int_{-M}^{M} \eta \,dx,
\end{equation*}

where $A$ is given in \eqref{A} and $\Psi = \Psi(p;1/F^2)$ is from \eqref{newfunction}.
Then $\textbf{I}\rightarrow{0}$ as $M \rightarrow{\infty}.$
\end{lemma}

\begin{proof}
Multiplying the height equation \eqref{quasi} by $\Psi$ then integrating by parts over the finite rectangle $|q| < M$, we obtain 

\begin{equation*}
    \begin{split}
    0=&\int_{-M}^{M}\int_{-1}^{0} \left[\left(\dfrac{1+h_q^2}{2h_p^2}-\dfrac{1}{2H_p^2}\right)\Psi_p+\dfrac{\Psi_p}{H_p^3}(h_p-H_p)\right]\, \,dp \, dq\\&-\int_{-M}^{M}\left.\left[\jump{ -\dfrac{1+h_q^2}{2h_p^2}+\dfrac{1}{2H_p^2}} \Psi -\jump{ \dfrac{\Psi_p}{H_p^3}} (h-H)\right]\right|_{I}\,dq\\&+\int_{-M}^{M}\left.\left[\left(-\dfrac{1+h_q^2}{2h_p^2}+\dfrac{1}{2H_p^2}\right)\Psi-\dfrac{\Psi_p}{H_p^3}(h-H)\right]\right|_{T}\,dq+\int_{-1}^{0}\left.\dfrac{h_q}{h_p}\Psi\right|_{q=-M}^{q=M} \,dp.
\end{split}
\end{equation*}
Here we have used the equation satisfied by $\Psi$ \eqref{IVP} to eliminate several terms. Notice that we can re-write the first integrand above as follows
\[\left(\dfrac{1+h_q^2}{2h_p^2}-\dfrac{1}{2H_p^2}\right)\Psi_p+\dfrac{\Psi_p}{H_p^3}(h_p-H_p)=\dfrac{H_p^3w_q^2+(H_p+2h_p)w_p^2}{2h_p^2H_p^3}\Psi_p.
\]
Combining this fact with the condition on $T$ and $I$ from \eqref{quasi} and IVP \eqref{IVP} yields
\begin{equation}\label{simplified}
\begin{split}
      \int_{-M}^{M}\int_{-1}^{0} \dfrac{H_p^3w_q^2+(H_p+2h_p)w_p^2}{2h_p^2H_p^3}&\Psi_p \, \,dp \, dq+\int_{-M}^{M}\jump{\dfrac{1+h_q^2}{2h_p^2}-\dfrac{1}{2H_p^2} } \Psi +\jump{\dfrac{\Psi_p}{H_p^3}}(h-H) \, dq\\& + A\left(\dfrac{1}{F^2}\right) \int_{M}^{M} \eta \,dx=-\int_{-1}^{0}\left.\dfrac{h_q}{h_p}\Psi\right|_{q=-M}^{M} \,dp.
\end{split}
\end{equation}
Observe that from \eqref{IVP},
the second integral on the left hand side of \eqref{simplified} vanishes. Now, sending $M \to \infty$ results in $h_q|_{-M}^M \rightarrow{0}$,  which forces the right hand side of \eqref{simplified} to vanish. Hence, the proof is complete.
\end{proof}\qedhere

\begin{theorem}[Critical waves are laminar] \label{thm:trivialsol}
 Let $(w,F)$ be in the function space defined in \eqref{qual w regularity} and solve \eqref{quasiw}.
 
 \begin{enumerate}[label=\normalfont{(\alph*)}]
    \item\label{A<0} Suppose $w>0$ on $T$, and $F$ is chosen such that $\Psi_p \geq 0$ for $[-1, \hat{p})$ and  $(\hat{p},0]$, then $A(1/F^2)<0$.
    
    \item\label{nonontrivialsolution} If $F=F_{\textup{cr}},$ then $w\equiv0$.
\end{enumerate}
\end{theorem}

\begin{proof}
The assumptions on the strict positivity of $w$ and $\Psi_p$ force the first integrand in the definition of $\textit{\textbf{I}}$ to be positive. Sending $M \rightarrow{\infty}$ and applying Lemma~\ref{convergence}, we therefore prove part \ref{A<0}. 

It remains to prove part \ref{nonontrivialsolution}. Thanks to Lemma~\ref{zerotheigenfunction}\ref{A<0}, we know that $A(1/F_{\textup{cr}}^2)=0$. For $F=F_{\textup{cr}}$, $\textit{\textbf{I}}$ takes the following form
\begin{equation}\label{intergralasfunctionofM}
 \int_{-M}^{M}\left[\int_{-1}^{0}\dfrac{H_p^3w_q^2+(H_p+2h_p)w_p^2}{2h_p^2H_p^3}(\Psi_{\textup{cr}})_p \,dp\right]\,dq, 
\end{equation} 
which still goes to zero as $M\rightarrow{\infty}$. As a result of Lemma~\ref{zerotheigenfunction}\ref{nonontrivialsolution}, it is clear that $(\Psi_{\textup{cr}})_p>0 \textrm{ for } -1 \leq p \leq \hat{p}$ and $\hat{p} \leq p \leq 0.$ Observe, the integral in \eqref{intergralasfunctionofM} is non-negative and non-decreasing as a function of $M$. Hence, the integrand must vanish for all $M$. This implies that $w_p \textrm{ and } w_q$ should be equal to zero everywhere. In other words, $w \equiv 0.$
\end{proof}

\subsubsection*{Upper bound} Now, we shall derive the upper bound of the Froude number. The argument here is based on \cite{chen2018existence} which is strongly inspired by the work of Pritchard--Keady  \cite{pritchard1974bounds} and Starr \cite{starr1947momentum}. Both earlier results show that for homogeneous irrotational fluid, $F<\sqrt{2}$. However, due to stratification and vorticity, our estimate here is presented in terms of several quantities associated to the underlying current and a bound away from stagnation along the cresline.  The integral identity \eqref{integralidentityforfroude} can, in fact, be applied to the homogeneous irrotational regime where it recovers the bound in \cite{starr1947momentum} and \cite{pritchard1974bounds}.

\begin{lemma}
Let $(w,F)$ belong to the space defined in \eqref{qual w regularity} and solve \eqref{quasiw}, then
\begin{equation}\label{integralidentityforfroude}
\begin{split}
    \dfrac{1}{F^2}&\left[\int_{-1}^{\hat{p}}|\rho_p|w(0,p)^2 \,dp+\int_{\hat{p}}^{0}|\rho_p|w(0,p)^2 \,dp +\rho(0)\eta(0)^2-\jump{\rho} w(\hat{p})^2\right]\\&=\int_{-1}^{0}\dfrac{w_p^2}{H_p^2h_p}(0,p)\,dp.
\end{split}
\end{equation}
\end{lemma}
\begin{proof}
 Recall from the discussion in Section~\ref{flow force section} that the flow force $\mathscr{S}$ is independent of $q$ and in semi-Lagrangian variables can be viewed as the functional \eqref{flow force functional} acting on $h$. Hence, by evaluating it at $q=0$ and $q=\pm\infty$, we obtain
\[
\begin{split}
  \mathscr{S}(h)=&\int_{-1}^{0}\left(\dfrac{1}{2h_p^2}+\dfrac{1}{2H_p^2}-\dfrac{1}{F^2}\rho(h-H)-\dfrac{1}{F^2}\int_{0}^{p} \rho H_p \,dp'\right)h_p\,dp\\&=\int_{-1}^{0}\left(\dfrac{1}{2H_p^2}+\dfrac{1}{2H_p^2}-\dfrac{1}{F^2}\int_{0}^{p} \rho H_p \,dp'\right)H_p\,dp=\mathscr{S}(H).
\end{split} 
\]
Gathering like terms, integrating by parts and simplifying terms lead to
\begin{equation}\label{IBPflowforce}
\begin{split}
 0=&\int_{-1}^{0} \left.\dfrac{w_p
^2}{2H_p^2h_p}\right|_{q=0}\,dp-\int_{-1}^{0}\dfrac{1}{F^2}\rho \left.w w_p\right|_{q=0}\,dp.
\end{split}
\end{equation}

Moreover, integrating by parts, we obtain
\begin{equation*}
    \begin{aligned}
        \int_{-1}^{0}\dfrac{1}{F^2}\rho \left.w w_p\right|_{q=0}\,dp=&\int_{-1}^{\hat{p}}\dfrac{1}{2F^2}|\rho_p| w(0,p)^2\,dp+\int_{\hat{p}}^{0}\dfrac{1}{2F^2}|\rho_p| w(0,p)^2\,dp\\&-\dfrac{1}{2F^2}\jump{\rho} w(0,p)^2+\left.\dfrac{1}{2F^2}\rho w^2\right|_{p=0}.
    \end{aligned}
\end{equation*}

Combining the previous two integrals with the one in \eqref{IBPflowforce}, we arrive at equation \eqref{integralidentityforfroude}. \qedhere
\end{proof}

\begin{theorem}[Upper bound of $F$]\label{thm:upperboundonF}
Let $(w,F)$ belong to the space defined in \eqref{qual w regularity} and solve \eqref{quasiw}. Then the Froude number F is bounded above:
\begin{equation}\label{froudeupperbound}
   F^2 \leq 2\norm{\rho }_{L^{\infty}} \norm{ H_p}_{L^{\infty}}^2 \norm{ h_p(0,\cdot)}_{L^{\infty}}.
\end{equation}
\end{theorem}
 \begin{proof}
From the Poincar\'e inequality, it is easy to see that
\[
\int_{-1}^{0}\dfrac{2}{F^2}\rho w \left.w_p \right|_{q=0}\,dp \leq \dfrac{2}{F^2}\norm{\rho }_{L^{\infty}} \norm{ w_p(0,\cdot)}_{L^2}^2.
\]
Moreover,
\[
\int_{-1}^{0} \left.\dfrac{w_p
^2}{H_p^2h_p}\right|_{q=0} \,dp \geq \left(\min_{p}(H_p^{-1})^2\right)\left(\min_{p}(h_p^{-1})(0,p)\right)\norm{ w_p(0,\cdot)}_{L^2}^2.
\]
Combining both estimates with the identity \eqref{IBPflowforce} and canceling some terms yields
\[
    \dfrac{2}{F^2}\norm{ \rho }_{L^{\infty}}\geq\left(\min_{p}(H_p^{-1})^2 \right)\left(\min_{p}(h_p^{-1})(0,p)\right),
\] which is equivalent to \eqref{froudeupperbound}.
\end{proof}
\subsection{Symmetry}
In this section, we will prove Theorem~\ref{symmetry} on the symmetry of supercritical solitary waves of elevation. The main machinery used here is the method of moving planes, first introduced by Alexandrov \cite{Alexandrov1962characteristic} in his study of spheres. Variations of this argument have been used by many authors, for instance, Serrin \cite{Serrin1971symmetry} in dealing with a symmetry problem in potential theory (see also \cite{nirenberg1988monotonicity, nirenberg1991movingplane}). Due to the full nonlinearity of the problem and the unboundedness of the domain, we adopt the version used in \cite{Li1991monotonicity}. The proof is patterned on that of \cite[Theorem 4.13] {chen2018existence}, which in turn is based partially on the work of Maia in \cite{Maia1997symmetryofinternalwaves}.

The main result, stated now in semi-Lagrangian variables for convenience, is as follows.

\begin{theorem}[Symmetry]\label{thm:symmetry} 
Let 
\[ (w,F)\in  C^{3}_{\textup{b}} (\overline{R^+}) \cap  C^{3}_{\textup{b}} (\overline{R^-}) \cap C^{0}(\overline{R}) \times \mathbb{R}\]
be a solution of the height equation \eqref{quasiw} that is a wave elevation 
\begin{equation}\label{elevation}
 w >0 \quad \textrm{in} \;R\cup I \cup T, 
\end{equation}
supercritical, and satisfies the upstream (or downstream) condition
\begin{equation}\label{decay}
 w, Dw \rightarrow{0}\quad \textrm{uniformly } as \quad q \rightarrow{ - \infty}\; \textrm{(or\; $\infty$)}. 
\end{equation}
Then, possibly after translation in $q$, $w$ is a symmetric and monotone solitary wave: there exists $q_*\in \mathbb{R}$ such that $q \mapsto w(q,\cdot)$ is even about $\{q=q_*\}$ and
\begin{equation}\label{axis}
 \pm w_q >0\textrm{ for }\pm (q_* -q )>0,\; -1 < p \leq 0.
\end{equation}
\end{theorem} 
Following the usual moving planes approach, we start by considering the reflected height function $h$ about the axis $q= \lambda,$
\begin{equation*} h^{\lambda}(q,p):=h(2\lambda-q,p).\end{equation*} 
Letting $v^{\lambda} := h^{\lambda}-h$, we see that $q=q_*$ is the axis of symmetry if and only if $v^{q_*}\equiv 0$. For $\lambda \in \mathbb{R}$, we define the sets 
\[ R_\lambda^\pm := R^\pm \cap \{ q < \lambda \},\]
and likewise for the boundary components $I_\lambda$, $B_\lambda$, and $T_\lambda$.  Throughout the section, we denote the restriction of $v^\lambda$ and $h^\lambda$ to $R_\lambda^\pm$ by $v^{\lambda\pm}$ and $h^{\lambda\pm}$, respectively.

Suppose that $h$ is a solution to the height equation \eqref{quasi}. Then for each $\lambda$, $v^{\lambda}$ solves the PDE
\begin{equation}\label{operators}
    \mathscr{L}v^{\lambda}=0 \quad \text{in} \; R_{\lambda}, \qquad \mathscr{B}v^{\lambda}=0 \quad \text{on} \; T_{\lambda} \qquad \mathscr{T}v^{\lambda}=0 \quad \text{on} \; I_{\lambda}
\end{equation}
where $\mathscr{L},\mathscr{B}, \; \text{and} \; \mathscr{T}$ are given as follows:
\begin{equation}\label{operatorL,B}
\begin{aligned}
    \mathscr{L}:=&\dfrac{1}{h^{\lambda}_p} \partial^{2}_q -\dfrac{2h^{\lambda}_q}{(h^{\lambda}_p)^2} \partial_q \partial_p + \dfrac{1+(h^{\lambda}_q)^2}{(h^{\lambda}_p)^3}\partial^{2}_p +\dfrac{h^{\lambda}_{pp}(h^{\lambda}_{q}+h_q)-2h^{\lambda}_{p}h_{pq}}{(h^{\lambda}_{p})^3}\partial_q\\&+\dfrac{h_{qq}\left(h^{\lambda}_p+h_p\right)-2h^{\lambda}h_{pq}+\left(\beta(-p)-F^{-2}\rho_ph\right)\left[(h^{\lambda}_p)^2+h^{\lambda}_ph_p+h_p^2\right]}{(h^{\lambda}_p)^3}\partial_p-\dfrac{1}{F^2}\rho_p, 
    \\ \mathscr{B}:=&\dfrac{h^{\lambda}_q+h_q}{2(h_p)^2}\partial_q -\dfrac{(h^{\lambda}_p+h_p)\left(1+(h^{\lambda}_q)^2\right)}{2(h^{\lambda}_p)^2 (h_p)^2} \partial_p +\dfrac{1}{F^2}\rho,
    \\\mathscr{T}:=&\jump{ \dfrac{h^{\lambda}_q+h_q}{2(h_p)^2}} \partial_q -\jump{ \dfrac{(h^{\lambda}_p+h_p)\left(1+(h^{\lambda}_q)^2\right)}{2(h^{\lambda}_p)^2 (h_p)^2} \partial_p } +\dfrac{1}{F^2}\jump{ \rho }. 
\end{aligned}
\end{equation} For a detailed derivation of the first two operators above and the ellipticity of $\mathscr{L}$, see \cite{Walsh2009Somecriteria}. The expression for the transmission operator $\mathscr{T}$ is new but follows from a similar calculation.  

The signs of the zeroth-order coefficients above will not allow us to apply the maximum principle directly. However, for supercritical waves, we can tackle this problem as follows. For $0<\epsilon \ll 1$, let $\tilde{\Psi}=\tilde{\Psi}(p;F,\epsilon)$ be the solution of the ODE
\begin{equation} \label{ODE}
\left ( \dfrac{\tilde{\Psi}_p}{H^{3}_p} \right)_p -\dfrac{1}{F^2}(\rho_p -\epsilon)\tilde{\Psi}=0 \qquad \textrm{on } (-1,0) 
\end{equation}
in the distributional sense with initial conditions
\[ \tilde\Psi(-1) = 0, \qquad \tilde\Psi(-1) = \epsilon.\]
Note that this implies a transmission condition on $p = \hat{p}$.  A straightforward adaptation of Lemma~\ref{Psitilde} gives the following result. 

\begin{lemma}\label{lemma1}
Suppose $F>F_{\textup{cr}}$, then for $\epsilon>0$ sufficiently small the solution $\tilde\Psi$ to \eqref{ODE} satisfies
\[\tilde{\Psi} >\epsilon \textrm{ for } -1< p\leq 0,\quad \tilde{\Psi}_p>0  \textrm{ for } [-1, \hat{p})\; and \; (\hat{p},0],\]
and
\begin{equation} \label{Psi tilde boundary sign}
-\dfrac{\tilde{\Psi}_p}{H^{3}_p}+\dfrac{1}{F^2}\rho \tilde{\Psi}<0\quad \text{ on} \quad p=0 \qquad
\jump{\tilde{\Psi}_pH^{-3}_p}-\dfrac{1}{F^2} \jump{ \rho } \tilde{\Psi} < 0 \quad  \text{on} \quad p=\hat{p}.
\end{equation}
\end{lemma}
\begin{proof}
The proof of this lemma is identical to the one of Lemma~\ref{Psitilde}.
\end{proof}

With this in hand, we can begin the moving planes method.  The first step is to show that $v^\lambda$ is sign definite on $R_\lambda$ when $\lambda$ is sufficiently large and negative.

\begin{lemma}\label{lemma2}
Under the hypothesis of Theorem~\ref{thm:symmetry}, there exists $K > 0$ such that 

\begin{equation}\label{vlambda}
v^{\lambda} \geq 0 \textrm{ on }R_{\lambda}\quad\textrm{ for all }\lambda<-K 
\end{equation} 
and 
\begin{equation}\label{hq}
h_q \geq 0\textrm{ in }R_{\lambda}\quad \textrm{ for all } \lambda < -K.
\end{equation}
\end{lemma}
\begin{proof}
Let $\tilde{\Psi}$ be defined as in equation \eqref{ODE}. By Lemma~\ref{lemma1}, for $0 < \epsilon \ll 1$, we have that $\tilde{\Psi} > \epsilon$. This allows us to define $\tilde\Psi u^{\lambda} := v^{\lambda}$. One can check that $u^\lambda$ solves the PDE
\begin{equation}\label{operatorstilde}
    \mathscr{\tilde{L}}u^{\lambda}=0 \quad \text{in} \; R_{\lambda}, \quad \mathscr{\tilde{B}}u^{\lambda}=0 \quad \text{on} \; T_{\lambda}, \quad \mathscr{\tilde{T}}u^{\lambda}=0 \quad \text{on} \; I_{\lambda}, \quad u^{\lambda}=0 \quad \text{on} \; B_{\lambda},
\end{equation} 
where
\begin{equation*}
\begin{aligned}
     \mathscr{\tilde{L}}u^{\lambda}&:=\tilde{\Psi} (\mathscr{L}u^{\lambda} +\dfrac{1}{F^2}\rho_p u^{\lambda})+\left(\dfrac{2(1+(h^{\lambda}_q)^{2})}{(h^{\lambda})^3}\tilde{\Psi}_p \right)u^{\lambda}_p-\left(\dfrac{2h^{\lambda}_q}{(h^{\lambda}_p)^2}\tilde{\Psi}_p \right)u^{\lambda}_q+Z u^{\lambda},\\
     \mathscr{\tilde{B}}u^{\lambda}&:=\tilde{\Psi}\mathscr{B}u^{\lambda}+(\mathscr{B}_{\textup{p}}\tilde{\Psi})u^{\lambda},\\
     \mathscr{\tilde{T}}u^{\lambda}&:=\tilde{\Psi}\mathscr{T}u^{\lambda}+(\mathscr{T}_{\textup{p}} \tilde{\Psi})u^{\lambda}.
\end{aligned}
\end{equation*}
Here, the zeroth-order coefficient in $\mathscr{\tilde{L}}$ is given by
\[ Z:=\dfrac{1+(h^{\lambda}_q)^{2}}{(h^{\lambda}_p)^3}\tilde{\Psi}_{pp}-\dfrac{1}{F^2}\rho_p \tilde{\Psi} +\left( \beta(-p)-\dfrac{1}{F^2}\rho_p h \right)\left((h^{\lambda}_p)^2+h^{\lambda}_p h_p+(h_p)^2\right)\dfrac{\tilde{\Psi}_p}{(h^{\lambda}_p)^3},\]
and the principal parts of the boundary operators are
\begin{align*} \mathscr{B}_{\textup{p}} &:=\dfrac{h^{\lambda}_q+h_q}{2h^{2}_p}\partial_q-\dfrac{(h^{\lambda}_p+h_p)(1+(h^{\lambda}_q)^2)}{2h^{2}_p(h^{\lambda}_p)^2}\partial_p, \\
\mathscr{T}_{\textup{p}}&:=\jump{\dfrac{h^{\lambda}_q+h_q}{2h^{2}_p}} \partial_q-\jump{ \dfrac{(h^{\lambda}_p+h_p)(1+(h^{\lambda}_q)^2)}{2h^{2}_p(h^{\lambda}_p)^2} } \partial_p.\end{align*}

We will show that there exists $K>0$ such that $u^{\lambda}>0$ in $R_{\lambda}$ for all $\lambda \leq -K$ which in turn proves \eqref{vlambda}. For the sake of contradiction, assume that for any $K$, there exists some $\lambda_0 \leq -K$ such that $u^{\lambda_0}$ takes a negative value in $R_{\lambda_0}$. By assumption, we know that $h$ is a wave of elevation, the same is true for $h^{\lambda}$ for any $\lambda$. Clearly by definition, $u^{\lambda}=0$ on $q=\lambda$, and  
\begin{equation*}
    u^{\lambda}=\dfrac{h^{\lambda}-h}{\tilde{\Psi}} > \dfrac{H-h}{\tilde{\Psi}}
\end{equation*}
where we note that the right-hand side of the inequality vanishes in the limit $q \to -\infty$.   Therefore, if $u^{\lambda_0}$ is negative in $R_{\lambda_0}$, then there must be a point $(q_0,p_0) \in R_{\lambda_0} \cup T_{\lambda_0} \cup I_{\lambda_0}$ such that
\begin{equation*}
  u^{\lambda_0}(q_0,p_0)= \inf_{R_{\lambda_0}} u^{\lambda_0}<0.
\end{equation*}
The cases when $(q_0,p_0)\in  R_{\lambda_0}$ and $(q_0,p_0)\in  T_{\lambda_0}$ are worked out in \cite[Lemma 4.18]{chen2018existence}. The only new possibility for the two-layer setting is the that $(q_0,p_0)\in  I_{\lambda_0}$.  Suppose that this is the case.  Applying the Hopf lemma to $R^{+}_{\lambda_0}$ and $R^{-}_{\lambda_0}$, we have the following:
\begin{equation} \label{Hopf2}
u^{\lambda_0+}_p(q_0,\hat{p})>0, \quad u^{\lambda_0-}_p(q_0,\hat{p})<0.
\end{equation}The minimum point assumption implies that $u^{\lambda_0}_q(q_0,\hat{p})=0$, which is equivalent to saying
\begin{equation}\label{h}
h^{\lambda_0}_q(q_0,\hat{p})=h_q(q_0,\hat{p}).
\end{equation}
From the uniform decay stated in \eqref{decay}, for any $\delta > 0$, we can choose a large enough $K$ such that we have
\begin{equation} \label{estimate1}
\begin{aligned}
   &|h(q_0,\hat{p})-H(\hat{p})|< \delta, \quad \left|\jump{h_p(q_0,\hat{p})-H_p(\hat{p})}\right|<\delta,\quad \\& \left|\jump{h^{\lambda_0}_p(q_0,\hat{p})-H_p(\hat{p})}\right|<\delta,\quad |h^{\lambda_0}(q_0,\hat{p})-H(\hat{p})|<\delta.  
\end{aligned}
\end{equation} Applying triangle inequalities to \eqref{estimate1} yields
\begin{equation}\label{estimatehp}
    \left|\jump{h^{\lambda_0}_p(q_0,\hat{p})-h_p(\hat{p})}\right|<C\delta.
\end{equation}
Using inequalities in \eqref{estimate1} and \eqref{estimatehp} allows us to write $\mathscr{\tilde{T}}u^{\lambda_0}$ in the following way
\begin{equation} \label{B2}
\mathscr{\tilde{T}}u^{\lambda_0}:= \jump{ -\dfrac{(h^{\lambda_0}_p +h_p)(1+(h^{\lambda_0}_q)^{2})}{2h^{2}_p (h^{\lambda}_p)^2}u^{\lambda_0}_p } \tilde{\Psi} + \left(- \jump{ \dfrac{1}{H^{3}_p} \tilde{\Psi}_p } +\dfrac{1}{F^2} \jump{ \rho } \tilde{\Psi} + \mathcal{O}(\delta)\right) u^{\lambda_0}=0.
\end{equation}
Observe that via \eqref{Psi tilde boundary sign}, we know that the coefficient $u^{\lambda_0}$ is positive. On the other hand $u^{\lambda_0}(q_0,\hat{p})<0$. Hence, the second term on the right hand side is negative. However, \eqref{Hopf2} shows that the first term on the right hand side on \eqref{B2} is negative. These, thereby, lead to a contradiction. We conclude that there exists a large enough $K>0$ such that \eqref{vlambda} holds. 

To prove \eqref{hq}, we first observe that $v^{\lambda}(\lambda,p)=0$ by construction.  In light of \eqref{vlambda}, this means it attains its global minimum on $R_\lambda$  there, hence $v_q^\lambda(\lambda, p) \geq 0$ for all $\lambda < -K$.   But recalling the definition of $v^\lambda$, we then have 
\[ 0 \geq -v^\lambda_q(\lambda,p) = 2h_q(\lambda,p) \qquad \textrm{for all } \lambda < -K. \qedhere \]
\end{proof}
\begin{proof}[Proof of Theorem~\ref{thm:symmetry}]
Define 
\begin{equation}\label{set}
\hat{\lambda}:=\text{sup}\{ \lambda_0 : v^{\lambda} >0 \; \text{in}\; R_{\lambda_0} \; \text{for all}\; \lambda <\lambda_0\}.
\end{equation} Note that $\hat{\lambda}$ is well-defined since the set above is non-empty.

\textbf{Case 1.} Let $\hat{\lambda}< \infty$. Under the continuity assumption on $h$, therefore $v^{\lambda}$, we can say that $v^{\hat{\lambda}} \geq 0$ in $ R_{\hat{\lambda}}$. Since, $v^{\hat{\lambda}}$ is in the kernel of the elliptic operators \eqref{operators} in $R_{\hat{\lambda}}$, applying maximum principle would guarantee that either $v^{\hat{\lambda}}>0$ or $v^{\hat{\lambda}} \equiv 0$ in $R_{\hat{\lambda}}$. By way of contradiction, suppose that the former holds. Since, $\hat{\lambda}$ is taken to be the supremum of all $\lambda_0$ defined in \eqref{set}, then there exists sequences $\{\lambda_l\}$ and $\{(q_l,p_l)\}$ in which $\lambda_l \searrow \hat{\lambda}$ together with $(q_l,p_l)\in \overline{R_{\lambda_l}}$ such that
\begin{equation*}
v^{\lambda_l}(q_l,p_l)=\inf_{R_{\lambda_l}}v^{\lambda_l}<0.
\end{equation*}

Since $v^{\lambda_l}=0$ on $B_{\lambda_l}$, via the maximum principle, we know $(q_l,p_l) \in T_{\lambda_l} \cup I_{\lambda_l}$. First, assume that $(q_l,p_l) \in T_{\lambda_l}$ which implies $v^{\lambda_l}_p(q_l,0)\leq 0 \text{ and } v^{\lambda_l}_q(q_l,0)=0$. 
Next, we need to show that $q_l$ is bounded below. For the sake of contradiction, assume that for all $l$ large enough, we have $q_l \leq -K$ where $K$ is the positive real number obtained in the previous lemma. Let us look at the following function $u^{\lambda_l}:=v^{\lambda_l}/\Psi$. By construction, $u^{\lambda_l}$ in $T_{\lambda_l}$ satisfies \eqref{operatorstilde}. Hence, we run into the same situation as case 2 in \cite[Lemma 4.18]{chen2018existence} where $\mathscr{\tilde{B}}u^{\lambda_l}(q_l,0)>0$. Hence, a contradiction. Therefore, $q_l$ is bounded below by $-K$ and certainly bounded above by $\lambda_l$. We can say that we have a convergence up to subsequence such that
\begin{equation*}
(q_l,0)\rightarrow{(\hat{q},0)}\in \overline{T_{\hat{\lambda}}} \quad \textrm{as} \; l \rightarrow{\infty},
\end{equation*} for some $\Hat{q} \in \lbrack -K,\hat{\lambda}\rbrack$. Now, since we assume that $v^{\hat{\lambda}}>0$ in $R^{\hat{\lambda}}$ then
\begin{equation*}
\lim_{l\rightarrow{\infty}} v^{\lambda_l}(q_l,0)=v^{\hat{\lambda}}(\Hat{q},0)=0. 
\end{equation*} Suppose that $\Hat{q}<\Hat{\lambda}$, then
\begin{equation}\label{vqv}
v^{\Hat{\lambda}}(\Hat{q},0)=v^{\hat{\lambda}}_q(\Hat{q},0)=0. 
\end{equation} 

By Hopf, $v^{\hat{\lambda}}_p(\Hat{q},0)<0$. Examining the operator $\mathscr{B}$ using these facts, we obtain that $\mathscr{B}v^{\hat{\lambda}}(\hat{q},0)>0$. But, since  $v^{\hat{\lambda}}$ is in the kernel of the operators in \eqref{operators}, then $\mathscr{B}v^{\hat{\lambda}}=0$. Hence, we have arrived at a contradiction.

On the other hand, suppose that $\Hat{q}=\Hat{\lambda}$ i.e. $(\hat{q},0)$ is the corner point of $R^{\hat{\lambda}}$. From \eqref{vqv}, it follows that $h^{\hat{\lambda}}_q(\Hat{\lambda},0)=0$. Moreover, the top boundary operator for $v^{\lambda}$ reads
\begin{equation}\label{rearrangeoperator}
 (h^{\lambda}_p)^{2}(h^{\lambda}_q+h_q)v^{\lambda}_q-(h^{\lambda}_p+h_p)(1+(h^{\lambda}_q)^2)v^{\lambda}_p+2\dfrac{1}{F^2}\rho h^{2}_p(h^{\lambda}_p)^2 v^{\lambda}=0. 
\end{equation} Letting $\lambda =\hat{\lambda}$, and taking the derivative of \eqref{rearrangeoperator} with respect to the $q$-variable, and computing the result at $(\hat{\lambda},0)$, we arrive at the following equality
\begin{equation}\label{simplifiedequality}
 2h_p(\hat{\lambda},0)v^{\hat{\lambda}}_{pq}(\hat{\lambda},0)=0,
\end{equation} where we have used the following facts
\begin{equation*}
h^{\hat{\lambda}}_q(\hat{\lambda},0)=-h_q(\hat{\lambda},0),\quad h^{\hat{\lambda}}_p(\hat{\lambda},0)=h_p(\hat{\lambda},0),\quad h^{\hat{\lambda}}_{qp}(\hat{\lambda},0)=-h_{qp}(\hat{\lambda},0).
\end{equation*}
Now, since $h_p>0,$ then from \eqref{simplifiedequality} we conclude that $v^{\hat{\lambda}}_{pq}(\hat{\lambda},0)=0$. 
Furthermore, since $v^{\hat{\lambda}}(\hat{\lambda},\cdot)=0$, then $v^{\hat{\lambda}}_p(\hat{\lambda},0)=v^{\hat{\lambda}}_{pp}(\hat{\lambda},0)=0$.

Additionally, by a simple calculation one also can show that $v^{\hat{\lambda}}_{qq}(\hat{\lambda},0)=0$. Hence, $v^{\hat{\lambda}}$ with all its derivatives up to order two vanish at the corner point $({\hat{\lambda}},0)$. By construction, we know $v^{\hat{\lambda}}$ is in the kernel of the elliptic operator in $R^{\hat{\lambda}}$. Thereby, it contradicts the Serrin edge point lemma which guarantees the strict signs on the first and second derivatives of $v^{\hat{\lambda}}$ .

 It remains to look at the case when $(q_l,p_l) \in I_{\lambda_l}$. Similar to \eqref{Hopf2}, applying the Hopf boundary lemma, we obtain the following inequalities:
\[
 v^{\lambda_l +}_p(q_l,\hat{p})>0, \quad v^{\lambda_l-}_p(q_l,\hat{p})<0.
\]In other words,
\begin{equation*}
 \jump{ v^{\lambda_l}_p(q_l,\hat{p})} >0  \quad \text{and} \quad v^{\lambda_l}_q(q_l,\hat{p})=0.  
\end{equation*}

Next, we need to show that $q_l$ is bounded below. For the sake of contradiction, assume that for all $l$ large enough, we have $q_l \leq -K$ where $K$ is the positive real number obtained in the previous lemma. Let us look at the following function $u^{\lambda_l}:=v^\lambda_l/ \Psi$. Hence, we run into the same situation as in \eqref{B2} that yields $\mathscr{\tilde{T}}u^{\lambda_l}<0$. However, by construction, $u^{\lambda_l}$ lies in the kernel of $\mathscr{\tilde{T}}$ in $I_{\lambda_l}$. Hence, we obtain a contradiction. Therefore, $q_l$ is bounded below by $-K$ and certainly bounded above by $\lambda_l$. Therefore, we have a convergence up to subsequence, in particular
\begin{equation*}
 (q_l,\hat{p})\rightarrow{(\hat{q},\hat{p})}\in \overline{I_{\lambda_l}} \quad \textrm{as} \; l \rightarrow{\infty},
\end{equation*} for some $\Hat{q} \in \lbrack -K,\hat{\lambda}\rbrack$. Now, since we assume that $v^{\hat{\lambda}}>0$ in $R^{\hat{\lambda}}$ then
\begin{equation*}
\lim_{l\rightarrow{\infty}} v^{\lambda_l}(q_l,\hat{p})=v^{\hat{\lambda}}(\Hat{q},\hat{p})=0.
\end{equation*} Suppose that $\Hat{q}<\Hat{\lambda}$, then
\begin{equation*}
 v^{\Hat{\lambda}}(\Hat{q},\hat{p})=v^{\hat{\lambda}}_q(\Hat{q},\hat{p})=0. 
\end{equation*}
By Hopf, $\jump{ v^{\hat{\lambda}}_p}(\Hat{q},\hat{p}) <0$. In view of the operator $\mathscr{T}$, we know that $\mathscr{T}v^{\hat{\lambda}}(\hat{q},\hat{p})<0$. But, since  $v^{\hat{\lambda}}$ solves, then $\mathscr{T}v^{\hat{\lambda}}=0$. Hence, we arrive at a contradiction.

Now, it remains to consider the case when $(\hat{q},\hat{p})=(\hat{\lambda},\hat{p})$. The argument done here is similar to the one in \eqref{rearrangeoperator}. Using the fact that $v^{\hat{\lambda}}_q(\hat{\lambda},\hat{p})=0$, $v^{\hat{\lambda}}_{qq}(\hat{\lambda},\hat{p})=0$, $v^{\hat{\lambda}}(\hat{\lambda},\hat{p})=0$, $h^{\hat{\lambda}+}_p(\hat{\lambda},\hat{p})=h^{+}_p(\hat{\lambda},\hat{p})$, and $h^{\hat{\lambda}-}_p(\hat{\lambda},\hat{p})=h^{-}_p(\hat{\lambda},\hat{p})$ along with operator $\mathscr{T}$, we arrive at the following equation 
\begin{equation}\label{rearrange2}
\begin{split}
\bigg ( 2(h^{\hat{\lambda}-}_p)^{2}& (h^{-}_p)^{2}(h^{\hat{\lambda}+}_p +h^{+}_p)(1+(h^{\hat{\lambda}}_q)^2)v^{\hat{\lambda}+}_{pq} \bigg)\\&-\bigg (2(h^{\hat{\lambda}+}_p)^{2} (h^{+}_p)^{2}(h^{\hat{\lambda}-}_p +h^{-}_p )(1+(h^{\hat{\lambda}}_q)^2)v^{\hat{\lambda}-}_{pq}\bigg)=0.
\end{split}
\end{equation} Doing some computation on \eqref{rearrange2} gives us
\begin{equation}\label{jump}
 \jump{v^{\hat{\lambda}}_{pq}h^{-3}_p}(\hat{\lambda},\hat{p})=0.
\end{equation}
Now, if we look at the region $R_{\hat{\lambda}}^+$ and  $R_{\hat{\lambda}}^-$ separately, we can apply the Serrin edge point lemma. 

But, since we know that $v^{\hat{\lambda}}_q(\Hat{q},\hat{p})=0$ and $v^{\hat{\lambda}}_p(\Hat{q},\hat{p})=0$, then we can rule out the first conclusion of the Serrin edge point lemma that states the strict sign on the derivative of $v^{\hat{\lambda}}$ in outward direction. Hence, the second derivative of $v^{\hat{\lambda}}$ in outward direction has a strict sign (see Theorem~\ref{Maximum Principle}\ref{Serrin}). Consider the following outward vectors associated to $R_{\hat{\lambda}}^+$ and  $R_{\hat{\lambda}}^-$  respectively:

\begin{equation}\label{vectors}
 \mathbf{t}:= \frac{1}{(h_p^+)^{3/2}} (1,-1), \quad \text{and} \quad \mathbf{s} := \frac{1}{(h_p^-)^{3/2}} (1,1).
\end{equation}
 Evaluating $\partial_{\mathbf{s}}^2 v^{\hat{\lambda}}$ and $\partial_{\mathbf{t}}^2 v^{\hat{\lambda}}$ at $(\hat{\lambda},\hat{p})$, the Serrin edge point lemma gives the inequalities
\begin{equation*}
 \dfrac{1}{(h^{-}_p)^3}v^{\hat{\lambda}-}_{pq} \Big|_{(\hat{\lambda},\hat{p})}>0 \quad \text{and}\quad \dfrac{1}{(h^{+}_p)^3}v^{\hat{\lambda}+}_{pq} \Big|_{(\hat{\lambda},\hat{p})}<0.
\end{equation*}
This implies that $\jump{ v^{\hat{\lambda}}_{pq}h^{-3}_p}(\hat{\lambda},\hat{p})<0$, which contradicts \eqref{jump}.
Therefore, it must be that $v^{\hat{\lambda}} \equiv 0$ in $R_{\hat{\lambda}}$, and hence $h$ is symmetric with respect to the axis $q=\hat{\lambda}$.

It is left to show the strict monotonicity of $h$. We know that for a fixed $\lambda< \hat{\lambda}$,  $v^{\lambda}>0$ in $R_{\lambda}$. Since $v^{\lambda}$ vanishes on $q=\lambda$ (right boundary of $R_{\lambda})$, then $v^{\lambda}$ attains its minimum there. Applying the Hopf boundary lemma yields
\begin{equation*}
  v^{\lambda}_q(\lambda,p) <0.
\end{equation*}
But, $v^{\lambda}_q(\lambda,p)=-2h_q(\lambda,p)$. Hence, we have
\begin{equation*}
  h_q(\lambda,p)>0, \quad \text{for all}\; \lambda< \hat{\lambda}\textrm{ in } [-1,0]. 
\end{equation*}

Next, we consider the case when $p=\hat{p}$. By continuity of $h_q$,  for fixed $\lambda < \hat{\lambda}$, we have that $h_q(\lambda,\hat{p}) \geq 0$. Suppose $h_q(\lambda,\hat{p})=0$. This implies that $v_q(\lambda,\hat{p})=0$. We also know the following facts : 
\begin{equation}\label{facts}
v^{\lambda}_{qq} =0, \quad v^{\lambda} =0, \quad h^{\lambda+}_{p} =h_{p}^+,\quad  h^{\lambda-}_{p}=h_{p}^+ \qquad \textrm{at } (\lambda, \hat{p}).
\end{equation} 
We now look at the operator $\mathscr{B}$ in \eqref{operatorL,B}. Differentiating $\mathscr{B} v^{\lambda}$ with respect to $q$, evaluating it at $(\lambda, \hat{p})$ and using the facts mentioned previously in \eqref{facts}, we obtain a similar equation to the one in \eqref{jump}:
\begin{equation}\label{jump2}
 \jump{ v^{\lambda}_{pq} h^{-3}_p}(\lambda,\hat{p})=0.
\end{equation}
Choosing two outward vectors associated to each $R_{\lambda}^+$ and  $R_{\lambda}^-$ as the one in \eqref{vectors} and using the conclusion of the Serrin edge point lemma, we therefore contradict \eqref{jump2}. Hence, $h_q(\lambda,\hat{p})>0$. Thus, we proved the strict monotonicity of $h$ at the internal interface.

On the top boundary $T^{\hat{\lambda}}$, we have $h_q \geq 0$ by continuity. We aim to show that $h_q$ has a strict sign there. By way of contradiction, suppose that there exists $\lambda < \hat{\lambda}$ such that $h_q(\lambda,0)=0$ as well. Differentiating the equation in \eqref{rearrangeoperator} with respect to $q$ and evaluating it at $(\lambda,0)$ utilizing the following identities:
\[
 h_q =-h^{\lambda}_q=0, \quad
 h_p=h^{\lambda}_p, \quad
 h_{qp}=-h^{\lambda}_{qp},\quad
 v^{\lambda}=v^{\lambda}_q=0 \qquad \textrm{at } (\lambda,0),
\]
we arrive at the equation
\begin{equation*}
 2h^{\lambda}_p(\lambda,0) v^{\lambda}_{qp}(\lambda,0)=0. 
\end{equation*}
Thus, the above equation and the no horizontal stagnation condition implies 
\begin{equation*}
 v^{\lambda}_{qp}(\lambda,0)=0. 
\end{equation*}
As before, we also have that $v^{\lambda}_p(\lambda,0)=v^{\lambda}_{pp}(\lambda,0)=v^{\lambda}_{qq}(\lambda,0)=0$. Therefore, via the Serrin edge point lemma, we arrive at a contradiction. We then conclude that $h_q>0$ on $T^{\hat{\lambda}}$. Hence, $h$ is monotone on $R_{\hat{\lambda}} \cup T_{\hat{\lambda}} \cup I_{\hat{\lambda}} $.

\textbf{Case 2.} Let $\hat{\lambda} =\infty.$ This just means that $v^{\lambda} \geq 0$ in $R_{\lambda}$ for all $\lambda$. Since $v^{\lambda}$ is in the kernel operators $\mathscr{L}$ in \eqref{operatorL,B}, therefore we can apply the maximum principle which guarantees that $v^{\lambda} > 0$ in $R_{\lambda}$ for all $\lambda$. Applying the same argument as above, we see that $h_q >0$ in $R_{\lambda} \cup T_{\lambda} \cup I_{\lambda} $. In other words, we now have that $h_q >0$ in $R$. Thus, this implies that $h$ is a monotone front. But this violates the nonexistence of monotone fronts (see Theorem~\ref{non existence of monotone fronts}). Hence, we can exclude case 2.
Thus, the proof is complete.
\end{proof}
\subsection{Asymptotic monotonicity and nodal properties}
The monotonicity property \eqref{axis} will eventually be crucial for the large-amplitude theory where it is used to prove pre-compactness of $\mathcal{F}^{-1}(0)$. The small-amplitude waves that will be constructed in Section~\ref{Small Amplitude} are waves of elevation, and hence monotone simply as a consequence of Theorem~\ref{thm:symmetry}. We will then need to show that this property holds along the global bifurcation curve.

However, the set of monotone functions is neither open nor closed in the topology we are working with. This is a common issue in global bifurcation theoretic studies of elliptic PDE.  To remedy it, we introduce additional sign conditions on the derivative of the solutions that are collectively called \textit{nodal properties}.  These conditions will in particular imply monotonicity, but are also open and closed in a relevant topology. Once we confirm that they are exhibited by the small-amplitude solutions, it immediately follows that they hold along a connected set extending the local curve. The main tool here is the maximum principle.  As in the previous subsection, we will take advantage of the translation invariant structure of the equation.

Let $q=0$ be the axis of even symmetry.  We start by dividing the region to the right of the crest into four sub-domains: the upper and lower rectangles are each split into a finite sub-rectangle and semi-infinite rectangle. Proving the nodal properties on the finite rectangle is done essentially as in the periodic case in  \cite[Section 5]{Walsh2009Stratified} and so can be omitted. We focus instead on the semi-infinite rectangles for which we adopt the adopt the idea of \cite[Section 2.2]{wheeler2015solitary}.  

To fix notation, we define the following regions:
\begin{equation}\label{qpositifregions}
    R^{\pm}_{>}:=R^{\pm} \cap \{q>0\}, \quad 
    L^+_{0}:=\{(0,p):\hat{p} <p \leq 0\}, \quad
    L^-_{0}:=\{(0,p):-1 \leq p < \hat{p} \},
\end{equation} 
with the boundary components $T_>, I_>, B_>$ given accordingly. We will use $R_{>}$ to denote $R^{+}_{>} \cup R^{-}_{>}$. In a similar way, $L_0:=L^+_{0} \cup L^-_{0}$.

The first step is to show that sufficiently small-amplitude solutions in the half-strip for which $w_q$ has a sign along the left boundary must be monotone throughout the half-strip.  This fact will allow us to infer monotonicity on the semi-infinite tail regions once monotonicity on the finite extent rectangles is known.

\begin{prop}[Asymptotic monotonicity]\label{prop:asymptoticmonotonicity}
There exists $\delta>0$ such that, if 
\[ (w,F) \in C^{3}_{\textup{b}}(R^{+}_{>})\cap C^{3}_{\textup{b}}(R^{-}_{>})\cap C^{0}_0(\overline{R_{>}}) \cap C^{1}_0 (\overline{R^{+}_{>}}) \cap C^{1}_0(\overline {R^{-}_{>}}) \times \mathbb{R} \]
is a solution of the height equation \eqref{quasiw} in $R_{>}$  for $F > \Fcr$, with $\norm{ w }_{C^2(R_{>})} < \delta$ and $\pm w_q \leq 0$ on  $\overline{L_{0}}$, then 
\begin{equation} \label{w_q}
\pm w_q <0 \;  \text{in}\; \overline{R_{>}}\setminus \left(B_> \cup \overline{L_{0}}\right).
\end{equation}
\end{prop}

\begin{proof}
Due to the translation invariance of the height equation \eqref{quasiw}, we know that  $v:=h_q$ is in the kernel of the linearized operator:
\begin{equation}\label{quasidiff}
  \begin{cases}
    \left(-\dfrac{h_q v_q}{h_p^2}+\dfrac{(1+h^{2}_q)v_p}{h_p^3}\right)_p+\left(\dfrac{v_q}{h_p}-\dfrac{h_q v_p}{h^{2}_p}\right)_q-\dfrac{1}{F^2}\rho_p v&=0 \qquad  \text{in}\; R_{>},\\
\dfrac{h_q v_q}{h_p^2}-\dfrac{(1+h^{2}_q)v_p}{h_p^3}+\dfrac{1}{F^2}\rho v &=0 \qquad  \text{on}\; T_>,\\
\jump{ \dfrac{h_q v_q}{h_p^2} }-(1+h^{2}_q)\jump{ \dfrac{v_p}{h_p^3}} +\dfrac{1}{F^2}\jump{ \rho } v &=0 \qquad  \text{on}\; I_>,\\
v&=0 \qquad  \text{on}\; B_>.\\
\end{cases}  
\end{equation}
Set $v:= \tilde\Psi u$ for $\tilde\Psi$ as in Lemma~\ref{lemma1}. Thus, we can rewrite each equation in \eqref{quasidiff} in terms of $u$. In particular, the interior equation becomes
\[
\begin{split}
    \left(\left(\dfrac{\tilde\Psi}{h_p^3}+\dfrac{\tilde\Psi h^{2}_q}{h_p^3}\right)u_p -\dfrac{\tilde\Psi h_q}{h^{2}_p}u_q\right)_p+\left(\dfrac{\tilde\Psi}{H_p}u_q -\dfrac{\tilde\Psi h_q}{h^{2}_p}u_p\right)_q  + \dfrac{(1+h^{2}_q)\tilde\Psi}{h^{3}_p} u_p -\dfrac{h_q \tilde\Psi _p}{h^{2}_p}u_q \\ +\left(\left(\dfrac{(1+h^{2}_q)\tilde\Psi_p}{h^{3}_p}\right)_p-\left(\dfrac{h_q \tilde\Psi_p}{h^{2}_p}\right)_q-\dfrac{1}{F^2}\rho_p\tilde\Psi \right)u =0.
    \end{split}
\]
Taking $\norm{ w }_{C^2(R_{>})}=\norm{h-H}_{C^2(R_{>})}$ to be small enough ensures the equation above represents a uniformly elliptic operator acting on $u$ on  $R_{>}$. From the equation \eqref{ODE} satisfied by $\tilde\Psi$, we see that the coefficient of $u$ on the last line is strictly negative for $0 < \delta \ll 1$, and therefore the maximum principle can be applied to infer that $u$ attains no non-negative maximum in the interior. 

We can do the same thing to the equation on the top boundary:
\[
    -\dfrac{(1+h^{2}_q) \tilde\Psi u_p}{h^{3}_p} +\dfrac{h_q \tilde\Psi u_q}{h^{2}_p}+\left(-\dfrac{\tilde\Psi_p}{H^{3}_p}+\dfrac{1}{F^2}\rho \tilde\Psi-\dfrac{h^{2}_q \tilde\Psi_p}{h^{3}_p}-\left(\dfrac{1}{h^{3}_p}-\dfrac{1}{H^{3}_p}\right) \tilde\Psi_p \right)u=0.
\]
From \eqref{Psi tilde boundary sign} and the smallness of $w$ in $C^2$, we can conclude that the coefficient of the zeroth-order term above is negative. Suppose then that $u$ attains its non-negative maximum on $T_{>}$. At that point, $u_q = 0$, while by the Hopf lemma $u_p > 0$, giving a contradiction.

Now, consider the equation on the internal interface $I_>$. By the same type of computation, we arrive at:
\begin{equation} \label{internalv}
    \jump{ \dfrac{h_q }{h_p^2} } \tilde\Psi u_q -(1+h^{2}_q)\jump{ \dfrac{u_p}{h_p^3}} \tilde\Psi +\left(-(1+h^{2}_q)\jump{ \dfrac{\tilde\Psi_p }{h_p^3}} +\dfrac{1}{F^2}\jump{ \rho } \tilde\Psi  \right)u  =0.
\end{equation}
Suppose that $u$ attains its maximum on the $I_{>}$. Thus $u_q$ vanishes there and by Hopf we know that $u^{+}_p>0$ and $u^{-}_p<0$. This implies $\jump{u_p h_p^{-3}} > 0$.  On the other hand, in view of \eqref{Psi tilde boundary sign}, the coefficient of the zeroth-order term in \eqref{internalv} is also negative for $\delta$ sufficiently small. Thus we have again produced a contradiction.

In total, this reasoning shows that $u$, and hence $w_q$, is strictly negative in $\overline{R_>} \setminus (B_> \cup \overline{L_0}$.  The proof of the proposition is therefore complete.
\end{proof}

Consider the following nodal properties:
\begin{subequations} \label{Nodal}
\begin{align}
    w_q&<0 \; \text {in} \; R^{+}_{>} \cup R^{-}_{>} \cup I_> \cup T_> , \label{nodal:wq}
     \\w_{qq}& <0 \; \text {on} \; \overline{L_{0}} \setminus \{(0,-1)\}, \label{nodal:wqq}
    \\ w_{qp} &<0 \; \text {on} \; B_> ,\label{nodal:wqp}
    \\ w_{qqp} &<0 \; \text {at} \; (0,-1). \label{nodal:wqqp}
    \end{align}
\end{subequations}
Shortly, we will prove that these define open and closed sets in an appropriate topology.  First, however, we present the following result that show one can deduce the full set of nodal properties from just \eqref{nodal:wq}.
    \begin{lemma}\label{nodalproperty}
     Let $(w,F)$ be a solution of the height equation \eqref{quasiw} in $R_>$, where $w \in C^{3}_{\textup{b,e}}(\overline{R^{+}_{>}})\cap C^{3}_{\textup{b,e}}(\overline{R^{-}_{>}}) \cap C^{1}_0(\overline{R^{+}_{>}}) \cap C^{1}_0(\overline{R^{-}_{>}})\cap C^{0}_0(\overline{R_{>}})$ satisfies \eqref{nodal:wq}.  Then $w$ satisfies all of the nodal properties \eqref{Nodal}.
    \end{lemma}
    \begin{proof}
 Observe that the statement of the lemma does not require $w$ to be ``small" nor that $F$ is supercritical. We shall prove each of the sign conditions in \eqref{Nodal} consecutively. To begin, \eqref{nodal:wq} is true by hypothesis.  
 
 Consider next \eqref{nodal:wqq}.  By \eqref{nodal:wq}, $w_q$ is non-positive in $R_>$ and by symmetry it vanishes on $L_0$.  Applying the Hopf lemma, we conclude that $w_{qq} < 0$ on $L_0$.  It remains to show that $w_{qq}<0$ holds at the points $(0,0)$ and $(0,\hat{p})$. The argument for the first point is the same as in \cite[Lemma 4.20]{chen2018existence}, so we focus on the second one which is new.  Via continuity of $w_{qq}$, we know that $w_{qq} \leq 0$ at $(0,\hat{p})$. Seeking a contradiction, suppose that $w_{qq}=0$. By evenness, 
\begin{equation}\label{hqhqp}
    w_q=0 \; \text{and} \; w_{qp}\equiv 0 \; \text{on} \; L_0.
\end{equation}
 Consider the internal equation:
 
\begin{equation}\label{heighteq}
     \jump{ \dfrac{1+h_q^2}{2h_p^2} }-\jump{ \dfrac{1}{2H_p^2}} +\dfrac{1}{F^2}\jump{ \rho } (h-H) =0 \qquad  \text{on}\; I.
\end{equation}
Taking the derivative of \eqref{heighteq} with respect to the $q$-variable twice, computing it at the point $(0,\hat{p})$ and using the fact that $h_{qq}=0$ combined with \eqref{hqhqp},
we obtain,
\begin{equation*}
    \jump{ h_{qqp}h_{p}^{-3}}(0,\hat{p}) =0.
\end{equation*}
Now, if we look at the region $R^{+}_{>}$ and  $R^{-}_{>}$ separately, we can apply the Serrin edge point lemma. But, since we know that $h_{qq}(0,\hat{p})=0$ and $h_{qp}(0,\hat{p})=0$, then we rule out the first possibility of the conclusion of the Serrin edge point lemma for each of the two regions. Hence, the later conclusion of the Serrin edge point lemma should hold. Consider the following two outward vectors associated to each $R^{+}$ and  $R^{-}$ respectively:
\begin{equation}\label{vectorss}
 \mathbf{t} := \dfrac{1}{(h^{+}_p)^{3}} \left(-1,-1 \right )\quad \text{and} \quad \mathbf{s} :=\dfrac{1}{(h^{-}_p)^{3}} \left (-1,1 \right).
\end{equation}
 Evaluating $\partial_{\mathbf{t}}^2 h_q$ and $\partial_{\mathbf{s}}^2 h_q$ at $(0,\hat{p})$, the Serrin lemma gives the inequalities
\begin{equation*}
 \dfrac{1}{(h^{+}_p)^3}h^{+}_{qqp}>0 \quad \text{and}\quad \dfrac{1}{(h^{-}_p)^3}h^{-}_{qqp}<0,
\end{equation*}
whence
\begin{equation*}
    \jump{h_{qqp} h_{p}^{-3}}(0,\hat{p}) >0.
\end{equation*} 
Thus we have produced a contradiction. Therefore, \eqref{nodal:wqq} holds.

 Next, \eqref{nodal:wqp} follows from the Hopf boundary lemma since $w_q$ vanishes identically along $B_>$.  It remains only to prove the inequality in \eqref{nodal:wqqp}. Recall, on $L_{0}^+$, $L_{0}^-$, and $B_{>}$ we know that $w_q\equiv0$. Therefore,
\[w_q=w_{qp}=w_{qq}=w_{qpp}=w_{qqq}=0 \quad \textrm{at } (0,-1).\] By the Serrin Edge point lemma, we conclude that $w_{qqp}<0 \textrm{ at } (0,-1)$ which then proves \eqref{nodal:wqqp}. 
\end{proof}

     \begin{lemma}[Open property]\label{Open property}
     Let $(w,F)$ and $(\tilde{w},\tilde{F})$ be supercritical solutions of the height equation \eqref{quasiw} on $R_>$ with
     \[ w, \, \tilde{w} \in C^{3}_{\textup{b,e}}(R^{+}_{>})\cap C^{3}_{\textup{b,e}}(R^{-}_{>})\cap C^{0}_0(\overline{R_{>}}) \cap C^{1}_0 (\overline{R^{+}_{>}}) \cap C^{1}_0(\overline {R^{-}_{>}}).\]
     If $w$ satisfies the nodal properties \eqref{Nodal}, then there exists $\epsilon=\epsilon(w)>0$ such that 
     \[\norm{ w- \tilde{w} }_{C^3(R)}+ \left|F-\tilde{F}\right|< \epsilon 
     \]
      implies $\tilde{w}$ also satisfies the nodal properties \eqref{Nodal}.
    \end{lemma}
    \begin{proof}
    Recall, to prove that $\tilde{w}$ satisfies \eqref{Nodal}, it is enough to show that it exhibits the monotonicity \eqref{nodal:wq}.     We start by dividing $R^{+}_{>}$ into two overlapping regions namely:
    \[
    R^{+}_{1,>}:=\{(q,p)\in R^{+}_{>}:q<2K\}\qquad R^{+}_{2,>}:=\{(q,p)\in R^{+}_{>}:q>K\}. 
    \]
    Likewise, the lower region $R^{-}_{>}$ is divided into $R^{-}_{1,>}$ and $R^{-}_{2,>}$. The top, bottom, internal and vertical boundaries of these rectangles are denoted $T_{i,>}, B_{i,>}, I_{i,>}, L^+_{i,0}, \textrm{ and } L^-_{i,0}$ for $i\in\{1,2\}$. 
    
    Let us first look at  $ R^{j}_{1,>} ,\textrm{ for } j\in \{+,-\}$. These two finite rectangles behave in the same way as in periodic case. Therefore, the proof could be done in the same way as in \cite[Lemma 5.1]{Constantin2004Exactsteady}. The basic idea is that, for any $K > 0$, there exists $\epsilon_K$ such that taking $0 < \epsilon < \epsilon_K$ ensures $\tilde w_q < 0$ in the interior.  One then uses a Taylor expansion of $w$ and the nodal properties to conclude the same holds up to the boundary.  
   
    On the other hand, since $w\in C^2_0(R)$, one can choose large enough $K$ such that 
\[\norm{w }_{R^{j}_{2,>}}<\delta/2,
\] where $\delta$ is given as in Proposition~\ref{prop:asymptoticmonotonicity}. Hence, setting $\epsilon:= \min\{\delta/2,\epsilon_K\}$, we have $\tilde{w}_q<0$ in $\overline{R_{1>}}\setminus \left(B_{1,>} \cup \overline{L_{1,0}}\right)$ which then implies $\tilde{w}_q \leq 0$ in $ L^j_{2,0}$. Again, applying Proposition~\ref{prop:asymptoticmonotonicity}, we infer that  $\tilde{w}_q < 0$ in $\overline{R_{2>}}\setminus \left(B_{2,>} \cup \overline{L_{2,0}}\right)$. Together with the previous paragraph, this shows $\tilde w$ satisfies \eqref{nodal:wq}.
    \end{proof}
  
     \begin{lemma}[Closed property]\label{Closed property}
     Let $\{(w_n,F_n)\} \subset \mathscr{U}$ be a sequence of solutions to the height equation \eqref{quasiw} on $R_>$. Suppose that there exists $(w,F) \in \mathscr{U} $ such that $(w_n,F_n) \rightarrow (w,F)$ in $C^{3}_{\textup{b}}(R^{+}_{>})\cap C^{3}_{\textup{b}}(R^{-}_{>})\cap C^{0}_0(\overline{R_{>}}) \cap C^{1}_0 (\overline{R^{+}_{>}}) \cap C^{1}_0(\overline {R^{-}_{>}}) \times \mathbb{R}$. If  each $w_n$ satisfies the nodal properties \eqref{Nodal}, then $w$ also satisfies the nodal properties \eqref{Nodal} unless $w \equiv 0$.
    \end{lemma}
    \begin{proof}
    Let $v:=w_q$. Because each $w_n$ satisfies the nodal properties, we may infer that $w_q \leq 0$ in $\overline{R_{>}}$. Moreover, $v$ solves of the uniformly elliptic PDE \eqref{quasidiff} in $R_{>}$. Furthermore, since $v=0$ in $\overline{L_{0}}$ and $B_>$, we can apply the maximum principle and conclude the following three possibilities; (i) $v<0$ in $\overline{R_{>}} \setminus \left (B_{>} \cup \overline{L_0} \right)$; or (ii) there exists some point $(q*,0)\in T_>$; or (iii) there exists some point $(q*,\hat{p})\in I_>$ such that $v(q*,0)=0$ or $v(q*,\hat{p})=0$. 
    
    The proof of the lemma if either the first and second possibilities occur follows from \cite[Lemma 4.22]{chen2018existence}. Therefore, we only consider the third possibility here. Assume that there exists $(q*,\hat{p})\in I_>$ such that $v(q*,\hat{p})=0$. By the transmission boundary condition in \eqref{quasidiff}, we obtain $\jump{v_p/h_p^3}(q*,\hat{p})=0$.
   But via Hopf boundary lemma, we know that $v_p(q*,\hat{p}^+)<0$ and $v_p(q*,\hat{p}^-)>0$. Hence, we arrive at a contradiction unless $v \equiv 0$
in $\overline{R_{>}}$ which is equivalent to saying $w\equiv 0$.
\end{proof}
    
\section{Small-amplitude existence theory}\label{Small Amplitude}
In this section, we will construct a curve of small-amplitude solutions to the height equation \eqref{quasiw} that bifurcates from the trivial solution $w = 0$ at the critical Froude number $\Fcr$ defined in \eqref{def Fcr}. With that in mind, we introduce the non-negative parameter $\epsilon^2:=\mu_{\textup{cr}}-\mu$, which will be positive for supercritical waves. The corresponding Froude number is thus $F^{\epsilon}:=\left(1/F_{\textup{cr}}^2-\epsilon \right)^{-1/2}$.  We will frequently abuse notation by writing $(w,\epsilon)$ rather than $(w,F)$.

\begin{theorem}[Small-amplitude waves] \label{smallamplitudethoery}
There exists $\epsilon_{*}>0$ and a continuous local curve 
\begin{equation}\label{localcurve}
        \mathscr{C}_{\textup{loc}}:=\{(w^{\epsilon},F^{\epsilon}):0 < \epsilon < \epsilon_{*} \} \mathscr{U} \subset \subset X \times \mathbb{R}
     \end{equation}
     of solutions to $\mathcal{F}(w,F)=0$ with the following properties
     \begin{enumerate}[label=\normalfont{(\alph*)}]
    \item \label{continuity}\textup{(Continuity)} The mapping $\epsilon \mapsto w^{\epsilon}$ is continuous from $(0,\epsilon_{*})$ to X, with $\norm{ w^{\epsilon} }_{X} \rightarrow{0}$ as $\epsilon \rightarrow{0}.$
    
    \item \label{Invertability} \textup{(Invertibility)} 
    The linearized operator $\mathcal{F}_w(w^\epsilon, F^\epsilon)$ is invertible $X \to Y$ for all $\epsilon \in (0,\epsilon_*)$
    
    \item \label{Uniqueness} \textup{(Uniqueness)}If $w\in X$ satisfies $w>0$ on T and $\norm{ w }_X$ is small enough, then for any $\epsilon \in (0,\epsilon_{*})$, $\mathcal{F}(w,F^{\epsilon})=0$ implies $w=w^{\epsilon}.$
    \item \label{Waves of Elevation} \textup{(Elevation)} $(w^{\epsilon},F^{\epsilon})$ is a wave of elevation: $w^{\epsilon}>0$ on $R \cup I \cup T.$
\end{enumerate}
    \end{theorem}
    In proving Theorem~\ref{smallamplitudethoery}, we will use the center manifold reduction technique introduced in \cite{chen2019center}, which is a variation of the classical theory due to Kirchg\"assner \cite{Kirchgassner1982Wavesolutions} and Mielke \cite{Mielke1986Reduction,mielke1988reduction}. This newer version is well-suited to the present work as it is conducted entirely in spaces of H\"older class functions and the computation of the reduced equation on the center manifold is done through a power series expansion that is comparatively straightforward. Moreover, the resulting ODE directly governs the internal interface, which allows us to prove that $w^\epsilon$ is a wave of elevation rather easily.
    
    Recall that from Lemma~\ref{SpectrumofEvalues}, the spectrum of the transversal linearized operator at the trivial solution $(w,\epsilon) = (0,0)$ consists of a simple $0$ eigenvalue with the remainder being strictly negative. For convenience, in this section, we write the linearized operator around the trivial flow as $\mathcal{L}=(\mathcal{L}_1, \mathcal{L}_2, \mathcal{L}_3)$ with 
    \begin{equation} \label{linearized operator around 0}
  \begin{cases}
    \mathcal{L}_1w:=\left(\dfrac{w_p}{H_p^{3}}\right)_p+\left(\dfrac{w_q}{H_p}\right)_q-\mucr \rho_{p}w,\\
   \mathcal{L}_2w:=\left(-\dfrac{w_p}{H_p^3}+\mucr \rho w\right)\Big|_T,\\
   \mathcal{L}_3w:=-\jump{ \dfrac{w_p}{H_p^3} }+\mucr \jump{ \rho} w|_I.
   \end{cases}  
   \end{equation}
   Here we have reintroduced the shorthand $\mucr = 1/\Fcr^2$. 
   
    As always, the computation of the center manifold reduction requires (temporarily) expanding our function spaces to allow small exponential growth in $q$.  In particular, we will view $\mathcal{L}$ as mapping $X_\nu \to Y_\nu$ for some $0 < \nu \ll 1$, where $X_\nu$ and $Y_\nu$ correspond to $X$ and $Y$ with the standard H\"older norm replaced by the exponentially weighted version defined in \eqref{weightedholderspace}.  Then it is easily confirmed that the kernel of this operator is two dimensional and takes the form
    \begin{equation} \label{L kernel}
    \textrm{ker }\mathcal{L}:X_\nu \to Y_\nu = \left\{(A+Bq)\Phi_0(p)\in X_\nu: ~ (A,B)\in \mathbb{R}^2 \right\},
    \end{equation}
where $\Phi_0$ generates the null space of the transversal linearized operator and is normalized so that $\Phi_0(\hat{p}) = 1$.  It will be convenient to introduce a projection $\mathcal{Q} : X_\nu \to X_\nu$ onto this kernel given by
\[
\mathcal{Q} w =\left(w(0,\hat{p})+w_q(0,\hat{p})q\right)\Phi_0(p).
\]

Finally, note that the height equation \eqref{quasiw} can be written as a quasilinear transmission problem for the elliptic PDE operator
\[ \nabla \cdot \mathcal{A}(p,w,\nabla w) + \mathcal{B}(p,w,\epsilon),\]
where $\mathcal{A}$ and $\mathcal{B}$ are $C^8$ in their arguments due to the regularity assumptions in \eqref{regularity ustar}.  Of course, $\mathcal{A}$ and $\mathcal{B}$ are actually analytic with respect to $w$ and $\nabla w$, but they have finite smoothness in $p$ because the coefficients involve $\rho$ and $H_p$.

Together these facts ensure that the hypotheses of the center manifold reduction result \cite[Theorem 1.1]{chen2019center} are satisfied (specifically, we use the extension of that theorem to transmission problems given in \cite[Section 2.7]{chen2019center}). As a direct consequence, we obtain the following.
\begin{lemma}[Center Manifold]\label{CenterManifold}
There exists $0 < \nu \ll 1$, neighborhoods $\mathcal{U} \subset X \times \mathbb{R}$ and $\mathcal{V} \subset \mathbb{R}^3$, and a $C^5$ coordinate map $\Lambda=\Lambda(A,B,\epsilon) : \mathbb{R}^3 \to X_\nu$ satisfying 
\[
\Lambda(0,0,\epsilon)=\Lambda_{A}(0,0,\epsilon)=\Lambda_{B}(0,0,\epsilon)=0 \textrm{ for all } \epsilon,
\]
such that the following hold
\begin{enumerate}[label=\normalfont{(\alph*)}]
    \item Suppose that $(w,\epsilon) \in \mathcal{U} $ solves \eqref{quasiw}. Then $v(q):=w(q,\hat{p})$ solves the second-order ODE
    \begin{equation}\label{secondoderderODE}
        v''=f(v,v',\epsilon),
    \end{equation}
    where $f:\mathbb{R}^3\mapsto\mathbb{R}$ which is defined as follows
    \begin{equation}\label{f}
        f(A,B,\epsilon):=\dfrac{d^2}{dq^2}\bigg|_{q=0}\Lambda(A,B,\epsilon)(q,\hat{p}).
    \end{equation}
    \item\label{solution of ODE to PDE} Conversely, if $v:\mathbb{R}\rightarrow{\mathbb{R}}$ solves the ODE \eqref{secondoderderODE} and  $\left(v(q),v'(q),\epsilon\right) \in \mathcal{V}$ for all $q$, then $v:=w(\cdot,\hat{p})$ for solution $(w,\epsilon)\in \mathcal{U}$ of the PDE \eqref{quasiw}. Moreover, we write it as
    \[
    w(q+\tau,p)=v(q)\Phi_{0}(p)+v'(q)\tau\Phi_{0}(p)+\Lambda\left(v(q),v'(q),\epsilon\right)(\tau,p),
    \] for all $\tau \in \mathbb{R}$. 
\end{enumerate}
\end{lemma}
\begin{remark} \label{local symmetry remark}
By inspection, it is easy to verify that the height equation \eqref{quasiw} is invariant under the reversal transformation $w \mapsto w(-\cdot,\cdot)$.  One can show that this gives rise to a symmetry for the coordinate map:
\begin{equation*}
\label{Lambda symmetry} 
\Lambda(A,B,\epsilon)(q,p) = \Lambda(A,-B,\epsilon)(-q,p),
\end{equation*}
and hence $f$ is even in $B$.
\end{remark}

The next step is to the derive the reduced ODE \eqref{secondoderderODE} on the center manifold. In \cite[Theorem 1.3]{chen2019center}, it is proved that the coordinate map $\Lambda$ admits the Taylor expansion 
\begin{equation} \label{coordinatemapexpansion}
    \Lambda(A,B,\epsilon) :=\sum_{\mathcal{J}} \Lambda_{ijk}A^iB^j \epsilon^k + \mathcal{O}\left( (|A|+|B|)(|A|+|B|+|\epsilon|)^{4}\right),
\end{equation}
where
\[
    \mathcal{J}=\{(i,j,k)\in \mathbb{N}^3 :2i +3j +k \leq 4, i+j+k \geq 2, i+j \geq 1 \}.
\] 
Each of the coefficient functions $\Lambda_{ijk} \in X_\nu$ lies in the kernel of $\mathcal{Q}$ and satisfies
\[\partial^{i}_A\partial^{j}_B\partial^{k}_\lambda\bigg|_{(A,B,\lambda)={(0,0,0)}}\mathcal{F}((A+Bq)\Phi_0(p)+\Lambda(A,B,\lambda),\lambda)=0 \qquad \textrm{ for all } 2i+3j+k \leq 4,
\]
with the above derivatives being of the formal G\^ateaux type. By \cite[Lemma 2.3]{chen2019center}, this determines the $\Lambda_{ijk}$ uniquely. Note that our need to expand to fourth order, and hence for $\Lambda$ to be $C^5$, is precisely the reason behind the regularity of the background flow assumed in \eqref{regularityofuv}. 

Explicitly, the index set $\mathcal{J}$ only contains the following 3-tuples:
\begin{equation*}
    \mathcal{J}=\{(2,0,0),(0,1,1),(1,0,1),(1,0,2)\}.
\end{equation*} 
Following the procedure outlined in \cite[Section 2.6]{chen2019center},  computing the coefficients in the expansion \eqref{coordinatemapexpansion} requires us to solve a hierarchy of equations taking the general form
\[\left\{
\begin{aligned}
     \mathcal{L}\Lambda_{ijk}&=R_{ijk},\\
     \mathcal{Q}\Lambda_{ijk}&=0, 
\end{aligned}
\right.
\]
where each $R_{ijk} \in Y_\nu$ depends on previously computed terms. This calculation is largely elementary but quite onerous. For that reason, we use a computer algebra package to verify the results. 

Computing the above G\^ateaux derivatives, we see that $\mathcal{L} \Lambda_{101} = 0$, and hence $\Lambda_{101} = 0$ by uniqueness.  The same type of calculation will also show that $\Lambda_{011} = 0$. 
The remaining coefficients, however, are nontrivial.  Indeed, we find that
\begin{equation} \label{Psi102 equation}
\mathcal{L} \Lambda_{102} = \begin{pmatrix} -\rho_p \Phi_0 \\ \rho(0) \Phi_0(0) \\ \jump{ \rho } \Phi_0(\hat{p}) \end{pmatrix} =: \begin{pmatrix} \mathcal{R}_1 \\ \mathcal{R}_2 \\ \mathcal{R}_3 \end{pmatrix}, \qquad \mathcal{Q} \Lambda_{102} = 0.
\end{equation}
 Since, $\mathcal{L}\Lambda_{102}$ is independent of $q$, we infer that $\mathcal{L}(\partial_q \Lambda_{102})=0$.
Therefore $\partial_q\Lambda_{102}$ is in the kernel of $\mathcal{L}$, and so by \eqref{L kernel} it must take the form 
\begin{equation}\label{psi102}
    \Lambda_{102}(q,p)=(A_1q+\frac{1}{2}B_1q^2)\Phi_0(p)+K_1(p),
\end{equation}
for some constants $A_1, B_1$, and function $K_1$ to be determined.
Applying the operator $\mathcal{L}$ to this ansatz and recalling \eqref{Psi102 equation}, we obtain 

\begin{equation}\label{eqL1}
    \mathcal{L}^\prime K_1 = \mathcal{R} - \begin{pmatrix} \frac{B_1 \Phi_0}{H_p} \\ 0 \\ 0 \end{pmatrix}.
    \end{equation}
Here, $\mathcal{L}^\prime$ is the transversal linearized operator found by restricting $\mathcal{L}$ to $q$-independent functions. Note that $\Phi_0$ generates the kernel of $\mathcal{L}^\prime$ by definition.  

Now, multiplying the first component of the equation in \eqref{eqL1} by $\Phi_0$ and integrating by parts, we find that
\[
    B_1 \int_{-1}^{0}\dfrac{\Phi_0^2}{H_p} \,dp +\jump{ \frac{(\Phi_0)_p K_1 - \Phi_0 (K_1)_p}{H_p^3} } + \left.\left( \frac{\Phi_0 (K_1)_p -(\Phi_0)_p K_1 }{H_p^3} \right) \right\vert_{p=0} =   \int_{-1}^{0} \mathcal{R}_1 \Phi_0 \,dp. 
\]
Using the fact that $\mathcal{L}^\prime \Phi_0 = 0$, the above identity simplifies to 
\[
    \displaystyle B_1 \int_{-1}^{0}\dfrac{\Phi_0^2}{H_p} \,dp +\mathcal{R}_3 \Phi_0(\hat{p})-\mathcal{R}_2 \Phi_0(0)= \int_{-1}^{0} \mathcal{R}_1 \Phi_0 \,dp.
\]
Hence, $B_1$ takes form
\begin{equation*}
    B_1 =\dfrac{ \displaystyle \int_{-1}^{0} \mathcal{R}_1 \Phi_0 \,dp -\mathcal{R}_3 \Phi_0(\hat{p})+\mathcal{R}_2 \Phi_0(0)}{\displaystyle \int_{-1}^{0}\dfrac{\Phi_0^2}{H_p} \,dp} > 0.
\end{equation*}
Notice that $A_1 q \Phi_0$ is in the kernel of $\mathcal{Q}$, and hence will not enter into the reduced equation.  Thus we have computed the relevant part of $\Lambda_{102}$.

Following the same strategy for $\Lambda_{200}$ yields
\begin{equation}\label{psi200}
    \Lambda_{200}=(A_{2} q+\frac{1}{2}B_{2}q^2)\Phi_0(p)+K_2(p).
\end{equation}
where the coefficient 
\begin{equation*}
    B_2 =-\dfrac{3}{2}\ \dfrac{\displaystyle  \int_{-1}^{0} \dfrac{(\Phi_0)^3_p}{H_p^4}\,dp}{\displaystyle \int_{-1}^{0}\dfrac{\Phi_0^2}{H_p} \,dp} < 0.
\end{equation*}

At this stage, we have all the information needed to find the reduced equation \eqref{f}.  First, we consider the truncated ODE where only the leading order terms of $\Lambda$ are retained:

\begin{equation}\label{TRUNCATEDreducedODE}
v^{0}_{qq}=B_1{\epsilon}^2v^{0}-B_2(v^{0})^2,
\end{equation} where $v^{0}(q):=w(q,\hat{p})$. Note that the functions $K_1$ and $K_2$ play no role as they are independent of $q$. One can verify directly that 
\begin{equation}\label{sech^2}
    v^{0}(q)=\dfrac{3B_1 \epsilon^2}{2B_2} \text{sech}^2 (\dfrac{\epsilon \sqrt{B_1}}{2}q).
\end{equation}
is an explicit solution to the truncated reduced equation that is homoclinic to $0$.  It is left to show that this orbit persists for the full reduced equation.
\begin{proof}[Proof of Theorem~\ref{smallamplitudethoery}]
Let us introduce the scaled variables
\begin{equation}\label{scalings}
 q=\epsilon^{-1} Q, \quad v(q)= \epsilon^2 V(Q).
\end{equation}
We can then rewrite the reduced ODE \eqref{TRUNCATEDreducedODE} as the following planar system:
\begin{equation}\label{fullODErescalled}
  \begin{cases}
    V_{Q}=W, & \\
    W_{Q}=B_1V-B_2V^2+R(V,W,\epsilon).&
  \end{cases}
\end{equation}
The error term $R(A,B,\epsilon)=\mathcal{O}(\epsilon^2|A|+|\epsilon||B|))$ by \eqref{coordinatemapexpansion} and is even in $B$ due to Remark~\ref{local symmetry remark}.  Taking $\epsilon=0$, we get back a rescaled version of the truncated reduced ODE in \eqref{TRUNCATEDreducedODE}. Moreover, the explicit solution $v^0$ \eqref{sech^2} in the rescaled variables becomes
\[
V^{0}(Q)=\dfrac{3B_1}{2B_2} \text{sech}^2 (\dfrac{ \sqrt{B_1}}{2}Q), \quad W^{0}(Q)=-\dfrac{3B_1^{3/2}}{2B_2} \tanh{(\dfrac{\sqrt{B_1}}{2}Q)} \text{sech}^2 (\dfrac{ \sqrt{B_1}}{2}Q).
\]
which is an explicit solution to \eqref{fullODErescalled} when $\epsilon=0$.  Moreover, this orbit is homoclinic to the origin and intersects the $V$-axis transversally. The symmetry property exhibited by $\Lambda$ \eqref{Lambda symmetry} implies that the ODE \eqref{fullODErescalled} is reversible in the sense that it is invariant with respect to   $\left(V(Q),W(Q)\right)\mapsto \left(V(-Q),-W(-Q)\right)$.  A standard planar systems argument then implies that the homoclinic orbit $(V^0,W^0)$ persists for sufficiently small $\epsilon$, giving a continuous one-parameter family of homoclinic solutions. Undoing the scaling, we obtain the local curve $\mathscr{C}_{\textup{loc}}$, proving \eqref{localcurve}.  Part~\ref{continuity} is a consequence of the continuity of the reduction function. 

Next, we will show that $\mathscr{C}_{\textup{loc}}$ consists of waves of elevation as claimed in \ref{Waves of Elevation}. It is easy to see from the equation \eqref{fullODErescalled} satisfied by $V^\epsilon := w^\epsilon(\cdot/\epsilon,\hat{p})/\epsilon^2$ that $V^\epsilon > 0$ and exponentially localized for $0<\epsilon \ll 1$. Moreover from the phase portrait as $Q\to\pm \infty$ we can infer that
\begin{equation*}
    \lim_{Q \to \pm \infty} \dfrac{W(Q)}{V(Q)}= \pm \sqrt{B_1}+\mathcal{O}\left(\epsilon\right). 
\end{equation*}
Taking $\epsilon$ small enough and undoing the scaling, this yields 
\begin{equation}\label{inequality of w and v}
    |v^\prime(q)|\leq C \epsilon |v(q)|,
\end{equation}
which holds for some $C > 0$ and all $0 < \epsilon \ll 1$.

On the other hand, combining the solution ansatz given by Lemma~\ref{CenterManifold}\ref{solution of ODE to PDE} with the expansion of the coordinate map $\Lambda$ in \eqref{coordinatemapexpansion}, we have
\begin{equation}\label{w epsilon}
   w^\epsilon(q,p)=v(q)\Phi_0(p)+\Lambda_{200}(0,p)v(q)^2+\Lambda_{102}(0,p)v(q)\epsilon^2+r\left(v(q),v^\prime(q),\epsilon\right)(0,p), 
\end{equation} where the remainder term satisfies 
\[ r(A,B,\epsilon)(0,\cdot) =\mathcal{O}\left( (|A|+|B|)(|A|+|B|+|\epsilon|)^{4}\right) \qquad \textrm{in } C^2([-1,\hat{p}]) \cap C^2([\hat{p},0]).\]

In concert with \eqref{w epsilon} and \eqref{inequality of w and v}, this shows that for $\epsilon$ sufficiently small we have $w_p > 0$ in $R^- \cup R^+$. Since $w =0$ on $p=-1$, this gives that $w > 0$ in $R$, proving part~\ref{Waves of Elevation}.

Lastly, we will show that the linearized operator $\mathcal{F}_{w}(w,F)$ on the $\mathscr{C}_{\textup{loc}}$ is invertible. Recall also that from Lemma~\ref{Fedholm of index zero of the linearized opt}, we know that $\mathcal{F}_{w}(w,F)$ is Fredholm of index 0. As a consequence, for all $(w,F) \in \mathscr{C}_{\textup{loc}}$, $\mathcal{F}_{w}(w,F)$ is invertible if and only if it has a trivial kernel. As we have seen many times, the translation invariance in $q$ means that $\mathcal{F}_{w}(w^\epsilon,F^\epsilon)w_q^\epsilon=0$. 

To identify other potential solutions of the linearized problem, we make use of \cite[ Theorem 1.6 ]{chen2019center}.  This result states that $\dot{w}$ satisfies $\mathcal{F}_w(w^\epsilon,F^\epsilon)\dot{w} = 0$ if and only if $\dot{v} := \dot{w}(\cdot, \hat{p})$ solves the linearized reduced ODE 
\[ \frac{d^2}{dq^2} \dot{v} = f_{(A,B)}(v^\epsilon,v_q^\epsilon,\epsilon) \cdot (\dot{v}, \dot{v}_q),\]
where $v^\epsilon = w^\epsilon(\cdot,\hat{p})$.  Performing the same rescaling as before, we see that the corresponding (linear) planar system is given by 
\[
\begin{pmatrix}
\dot{V}_{Q}\\
\dot{W}_{Q}
\end{pmatrix}
=
\begin{pmatrix}
0 & 1\\
f_{A} & f_{B}
\end{pmatrix}
\begin{pmatrix}
\dot{V}\\
\dot{W}
\end{pmatrix}
,\]
where $f_A$ and $f_B$ are evaluated at $(\epsilon^2 V_Q^\epsilon(Q), \epsilon V^\epsilon(Q), \epsilon)$.  Sending $|Q| \to \infty$, we therefore obtain 
\[
\lim_{Q\to \pm \infty} \begin{pmatrix}
0 & 1\\
f_{V} & f_{W}
\end{pmatrix}
=\begin{pmatrix}
0 & 1\\
B_1 + \mathcal{O}(\epsilon^2) & 0
\end{pmatrix}.\] 
Clearly, the eigenvalues of the above matrix are both real, with one strictly positive and the other strictly negative.  By standard dynamical system theory, there cannot exist two linearly independent bounded solutions to the reduced ODE. Hence, the only nontrivial solution of $\mathcal{F}_w(w^\epsilon,F^\epsilon)\dot w = 0$ is $\dot w = w_q^\epsilon$.  But we have already established that $w^\epsilon$ is even, hence $w_q^\epsilon$ is not in $X$.  Thus the kernel of $\mathcal{F}_w(w^\epsilon,F^\epsilon) : X \to Y$ is trivial, completing the proof.
\end{proof}

\section{Large-amplitude existence theory}\label{largeamplitude}
In this final section, we complete the argument for the existence of the curve $\mathscr{C}$ of large-amplitude solutions and show that it exhibits the properties asserted in Theorem~\ref{thm:main}.  The global curve is constructed through continuation of the local curve $\mathscr{C}_\loc$ obtained in Section \ref{Small Amplitude}.  Our strategy is in the spirit of Wheeler's \cite{wheeler2013large,wheeler2015froude} work on homogeneous rotational waves and that of Chen, Walsh, and Wheeler's \cite{chen2018existence} study of continuously stratified fluids.  In particular, the latter of these papers develops a general analytic global bifurcation theory that is adapted to monotone solutions on unbounded domains; see Appendix~\ref{quoted results appendix}.  Using that machinery allows us to prove the existence $\mathscr{C}$.  To verify the extreme wave limit requires the bounds on the velocity field given by Theorem~\ref{thm:velocity}, which we prove in the next subsection, and uniform regularity estimates that are tackled in Section~\ref{uniform regularity section}.

\subsection{Velocity bound}
Here, we derive some uniform $L^\infty$ bounds on the velocity for solitary stratified waves. Throughout the analysis, the far-field state (as described by $H$ and $\rho$ or equivalently $\mathring{u}$ and  $\varrho^*$) is fixed.  To simplify the presentation, we do not track how the constants depend on these quantities. Recall that the relative velocity in terms of the stream function is given by $\nabla^{\perp} \psi$.  Having uniform $L^\infty$ control on the velocity will ensures  the uniformly ellipticity of the height equation \eqref{quasi} along  $\mathscr{C}$. 

In earlier studies of rotational waves in constant density water \cite{Varvaruca2009extremewaves} and continuously stratified fluids \cite{chen2018existence}, velocity bounds were obtained by first establishing a lower bound of the pressure via the maximum principle.  Bernoulli's equation then allows one to uniformly control the magnitude of the relative velocity. However, in the present work, it is not obvious that we can apply the same strategy due to the transmission boundary condition. Adopting instead the approach of \cite{amick1986global}, we start by deriving a ``local" $L^2$ bound of $\nabla \psi$ which is recorded in the lemma below.

\begin{lemma}[Local velocity bound]\label{Local velocity bound}
There exists $C = C(F_0, K) > 0$ such that, for any solution $(\psi,\eta)$ to \eqref{YIH}--\eqref{boundaryconditioninstreamfunction} with $\psi_y < -1/K$, $F \geq F_0 > 0$, and any $m \in \mathbb{R}$, 
\begin{equation}\label{L^2 bound of grad Psi}
    \int_{m-1}^{m+1} \int_{-1}^{\eta(x)} | \nabla \psi|^2 \,\,dy \,dx < C.
\end{equation}
\end{lemma}
\begin{proof}
Working in semi-Lagrangian variables, this is equivalent to
\[
\int_{m-1}^{m+1} \int_{-1}^{0} \dfrac{1+h_q^{2}}{h_p} \,dp \, dq \leq C.
\]
Let $\xi= \xi(q)$ be a bump function supported in $[m-2,m+2]$, with $0\leq\xi\leq1$ and $\xi \equiv 1$ on $[m-1,m+1]$. If we multiply the interior height equation \eqref{quasi} by $\xi^2 w$, and integrate over the domain, we obtain
\begin{equation}\label{velbound}
  \begin{split}
  & \iint_{R}\left (\left(\dfrac{1+h_q^2}{2h_p^2}-\dfrac{1}{2H_p^2}\right) w_p -\left(\dfrac{h_q}{h_p}\right)w_q \right)\xi^2  \,dp \, dq  \\
  & \qquad = \iint_{R}\dfrac{2h_q}{h_p}w\xi \xi_q+\dfrac{1}{F^2}\rho_p w^2 \xi^2   \,dp \, dq + \int_{T}\left(\dfrac{1+h_q^2}{2h_p^2}-\dfrac{1}{2H_p^2}\right)w \xi^2 \, dq  \\ &\qquad\qquad + \int_{I} \jump{ - \dfrac{1+h_q^2}{2h_p^2}+ \dfrac{1}{2H_p^2}} w \xi^2 \, dq. 
  \end{split}
\end{equation}

Observe that the factor of $\xi^2$ in the integrand on the left hand side of the above equation can be rewritten as follows:
\[\left(\dfrac{1+h_q^2}{2h_p^2}-\dfrac{1}{2H_p^2}\right) w_p -\left(\dfrac{h_q}{h_p}\right)w_q=\dfrac{2H_p +w_p}{2h_p^2}\left(-\dfrac{w_p^2}{H_p^2}-w_q^2 \right).\]
Therefore, via \eqref{velbound} and the above equality along with Young's inequality, we have
\begin{equation}\label{velocityboundineq}
\begin{split}
  \iint_{R}\dfrac{2H_p +w_p}{2h_p^2}\left(\dfrac{w_p^2}{H_p^2}+w_q^2 \right) \xi^2  \,dp \, dq&=-\iint_{R} \dfrac{2h_q}{h_p}w\xi \xi_q  \,dp \, dq \\
  & \qquad +2\iint_{R}\dfrac{1}{F^2}\rho w_p w \xi^2  \,dp \, dq \\&\lesssim \epsilon_1 \iint_{R} \dfrac{h_q^2}{h_p^2}w^2\xi^2  \,dp \, dq +\dfrac{1}{\epsilon_1}\int_{\mathbb{R}}\xi_q^2 dq\\& \qquad +\epsilon_2\iint_{R}\rho^2 w_p^2 w^2 \xi^2   \,dp \, dq+\dfrac{1}{F^2\epsilon_2} \int_{\mathbb{R}}\xi^2 \, dq,
\end{split}
\end{equation}
where the constants $\epsilon_1$ and $\epsilon_2$ are defined to be
\[
        \epsilon_1 := \dfrac{\inf H_p}{2K^2}, \qquad
        \epsilon_2 :=\dfrac{1}{2\norm{ \rho }^2_{L^\infty}K^2\norm{ H_p }_{L^\infty}}.
\]
Combining this with \eqref{velocityboundineq} yields
\[\iint_{R}\dfrac{1}{h_p} \xi^2 \,dp \, dq +\iint_{R}\dfrac{ h_q^2}{h_p}\xi^2 \,dp \, dq \leq C(F_0, K).\] This then proves the estimate in \eqref{L^2 bound of grad Psi}.
\end{proof}
Next, we show that along the internal interface, the velocity is uniformly controlled in $L^\infty$. Our approach is based on that of Amick and Turner \cite{amick1986global} and \cite{chen2020globalmonotone}. In both those papers, however, the stream function is harmonic in each layer, which permits them to use the classical monotonicity formula of Alt--Caffarelli--Friedman \cite[Lemma 5.1]{alt1984variational}. Because we allow for general stratification, we must instead use the slightly weaker ``almost monotonicity formula'' due to Caffarelli--Jerison--Kenig \cite{Caffarelli2002monotonicity}. The precise application is presented in the next lemma. However, prior to using the formula, an intermediate step is done to make sure that the upper and internal layers are uniformly separated. This is the content of the next corollary which follows from the local estimate \eqref{L^2 bound of grad Psi} and the fact that the relative pseudo-volumetric mass flux is fixed.
\begin{corollary}[Interface separation bound]\label{separation corollary}
    Under the hypothesis of Lemma~\ref{Local velocity bound}, we have
\begin{equation}\label{separation}
    \inf_x \left( \eta(x)-\zeta(x) \right) > C|\hat{p}|^2,
\end{equation}
where the constant $C = C(F_0, K) > 0$.
\end{corollary}
\begin{proof}
Via the definition of the relative pseudo-volumetric mass flux in the upper layer, we have for all $x$ and $m \in \mathbb{R}$
\begin{equation*}
\begin{aligned}
    |\hat{p}|&=\int_{\zeta(x)}^{\eta(x)}\sqrt{\varrho}(c-u)\, dy =\frac{1}{2} \int_{m-1}^{m+1}\int_{\zeta(x)}^{\eta(x)}\sqrt{\varrho}(c-u)\, dy \, dx\\&\leq \frac{1}{2}\left(\int_{m-1}^{m+1}\int_{\zeta(x)}^{\eta(x)} 1\,dy\, dx\right)^{1/2} \left( \int_{m-1}^{m+1}\int_{\zeta(x)}^{\eta(x)}\varrho (c-u)^2\,dy\, dx\right)^{1/2}\\&< C \sqrt{\eta(x)-\zeta(x)},
\end{aligned}
\end{equation*}
where we have used the fact that $\varrho (c-u)^2 \leq |\nabla \psi|^2$ and the local velocity bound \eqref{Local velocity bound}. This then leads to the inequality in \eqref{separation}.
\end{proof}

\begin{lemma}[Interfacial velocity bounds] \label{interface velocity bounds lemma}  
For any solution $(\psi,\eta,\zeta)$ to the water wave problem \eqref{YIH}--\eqref{boundaryconditioninstreamfunction} with $\sup \psi_y < -1/K$ and $F \geq F_0 > 0$ satisfies the bound
\begin{equation} \label{interface velocity bound}
    |\nabla \psi_+|^2 + |\nabla \psi_-|^2 < C \qquad \mathrm{on }~ \mathscr{I},
\end{equation}
where the constant $C = C(F_0, K) > 0$. 
\end{lemma}
\begin{proof}
Again, we will use $C$ to denote a generic positive constant depending only on the quantities listed in the statement.  

Fix $(x_0,y_0)\in \mathscr{I}$, where recall $\mathscr{I}$ denotes the internal interface in the original coordinate system. Let $a:=\mathrm{dist}((x_0,y_0),\partial \Omega\setminus\mathscr{I})$. Observe that, in view of Corollary~\ref{separation corollary}, $a$ is uniformly positive.  

We will work with the rescaled coordinates $(\tilde x, \tilde y) = ((x-x_0)/a, (y-y_0)/a)$.  Likewise, we introduce the modified and rescaled  stream functions $\tilde\psi_-$ and $\tilde \psi_+$ defined by 
\[ \tilde \psi_-(\tilde x, \tilde y) := \max\left\{0, \, \frac{\psi(x,y) + \hat{p}}{\mathcal{M}_-}\right\}, \qquad \tilde \psi_+(\tilde x, \tilde y) := \max\left\{0, \,- \frac{\psi(x,y) + \hat{p}}{\mathcal{M}_+}\right\} \]
for $(\tilde x, \tilde y) \in B_1$, the unit ball centered at $(0,0)$ in the $(\tilde x, \tilde y)$-plane.  Here,
\[
\mathcal{M}_-:= 2a^2 \max\left\{\norm{\beta_+}_{L^\infty},\,  \dfrac{1}{F^2} K \norm{ \rho_p }_{L^{\infty}}\right\},
\quad
\mathcal{M}_+:= 2a^2 \max\left\{ \| \beta_- \|_{L^\infty}, \, \frac{1}{F^2} \| \rho_p \|_{L^\infty} \right\}.
\]
Notice that $\tilde \psi_\pm$ is non-negative and $C^{0+\alpha}(B_1)$ but the product $\tilde \psi_- \tilde \psi_+$ vanishes identically.  Moreover, from the definitions of $\mathcal{M}_\pm$ and Yih's equation \eqref{YIH}, we find that
\[
    \Delta_{(\tilde x, \tilde y)} \tilde \psi_\pm \geq -1 \qquad \textrm{on } B_1,
\]
with the inequality holding in the sense of distributions.  

Consider the function
\[\phi(r):=\left(\dfrac{1}{r^2}\iint_{B_r}|\nabla \tilde\psi_+|^2 \,dx\,dy\right)\left(\dfrac{1}{r^2}\iint_{B_r}|\nabla \tilde\psi_-|^2 \,dx\,dy\right) \qquad \textrm{for } 0 < r < 1,\]
where $B_r$ denotes the ball of radius $r$ centered at the origin in the $(\tilde x, \tilde y)$ variables.  
Applying \cite[Theorem 1.3]{Caffarelli2002monotonicity}, we can infer that for any $0<r<1$,
\begin{equation}\label{bound for phi(r)}
    \phi(r)\leq C_0 \left(1+ \iint_{B_1} |\nabla\tilde{\psi}_+|^2 \,dx\,dy+ \iint_{B_1} |\nabla\tilde{\psi}_-|^2 \,dx\,dy \right)^2,
\end{equation}
for a universal constant $C_0 >0$.
Because of the H\"older regularity of $\tilde \psi_\pm$, from \cite[Theorem 1.6]{Caffarelli2002monotonicity}, we know that $\phi(r)$ has a limiting value as $r \to 0$, which must then coincide with
\[
\phi(0) :=\dfrac{\pi^2}{4}|\nabla\tilde\psi_+(0,0)|^2 |\nabla\tilde\psi_-(0,0)|^2.
\]
Combining this with \eqref{bound for phi(r)} and undoing the scaling, we can say that on internal interface
\begin{equation}\label{product of nabla psi}
   |\nabla {\psi}_+(x_0,y_0)|^2|\nabla {\psi}_-(x_0,y_0)|^2 < C. 
\end{equation}
Finally, from equation \eqref{product of nabla psi} and the Bernoulli condition \eqref{jump in grad psi}, we arrive at the desired bound \eqref{interface velocity bound}. 
\end{proof}
Having derived the velocity bounds on the interface, we are now prepared to prove Theorem~\ref{thm:velocity}.  We state this result in the Dubreil-Jacotin variables as this is most convenient for applying it to the global bifurcation theory.
\begin{theorem}[Global velocity bounds] \label{thm:relative velocity bound}
Let $(h,F) \in X \times \mathbb{R}$ be a  solution to the height equation \eqref{quasi} with $\|h_p\|_{L^\infty} < K$ and $F \geq F_0 >0$.  Then
\[\norm{ \dfrac{1}{h_p} }_{L^{\infty}(R)} + \norm{ \dfrac{h_q}{h_p} }_{L^{\infty}(R)}\leq C,\]
where the constant $C = C(F_0, K) >0$.
\end{theorem}
\begin{proof}
Throughout the proof, we use $C$ to denote a generic positive constant depending on the quantities in the statement of the theorem. Observe that, when converted to Eulerian variables, these correspond to the same quantities appearing in the statements of Lemma~\ref{Local velocity bound}, Corollary~\ref{separation corollary}, and Lemma~\ref{interface velocity bound}. In particular, the previous lemma shows that $|\nabla \psi| < C$ on the internal interface. 

It remains, to control $|\nabla \psi|$ away from $\mathscr{I}$. For this, we use a maximum principle argument based on \cite[Proposition 4.1]{chen2018existence}. Define 
\[f:=P+M\psi,\]
for $M > 0$ to be determined. Using the boundedness of $|\nabla \psi|$ and Bernoulli's law, one can then show that $|P| < C$ on $\mathscr{I}$. From Yih's equation \eqref{YIH} and an elementary calculation, we see that $f$ satisfies the elliptic PDE
\begin{equation}\label{ellip}
    \Delta f -b_1 f_x -b_2 f_y=\dfrac{2F^{-2} \rho(2M+\Delta \psi)\psi_y-2F^{-4}\rho^2}{|\nabla \psi|^2}- (2M+\Delta \psi)M+\dfrac{1}{F^2}\rho_p\psi_y,
\end{equation} where $b_1$ and $b_2$ are given as follows:
\begin{equation*}
    b_1:=2\dfrac{\psi_x(2M+ \Delta \psi)}{|\nabla \psi|^2}, \qquad
    b_2:=2\dfrac{\psi_y(2M+\Delta \psi)-2F^{-s} \rho}{|\nabla \psi|^2}.
\end{equation*}
The point here is that there are no zeroth-order terms on the left-hand side of \eqref{ellip}. By proving the right-hand side is non-negative, we will be able to apply the maximum principle to $f$.

With that in mind, observe that from Bernoulli's law we have
\begin{equation*}
    \dfrac{\eta}{F^2} < \dfrac{(\mathring{u}(0)-c)^2}{2}=\dfrac{u^{*}(0)^2}{2gd}.
\end{equation*}
Via Yih's equation, we know that
\begin{equation*}
    \Delta \psi = -\beta +\dfrac{1}{F^2}y \rho_p \geq - \norm{ \beta_{+} }_{L^{\infty}}-\dfrac{u^{*}(0)^2}{2gd} \norm{ \rho_p }_{L^{\infty}}.
\end{equation*}
Hence, choosing

\begin{equation*}
    M>M_1:=\norm{ \beta_{+} }_{L^{\infty}}+\dfrac{u^{*}(0)^2}{2gd} \norm{ \rho_p }_{L^{\infty}},
\end{equation*}
yields $\Delta \psi +M \geq 0.$ In conjunction with \eqref{ellip}, this leads to the inequality
\begin{equation*}
  \Delta f -b_1 f_x -b_2 f_y \leq - \bigg[M^2-\dfrac{1}{F^2}\rho_p \psi_y \bigg].  
\end{equation*} In view of the last inequality, we define
\begin{equation*}
    M_2:=\dfrac{1}{F_{0}}\norm{ \psi_y }^{1/2}_{L^{\infty}}\norm{ \rho_p }^{1/2}_{L^{\infty}}.
\end{equation*}
Provided $M \geq M_1, M_2$, we then have
\begin{equation*}
\Delta f -b_1 f_x -b_2 f_y \leq 0.
\end{equation*}Furthermore, as $x\to \pm \infty$, the pressure approaches hydrostatic, which implies that $f$ is non-negative in the upstream and downstream limits. 

First, we shall show that $\inf f\geq 0$ in $\overline{\Omega_+}$. On the upper surface this holds by definition. Furthermore, we have already established that $|P| < C$ on $\mathscr{I}$. Thus, there exists some $M_3$ depending on the same constants, so that $f = P+M_3|\hat{p}| \geq 0$ on $\mathscr{I}$. 

Finally, setting $M:= \max \{M_1,M_2,M_3\}$, we have that $f \geq 0$ in $\overline{\Omega_+}$ via the maximum principle. 

Similarly, we can show that $f \geq 0$ in $\overline{\Omega_-}$. On the bed,
\begin{equation*}
    f_y=P_y+M \psi_y=-\dfrac{1}{F^2}\rho+M \psi_y <0.
\end{equation*}
By the Hopf lemma, this implies that $f$ cannot attain a minimum there. The claim follows easily via the maximum principle. 

Thus, $f \geq 0$ throughout the domain. Recalling its definition, this gives a lower bound on the pressure. Using Bernoulli's law, we can then control the velocity:
\begin{equation*}
      -\dfrac{1}{2}|\nabla \psi|^2-\dfrac{1}{F^2}\rho y+E=P\geq -M\psi \geq -M.
\end{equation*}
After rearrangement and using the fact that $y >-1$, $F > F_{0}$, this becomes
\begin{equation}\label{rearrangedbernoulli}
    |\nabla \psi|^2 \leq 2M-\dfrac{2}{F^2}\rho y +2E \leq 2M +\dfrac{2}{F_{0}^2}+2E. 
\end{equation}
Suppose that $M_2 \geq M_1,M_3$, that is  $M=M_2$. Taking the supremum of the left hand side of \eqref{rearrangedbernoulli} and dropping the $\psi_x$ term followed by applying Young's inequality yields
\begin{equation*}
     \norm{ \psi_y }_{L^{\infty}}^2 \leq  \dfrac{3}{F_{0}^{4/3}} \norm{ \rho_p }_{L^{\infty}}^{2/3}+\dfrac{4}{F_{0}}+4\norm{ E }_{L^{\infty}}.
\end{equation*} Using the definition of $M_2$ and the $L^{\infty}$ bound of $\psi_y$ above, we can infer
\begin{equation*}
    M_2 \leq \dfrac{1}{F_{0}} \norm{ \rho_p }_{L^{\infty}}^{1/2}\left(\dfrac{3}{F_{0}^{4/3}} \norm{ \rho_p }_{L^{\infty}}^{2/3}+\dfrac{4}{F_{0}}+4\norm{ E }_{L^{\infty}}\right)^{1/4}.
\end{equation*}
Plugging $M_2$ back into \eqref{rearrangedbernoulli} gives
\[
    \norm{ \nabla \psi }_{L^{\infty}}^2 \leq \dfrac{2}{F_{0}} \norm{ \rho_p }_{L^{\infty}}^{1/2}\left(\dfrac{3}{F_{0}^{4/3}} \norm{ \rho_p }_{L^{\infty}}^{2/3}+\dfrac{4}{F_{0}}+4\norm{ E }_{L^{\infty}}\right)^{1/4}+\dfrac{2}{F_{0}}+2\norm{ E }_{L^{\infty}}.
\]

Now, let us look into the case when $M_1 \geq M_2,M_3$, that is $M=M_1$. Plugging $M_1$ into \eqref{rearrangedbernoulli}, this directly gives us
\[
    \norm{ \nabla \psi }_{L^{\infty}}^2 \leq 2M_1 +\dfrac{2}{F_{0}}+2\norm{ E }_{L^{\infty}} \leq C.
\]
Lastly, if $M_3 \geq M_1,M_2$, then setting $M=M_3$ leads to 
\[
    \norm{ \nabla \psi }_{L^{\infty}}^2 \leq 2M_3 +\dfrac{2}{F_{0}}+2\norm{ E }_{L^{\infty}}\leq C.
\]
Rewriting the above inequality in terms of semi-Lagrangian variables, we arrive at the desired inequality stated in Theorem~\ref{thm:relative velocity bound}.
\end{proof}
\begin{remark}\label{wandwp}
As a result of the previous theorem, we can infer that $\norm{ h_q }_{L^{\infty}}$ is controlled by $\norm{ h_p }_{L^{\infty}}$.  Moreover, because
\[w(q,p) = \int_{-1}^p w_p(q,p^\prime)\, dp^\prime ,\]
we have
\begin{equation}\label{wbound}
    \norm{ w }_{L^{\infty}(R)} \lesssim \norm{ w_p }_{L^{\infty}(R)}.
\end{equation}
\end{remark}

\begin{corollary}[Bounds on $w$ and  $\nabla w$]\label{boundonw}
There exist constants $C=C(K)$ and $\delta=\delta(K)$ such that every supercritical solitary wave with $\| h_p \|_{C^0} < K$ satisfies
\begin{equation*}
     \inf_{R}(w_p+H_p) \geq \delta,
\end{equation*} and 
\begin{equation*}
     \norm{ w }_{C^{1}(R)} \leq C.
\end{equation*}
\begin{proof}
As before, we will use $C$ to denote a generic positive depending on $K$. Taking $F_{0}=F_{\textup{cr}}$ and applying Theorem~\ref{thm:relative velocity bound} gives the bound 
\begin{equation}\label{absolute value of grad Psi}
    \dfrac{1}{h_q^2}+\dfrac{h_q^2}{h_p^2} < C.
\end{equation}
Dropping the second term on left hand side of \eqref{absolute value of grad Psi} gives
\begin{equation}\label{infimum of h_p}
    \inf_{R} h_p \geq \dfrac{1}{C}:=\delta. 
\end{equation}
Similarly, dropping the first term on the left hand side of \eqref{absolute value of grad Psi} gives
\[
\sup_R{|h_q|}\leq \sqrt{C} \sup_{R}{h_p} \leq C \left(1+\|w_p\|_{C^0(R)} \right).
\]
Combining this with \eqref{wbound}, we obtain the desired $C^{1}$ bound for $w$.
\end{proof}  
\end{corollary}
\subsection{Uniform regularity} \label{uniform regularity section}
The purpose of this subsection is to establish that the full $X$ norm of $w$ can be controlled in terms of $\| w_p \|_{L^\infty}$. This will be used later to prove that, following the global bifurcation curve $\mathscr{C}$, blow-up in norm corresponds to the onset of horizontal stagnation.  Specifically, the main result is as follows.

\begin{theorem}[Uniform Regularity]\label{thm:uniformregularity} For all $\delta > 0$, there exists $C:=C(\delta)>0$ such that, if $(w,F) \in X \times \mathbb{R}$ is a solution of the height equation \eqref{quasiw} satisfying 
 \begin{equation}\label{restriction on the solutions}
     \inf_{R}(H_p+w_p)\geq \delta, \qquad \norm{w_p}_{C^0}+\norm{H_p}_{C^0}<\frac{1}{\delta},
 \end{equation}
 then it obeys the bound:
 \[\norm{ w }_{C^{3+\alpha}(R)} \leq C.\] 
\end{theorem}

 For steady waves in constant density or continuously stratified density water, this type of result is very well-known (see e.g., \cite{Constantin2004Exactsteady,Walsh2009Stratified, wheeler2013large,Walsh2014stratified}). It is a direct consequence of the ellipticity of the height equation and  obliqueness of the Bernoulli condition on the upper boundary  --- which are consequences of \eqref{restriction on the solutions} --- along with the translation invariance of the system.   However, the two-fluid problem considered in the present paper requires a significantly different approach.  To derive estimates near the internal interface, one can adopt the idea of \cite{amick1986global} and work with a weak formulation of the height equation \eqref{quasiw}. The proof of Theorem~\ref{thm:uniformregularity} is then straightforward to obtain following the general argument in \cite[Section 5.6]{chen2020globalmonotone}. For that reason, we will only sketch the details.
 
 Specifically, the height equation 
 \eqref{quasiw} is recast as the distributional equation
\begin{equation}\label{weak formulation of the height eq}
  \begin{cases}
    \nabla \cdot \left(G\left(\nabla w,H_p\right)-\dfrac{1}{F^2}\begin{pmatrix}
    0\\\rho w
    \end{pmatrix}\right)+\dfrac{1}{F^2}\rho w_p &=0 \qquad  \text{in}\; R\cup I,\\
 G_2(\nabla w,H_p)-\dfrac{1}{F^2}\rho w&=0 \qquad  \text{on}\; T,\\
w&=0 \qquad  \text{on}\; B,\\
\end{cases}  
\end{equation}\\
where $G=(G_1,G_2)=(\partial_{\xi_1} f, \partial_{\xi_2} f)$,
for $f:\mathbb{R}^3 \rightarrow \mathbb{R}$ defined by,
\begin{equation}\label{definiton of f}
    f(\xi_1,\xi_2,a):=\dfrac{a^2\xi_1^2+\xi_2^2}{2(a+\xi_2)a^2}.
\end{equation}
Notice that the transmission condition on the internal interface is enforced by the fact that the first equation above holds on $R \cup I$.  

The main tool for proving regularity near the internal interface is the following theorem of Meyers, stated here in a simplified form appropriate to our setting.

\begin{theorem}[Meyers, \cite{Meyers1963Estimate}] \label{thm:meyers}
 Let $\mathscr{D} \subset \mathbb{R}^2$ be a smooth domain. Consider 
 \begin{equation}\label{meyerseq}
     \nabla \cdot (A \nabla u)=\nabla \cdot G +g \quad \text{in} \quad \mathscr{D},
 \end{equation}
 and 
 \[u=0 \quad \text{on} \quad \partial\mathscr{D},\] where $A=A(x)$ is a matrix  with measurable coefficients and enjoys $c_1I \leq A \leq c_1^{-1} I$ for some $c_1>0$, where I is the $2 \times 2$ identity matrix. Then there exists $r=r(c_1) \geq 2$ such that for all $G \in L^r(\mathscr{D})$ and $g \in L^2(\mathscr{D})$, \eqref{meyerseq} admits a unique solution $u \in W^{1,r}_{0}(\mathscr{D}),$ and satisfies the following inequality
 \[\norm{ \nabla u }_{L^{r}(\mathscr{D})} \leq C\left(\norm{ G }_{L^{r}(\mathscr{D})} +\norm{ g }_{L^2(\mathscr{D})}\right).\]
\end{theorem}
 To use this result, we differentiate the equation \eqref{weak formulation of the height eq} in the $q$-variable $k$ times, say, regroup terms as in \eqref{meyerseq} treating $u := \partial_q^k w$ as the unknown.  Iterating this process furnishes $W^{1,r}$ estimates for successively higher-order $q$ derivative of $w$. Eventually, by Morrey's inequality, this will lead to a sufficient H\"older regularity of the trace of $w$ on the boundary. Applying a simple Schauder estimates for the Dirichlet problem yields the desired regularity. 
 
\subsection{Proof of the main result}
Now we are at last prepared to give the proof of Theorem~\ref{thm:main}. Recall from Section~\ref{Formulation section} that the height equation is expressed in terms of the nonlinear operator $\mathcal{F} : \mathscr{U} \subset X \times \mathbb{R} \to Y$, where the open set $\mathscr{U}$ was defined in \eqref{solutions set}.  Theorem \ref{smallamplitudethoery} states that there exists a continuous local curve of solutions to the height equation \eqref{quasiw} denoted by  \begin{equation*}
        \mathscr{C}_{\textup{loc}}:=\{(w^{\epsilon},F^{\epsilon}):0 < \epsilon < \epsilon_{*} \},
\end{equation*} which contains nontrivial waves of elevation that are symmetric, monotone and slightly supercritical. First, we show in the next theorem that $\mathscr{C}_{\textup{loc}}$ can be extended to a global curve of solutions.

\begin{theorem}[Global continuation]\label{thm:global} 
The continuous local curve of solutions $\mathscr{C}_\loc$ of the nonlinear operator $\mathcal{F}(w,F)=0$ is contained in a global $C^0$ curve $\mathscr{C}$ parameterized as 
\begin{equation*}
        \mathscr{C}:=\{(w(s),F(s)):0 < s < \infty \}\subset \mathscr{U},
\end{equation*} and exhibiting the following properties
\begin{enumerate}[label=\normalfont{(\alph*)}]
    \item \label{global alternatives} One of the following alternatives must hold:
    \begin{enumerate}[label=\normalfont{(\roman*)}]
        \item \label{blowup} \textup{(Blowup)} as $s\rightarrow \infty,$
        \begin{equation*}
        \label{N definition}
            N(s):= \norm{ w(s) }_X +\dfrac{1}{\inf_R \left(w_p(s)+H_p\right)}+F(s)+\dfrac{1}{F(s)-F_{\textup{cr}}} \rightarrow{\infty}.
        \end{equation*}
        \item \label{loss of compactness} \textup{(Loss of compactness)} There exists a sequence of $s_n\rightarrow{\infty}$ as $n\rightarrow \infty$ such that  $\sup_n N(s_n)<\infty$ but $\{w(s_n)\}$ does not have subsequences converging in $X$.
    \end{enumerate}
    \item Fix parameter $s^\star \in (0,\infty)$, around the neighborhood $(w(s^\star),F(s^\star))\in \mathscr{C}$, we can reparametrize $\mathscr{C}$ so that $s\mapsto(w(s),F(s))$ is real-analytic. 
    \item\label{solutions not in local curve} For all $s \gg 1$, $(w(s),F(s)) \notin \mathscr{C}_\loc$.
\end{enumerate}
\end{theorem}
\begin{proof}
Clearly, $\mathcal{F}$ is a real analytic as a mapping $\mathscr{U} \subset X \times \mathbb{R} \to Y$.  Moreover, from Lemma~\ref{Fedholm of index zero of the linearized opt} the linearized operator $\mathcal{F}_w(w,F)$ is Fredholm of index zero for all $(w,F)\in \mathscr{U}$.  In Theorem~\ref{smallamplitudethoery}, we proved that $\mathcal{F}_w(w^\epsilon,F^\epsilon)$ has a trivial kernel for all  $(w^\epsilon, F^\epsilon) \in \mathscr{C}_\loc$. Together these facts import invertibility of $\mathcal{F}_w$ along the local curve $\mathscr{C}_\loc$. The statements of the theorem now follows directly from an application of the abstract global bifurcation result Theorem~\ref{cww global}
\end{proof}
 
The above theorem establishes the existence of a global curve, but we have yet to show that the solutions along it limit to stagnation as claimed in Theorem~\ref{thm:main}\ref{stagnation}.  For that, we prove a series of lemmas bounding the various quantities occurring in the definition of $N(s)$.
\begin{lemma}\label{Nodal on Global Curve}
The nodal properties \eqref{Nodal} hold along the global curve $\mathscr{C}$.
\end{lemma}
\begin{proof}
We start by showing that the nodal properties hold along the local curve $\mathscr{C}_{\textup{loc}}$. Let $(w,F)\in \mathscr{C}_{\textup{loc}}$. By Theorem~\ref{smallamplitudethoery}\ref{Waves of Elevation}, $(w,F)$ is a wave of elevation which means that $w>0$ in $\overline{R}\setminus B$. Then, Theorem~\ref{thm:symmetry} tells us that $w_q<0$ in $\overline{R_{>}}\setminus \left(B_> \cup \overline{L_{0}}\right)$. Hence, by Lemma~\ref{nodalproperty}, we know that the nodal properties hold along the local solutions curve $\mathscr{C}_{\textup{loc}}$. It remains to show that these properties get passed on to solutions along the global curve.

Let $\mathcal{S}\subset\mathscr{C}$ contains the solutions $(w,F)\in \mathscr{C}$ that satisfy the nodal properties \eqref{Nodal}. From Theorem~\ref{thm:global}, the curve $\mathscr{C}$ is continuous, therefore connected as a subset of $X \times \mathbb{R}$. Further, via Lemmas~\ref{Open property} and \ref{Closed property}, we know that $\mathcal{S}$ is a relatively open and closed subset of $\mathscr{C}$. Also, the argument in the previous paragraph guarantees that  $\mathscr{C}_{\textup{loc}} \subset \mathcal{S}$, hence $\mathcal{S}\neq \varnothing$. Thus, we conclude that $\mathcal{S}=\mathscr{C}$ which completes the proof of the lemma.
\end{proof}
\begin{lemma}\label{inf+F} For every $K > 0$, there exists a constant $C = C(K)>0$ such that every
 $\left(w,F\right) \in \mathscr{U} \cap \mathcal{F}^{-1}(0)$ with $\norm{ w}_{C^{1}(R)} \leq K$ obeys the bound 
\[ \dfrac{1}{\inf_R \left(w_p+H_p\right)}+F < C.\]
\end{lemma}

\begin{proof}
The bounds are a direct consequence of  Theorem~\ref{thm:upperboundonF} and Corollary~\ref{boundonw}.
\end{proof}

Next, we will rule out the loss of compactness alternative in Theorem~\ref{thm:global}\ref{global alternatives}.  A key component of the argument is the nonexistence of (nontrivial) monotone front-type solutions of the height equation.  This fact is proved in \cite[Corollary 4.12]{chen2018existence} and recalled below.
\begin{theorem}[Nonexistence of monotone fronts] \label{non existence of monotone fronts}
Suppose that $h \in C^2_{\textup{b}}(\overline{R^+}) \cap  C^2_{\textup{b}}(\overline{R^-})\cap C^0(\overline{R})$ is a front solution to the height equation \eqref{quasi} in the sense that 
\[h(q,p) \to H_{\pm}(p) \quad \textrm{as } q\to \pm \infty.\]
 If $\inf_{R} h_p >0$ and 
 \[ H_+ \geq H_- = H \textrm{ on } [-1,0] \quad \textrm{or} \quad H_+ \leq H_- = H \textrm{ on } [-1,0], \]
 then $H_+ = H_- = H$.
 \end{theorem}

\begin{lemma}[Local compactness] \label{compactness lemma}
Suppose that $\{(w_n,F_n)\} \subset \mathscr{U}$ is a sequence of monotone solutions to the height equation \eqref{quasiw} that is uniformly bounded in $X \times \mathbb{R}$.  Then we can extract a subsequence converging in $X \times \mathbb{R}$ to some $(w,F) \in X \times [\Fcr,\infty)$.  
\end{lemma}
\begin{proof}  Let a sequence $\{(w_n,F_n)\}$ be given as above.  By Lemma~\ref{inf+F} and boundedness, it follows that these solutions lie in a subset of $X \times \mathbb{R}$ on which the height equation \eqref{quasiw} is uniformly elliptic with a uniformly oblique boundary condition on the top and a co-normal transmission condition on the interior.  By a straightforward adaptation of \cite[Lemma 6.3]{chen2018existence} to transmission problems, we may then conclude that either
\begin{enumerate}[label=(\roman*)]
\item $\{ (w_n, F_n) \}$ is pre-compact in $X\times\mathbb{R}$; or 
\item we can extract a subsequence and find $q_n \to \infty$ so that the translated sequence $\{ \tilde w_n\}$ defined by $\tilde w_n := w_n(\cdot+q_n, \cdot)$ converges in $C^9_{\mathrm{loc}}$ to some $\tilde w \in X_{\mathrm{b}}$ which solves \eqref{quasiw} and has $\tilde w \not\equiv 0$ and $\partial_q \tilde w \leq 0$.
\end{enumerate}
However, the second of these alternatives is impossible in light of Theorem~\ref{non existence of monotone fronts}.  To see this, observe that were it to occur, then $\tilde h := \tilde w+H$ would be a front-type solution of the height equation \eqref{quasi} in the sense that 
\[ \tilde h \to \tilde H_- \textrm{ as }q \to -\infty \qquad \tilde h \to H \textrm{ as } q \to \infty,\]
for some $q$-independent solution $\tilde H_-$ to \eqref{quasi}.  Since $\tilde w_q \leq 0$, it must be that $\tilde H_- \geq H$.  Theorem~\ref{non existence of monotone fronts} then ensures that $\tilde H_- = H$, meaning $\tilde h \equiv H$ or equivalently $\tilde w \equiv 0$, a contradiction. The proof is therefore complete.
\end{proof} 
 
The next lemma applies the above result to conclude that the extreme of $\mathscr{C}$ does not limit to a critical flow.

\begin{lemma}[Asymptotic supercriticality]\label{lemma:asymptoticsupercriticality} 
If $\norm{ w(s) }_{X}$ is bounded uniformly along the bifurcation curve $\mathscr{C},$ then
\[\liminf\limits_{s\rightarrow \infty}F(s) > \Fcr.\]
\end{lemma}

\begin{proof}
We follow closely the argument in \cite[Lemma 6.9]{chen2018existence}.
By way of contradiction, suppose that there exists a sequence $s_n \rightarrow{\infty}$ such that 
\[\limsup_{n\rightarrow\infty}\norm{ w(s_n) }_{X} < \infty \qquad \textrm{and} \qquad \lim_{n \to \infty} F(s_n) = \Fcr.\] 
We have already proved that each $w(s_n)$ is a monotone and even solution to the quasilinear elliptic PDE \eqref{quasiw}.  Lemma~\ref{compactness lemma} therefore tells us that the sequence is  pre-compact and so passing to a subsequence, we may assume that $\{ (w(s_n),F(s_n)\}$ converges in $X \times \mathbb{R}$ to some  $(w^{\star}, F^{\star})$ with  $\mathcal{F}(w^{\star}, F^{\star})=0$ where $F^{\star}=F_{\textup{cr}}.$ However, by Theorem~\ref{thm:trivialsol}\ref{nonontrivialsolution}, this implies that $w^{\star} \equiv 0$ which is equivalent to saying ${\norm{w(s_n) }_{X} \rightarrow{0}}$ as $n \rightarrow{\infty}$. Lemma~\ref{Nodal on Global Curve} ensures that that each $w(s_n)$ is a wave of elevation, therefore by the uniqueness of small-amplitude solutions, $(w(s_n),F(s_n)) \in \mathscr{C}_{\textrm{loc}}$ for  $n\gg 1$. But, this contradicts the statement in Theorem~\ref{thm:global}\ref{solutions not in local curve} that the curve does not reconnect to the trivial solution.  The proof of the lemma is therefore complete.
\end{proof}
Finally, we are ready to complete the proof of Theorem~\ref{thm:main}. It only remains to assemble all the information obtained earlier. 
\begin{proof}[Proof of Theorem~\ref{thm:main}]
Let $\mathscr{C}$ be the global curve given by Theorem~\ref{thm:global}. The statement in part~\ref{critical laminar} follows by construction of the local curve $\mathscr{C}_\loc$, specifically  Theorem~\ref{smallamplitudethoery}\ref{continuity}. To prove part~\ref{symmetry and monotonicity}, recall from Theorem~\ref{thm:global} that $\mathscr{C} \subset \mathscr{U}$ which is defined in \eqref{solutions set}. This shows that all the solutions contained in $\mathscr{C}$ are symmetric and supercritical. Moreover, they are monotonic as a consequence of the nodal properties established in Lemma~\ref{Nodal on Global Curve}.

Finally, we consider the stagnation limit claimed in part~\ref{stagnation}. It was already shown in Lemma~\ref{compactness lemma} that the loss of compactness alternative in Theorem~\ref{thm:global}\ref{global alternatives}\ref{loss of compactness} does not occur. Thus, the blowup alternative \ref{blowup} must happen:
\[
 N(s):= \norm{ w(s) }_X +\dfrac{1}{\inf_R \left(w_p(s)+H_p\right)}+F(s)+\dfrac{1}{F(s)-F_{\textup{cr}}} \rightarrow{\infty}, \quad \textrm{as }s \rightarrow{\infty}
.\] 
From the bounds in Lemmas~\ref{inf+F} and \ref{lemma:asymptoticsupercriticality}, this can be further refined to
\[
\norm{ w(s) }_X \rightarrow{\infty},\qquad \textrm{as } s\to \infty.
\]
By definition of $X$ in \eqref{BanachSpaces} and Theorem~\ref{thm:uniformregularity}, the above limit  simplifies to  $\norm{ w_p(s) }_{C^{0}}\rightarrow{\infty}$.

We now translate this back to the physical variables.  In Eulerian dimensionless form, it reads
\begin{equation}\label{infeulerian}
    \inf_{\tilde{\Omega}(s)}(\tilde{c}-\tilde{u}(s))=\dfrac{1}{\sup_{R}|\sqrt{\rho}\partial_p h(s)|} \rightarrow{0} \quad \textrm{as } s\to \infty.
\end{equation} 
Recall, that the dimensional and dimensionless Eulerian horizontal velocities are related by
\begin{equation}\label{dimensionalless}
    u-c=\dfrac{m}{\sqrt{\rho|_\eta}d}(\tilde{u}-\tilde{c})=F \sqrt{gd}(\tilde{u}-\tilde{c}).
\end{equation}
Combining \eqref{infeulerian} with the bounds on the Froude number given in Theorem~\ref{thm:upperboundonF}, we obtain
\[
F(s)^2\leq \dfrac{C}{\inf_{\tilde{\Omega}(s)}(\tilde{c}-\tilde{u}(s))}.
\] Taking the of infimum both sides of the equation \eqref{dimensionalless} and combining with the inequality above results in the following
\[
 \inf_{\Omega(s)}(c-u(s))=F(s) \sqrt{gd}\inf_{\tilde{\Omega}(s)}(\tilde{c}-\tilde{u}(s))\leq C \sqrt{\inf_{\tilde{\Omega}(s)}(\tilde{c}-\tilde{u}(s))}\rightarrow{0}\] as $s\rightarrow{\infty}.$ Thus a point of horizontal stagnation develops in the limit.
 \end{proof}

\section*{Acknowledgments}

The author's work on this project was partially supported by the National Science Foundation through the award NSF DMS-1812436.

\appendix

\section{Quoted results} \label{quoted results appendix}

To keep the presentation reasonably self-contained, this appendix collects two important results from the literature that are used in the present work. We begin with a theorem that contains the maximum principle, Hopf boundary lemma and Serrin edge point lemma. Notably, this includes versions that allow for the ``bad sign'' of the zeroth order term in the operator provided the sign of the solution is known (see, for instance, \cite{Fraenkel2000Maximumprinciples}, \cite{Nirenberg1979Symmetry} and \cite{Serrin1971symmetry}).
\begin{theorem}\label{Maximum Principle}
Let $\Omega \subset \mathbb{R}^2$ be a connected, open set (possibly unbounded), consider the second-order operator
\begin{equation*}
    L:=\sum_{i,j=1}^n a_{ij}(x)\partial_i\partial_j+\sum_{i}^{n}b_i(x)\partial_i+c(x)
\end{equation*}
where $\partial_i$ denotes the spatial derivative in $x_i$ coordinate and the coefficients $a_{ij}, b_i, c$ are of class $C^0(\overline{\Omega})$. We also assume that $L$ is uniformly elliptic; that is there exists $\lambda>0$ with
\begin{equation*}
    \sum_{ij} a_{ij}(x) \xi_i \xi_j \geq \lambda |\xi|^2 \qquad \textrm{ for all } \xi \in \mathbb{R}^n, \quad x\in \overline{\Omega}, 
\end{equation*}
and $a_{ij}$ being symmetric. Let $u\in C^2(\Omega) \cap C^0(\overline{\Omega})$ be a classical solution of $Lu=0$ in $\Omega$.
\begin{enumerate}[label=\normalfont{(\alph*)}]
    \item \textup{(Strong maximum principle)} Suppose u attains its maximum value on $\overline{\Omega}$ at a point in the interior of $\Omega$. If $c \leq 0$ in $\Omega$, or if $\sup_{\Omega}u=0$, then u is a constant function.
     \item\label{HopfBL} \textup{(Hopf boundary lemma)} Suppose that $u$ attains a maximum on $\overline{\Omega}$ at a point $x^*\in \partial \Omega$ for which there exists an open ball $B\subset \Omega$ such that $\overline{B}\cap\partial\Omega=\{x^*\}$. Assume eithere $C\leq0$ in $\Omega$ or else $\sup_Bu=0$. Then $u$ is a constant function or
     \[ \nu \cdot \nabla u(x^*)>0,
     \]where $\nu$ is the outward unit normal to $\Omega$ at $x^*$.
     \item\label{Serrin} \textup{(Serrin edge point lemma)} Let $x^* \in \partial \Omega$ be an ``Edge point" in the sense that near $x^*$ consists of two trasnversally intersecting $C^2$`hypersurfaces $\{\gamma(x)=0\}$ and $\{\sigma=0\}$. Suppose that $\gamma, \sigma <0$ in $\Omega$. If $u \in C^2(\overline{\Omega})$, $u>0$, $u(x^*)=0$. Assume further that $a_{ij}\in C^2$ around the neighborhood of $x^*$,
     \[
     B(x^*)=0, \qquad \textrm{and} \qquad \partial_\tau B(x^*)=0 
     \] for every differential $\partial_{\tau}$ tangential to $\{\gamma=0\} \cap \{\sigma=0\}$ at $x^*$. Then for any unit vector $s$ outward from $\Omega$ at $x^*$, either
     \[
     \partial_s u(x^*)<0 \textrm{ or } \partial^2_s u(x^*)<0.
     \]
\end{enumerate}
\end{theorem}

Secondly, we record here the abstract analytic global bifurcation result from \cite[Theorem 6.1]{chen2018existence}.

\begin{theorem}[Chen, Walsh, Wheeler \cite{chen2018existence}] \label{cww global}
Let $\mathcal{X}$ and $\mathcal{Y}$ be Banach spaces, with $\mathcal{U} \subset \mathcal{X} \times \mathbb{R}$ an open set.  Suppose that $\mathcal{F} = \mathcal{F}(x,\lambda) : \mathcal{U} \to \mathcal{Y}$ is real analytic.  

Assume that there exists a continuous local curve $\mathscr{C}_{\textup{loc}}$ of solutions to $\mathcal{F}(x,\lambda)=0$ parametrized as 
\begin{equation*}
        \mathscr{C}_{\textup{loc}}:=\{(\tilde{x}(\lambda),\lambda):0 < \lambda < \lambda_{*} \},
\end{equation*}
where $\lambda_*>0$ and the map $\tilde{x}:(0,\lambda_*)\to \mathcal{U}$ is continuous.  If
\[\tilde{x}(\lambda)\to 0 \in \partial \mathcal{U} \textrm{ as } \lambda \to 0^+
\textrm{ and } \mathcal{F}_x(\tilde{x}(\lambda),\lambda):\mathcal{X}\to\mathcal{Y} \textrm{ is
invertible for all }\lambda,\] then 
the local curve $\mathscr{C}_\loc$ of the nonlinear operator $\mathcal{F}(x,\lambda)=0$ is contained in a global $C^0$ curve $\mathscr{C}$ parameterized as 
\begin{equation*}
        \mathscr{C}:=\{(x(s),s):0 < s < \infty \}\subset \mathcal{F}^{-1}(0),
\end{equation*} for some continuous $(0,\infty)\ni s \mapsto \left(x(s),s\right)\in \mathcal{U}\times \mathcal{I}$ and exhibiting the following properties
\begin{enumerate}[label=\normalfont{(\alph*)}]
    \item One of the following alternatives must hold:
    \begin{enumerate}[label=\normalfont{(\roman*)}]
        \item \textup{(Blowup)} as $s\rightarrow \infty,$
        \begin{equation*}
            N(s):= \norm{ x(s) }_{\mathcal{X}} +\dfrac{1}{\textup{dist}(x(s),\partial \mathcal{U})}+\lambda(s)+\dfrac{1}{\textup{dist}(x(s),\partial \mathcal{I})}\rightarrow{\infty}.
        \end{equation*}
        \item \textup{(Loss of compactness)} There exists a sequence of $s_n\rightarrow{\infty}$ such that  $\sup_n N(s_n)<\infty$ but $\{x(s_n)\}$ does not have subsequences converging in $\mathcal{X}$.
    \end{enumerate}
    \item Fix parameter $s^\star \in (0,\infty)$, around the neighborhood $(x(s^\star),\lambda(s^\star))\in \mathscr{C}$, we can reparametrize $\mathscr{C}$ so that $s\mapsto(x(s),\lambda(s))$ is real-analytic. 
    \item For all $s \gg 1$, $(x(s),\lambda(s)) \notin \mathscr{C}_\loc$.
\end{enumerate}
\end{theorem}
\begin{center}
\bibliographystyle{alpha}
\bibliography{Biblio.bib}
\end{center}
\end{document}